\documentclass[a4paper,oneside,11pt]{article}

\usepackage{amsmath,amsfonts,amscd,amssymb}
\usepackage{longtable,geometry}
\usepackage[english]{babel}
\usepackage[utf8]{inputenc}
\usepackage[active]{srcltx}
\usepackage[T1]{fontenc}
\usepackage{graphicx}
\usepackage{enumitem}
\usepackage{bbm}
\usepackage{enumerate}   
\usepackage{mathtools}
\usepackage{hyperref}\hypersetup{colorlinks=false, pdfborder={0 0 0},}

\usepackage{MnSymbol}
\usepackage{stmaryrd}
\usepackage{nicefrac}

 \usepackage{rotating}

\usepackage{xcolor}
\usepackage{framed}

\colorlet{shadecolor}{blue!15}

\newtheorem{theorem}{Theorem}[section]

\newtheorem{corollary}[theorem]{Corollary}
\newtheorem{lemma}[theorem]{Lemma}
\newtheorem{proposition}[theorem]{Proposition}

\newtheorem{remark}[theorem]{Remark}

\newcommand{\be}[1]{\begin{equation}\label{#1}}
\newcommand{\ee}{\end{equation}}
\numberwithin{equation}{section}

\newcommand{\ba}[1]{\begin{align}\label{#1}}
\newcommand{\ea}{\end{align}}
\numberwithin{equation}{section}

\newcommand{\ben}{\begin{equation*}}
\newcommand{\een}{\end{equation*}}
\numberwithin{equation}{section}

\newenvironment{proof}[1][\relax]
  {\paragraph{Proof\ifx#1\relax\else~of #1\fi}}%
  {~\hfill$\square$\par\bigskip}

\newcommand{\calA}{\mathcal{A}}
\newcommand{\calB}{\mathcal{B}}
\newcommand{\calC}{\mathcal{C}}
\newcommand{\calD}{\mathcal{D}}
\newcommand{\calE}{\mathcal{E}}
\newcommand{\calF}{\mathcal{F}}

\newcommand{\calH}{\mathcal{H}}

\newcommand{\calO}{\mathcal{O}}

\newcommand{\calQ}{\mathcal{Q}}
\newcommand{\calR}{\mathcal{R}}
\newcommand{\calS}{\mathcal{S}}
\newcommand{\calT}{\mathcal{T}}

\newcommand{\calV}{\mathcal{V}}

\newcommand{\calX}{\mathcal{X}}

\newcommand{\bbE}{\mathbb{E}}

\newcommand{\bbG}{\mathbb{G}}

\newcommand{\bbN}{\mathbb{N}}
\newcommand{\bbO}{\mathbb{O}}
\newcommand{\bbP}{\mathbb{P}}
\newcommand{\bbQ}{\mathbb{Q}}
\newcommand{\bbR}{\mathbb{R}}

\newcommand{\bbT}{\mathbb{T}}

\newcommand{\bbZ}{\mathbb{Z}}

\newcommand{\sfC}{{\sf C}}
\newcommand{\sfP}{{\sf P}}

\newcommand{\rk}[1]{\bgroup\color{red}%
  \par\medskip\hrule\smallskip%
  \noindent\textbf{#1}%
  \par\smallskip\hrule\medskip\egroup}

\definecolor{speccol}{rgb}{0.7,0.1,0.5}

\definecolor{speccol2}{rgb}{0.237,0.145,0.033}
\newcommand{\alexcom}[1]{\textcolor{orange}{#1}}

\renewcommand{\alexcom}[1]{}

\newcommand{\Iceconf}{\Omega_{6V}}
\newcommand{\heightfcns}{\mathcal{H}}
 
\newcommand{\PRcyl}[3]{\mathbb{P}^{(#1)}_{#2, #3}} 
\newcommand{\Zcyl}[3]{Z^{(#1)}_{#2, #3}} 
\newcommand{\len}{\mathrm{length}}

\newcommand{\Var}{\mathrm{Var}}

\newcommand{\xlra}{\xleftrightarrow}

\newcommand{\La}{\Lambda}

\newcommand{\ind}{\mathbbm 1}

\newcommand{\Slice}{{\rm Slice}}

\title{Delocalization of the height function of the six-vertex model}
\author{Hugo Duminil-Copin\thanks{Institut des Hautes \'Etudes Scientifiques and Universit\'e de Gen\`eve. \textit{Mail:} IH\'{E}S, 35 route de Chartres,
91440 Bures-sur-Yvette, France. \textit{E-mail:} \texttt{duminil@ihes.fr}}, 
Alex Karrila\thanks{\AA bo Akademi University. \textit{Mail:} Henriksgatan 2, 20500 \AA bo, Finland. \textit{E-mail:} \texttt{alex.karrila@abo.fi}},
Ioan Manolescu\thanks{D\'{e}partement de math\'{e}matiques, 
Universit\'{e} de Fribourg. \textit{Mail:}
Chemin du Mus\'{e}e 23,
CH-1700 Fribourg, Switzerland. \textit{E-mail:} \texttt{ioan.manolescu@unifr.ch}},
Mendes Oulamara\thanks{Institut des Hautes \'{E}tudes Scientifiques and Université Paris-Saclay. \textit{Mail:} IH\'{E}S, 35 route de Chartres,
91440 Bures-sur-Yvette, France. \textit{E-mail:} \texttt{oulamara@ihes.fr}}
\date{\today}
}

\begin{document}

\maketitle

\begin{abstract}
We show that the height function of the six-vertex model, in the parameter range $\mathbf a=\mathbf b=1$ and $\mathbf c\ge1$, is delocalized with logarithmic variance when $\mathbf c\le 2$. This complements the earlier proven localization for $\mathbf c>2$. Our proof relies on Russo--Seymour--Welsh type arguments, and on the local behaviour of the free energy of the cylindrical six-vertex model, as a function of the unbalance between the number of up and down arrows. 
\\ \\
\textbf{MSC classes:} 60K35, 82B20, 82B27 
\end{abstract}

\tableofcontents

\section{Introduction}

\subsection{Motivation}\label{subsubsec: intro - motivation}

The six-vertex model was initially proposed by Pauling in 1935 in order to study the thermodynamic properties of ice~\cite{Pau35}. It became the archetypical example of a planar integrable model after Lieb's solution of the model in 1967 in its anti-ferroelectric and ferroelectric phases~\cite{Lie67a,Lie67b,Lie67c} using the Bethe ansatz (see~\cite{BetheAnsatz1} and references therein for an introduction). In the last fifty years, further analysis of the model has provided deep insight into the subtle structure of two-dimensional integrable systems, for instance with the development of the Yang-Baxter equation, quantum groups, and transfer matrices; see e.g.~\cite{Bax89,Resh10}. 

The six-vertex model lies at the crossroads of a vast family of two-dimensional lattice models. Among others, it has been related to the dimer model
, the Ising and Potts models, the critical random-cluster model, the loop $O(n)$ models, the Ashkin-Teller models, random permutations, stochastic growth model
, and quantum spin chains, to cite but a few examples (see references below). In recent years, the interplay between all these models has been used to prove a number of new results on the behaviour of each one of them. Let us mention here the extensive study of the free fermion point in relation to dimers~\cite{BK,FerSpo06,Ken00}; the analysis of critical points of random-cluster models and loop $O(n)$ models~\cite{Nie82,RS19}; the development of parafermionic observables based on the six-vertex model, culminating with the proof of conformal invariance of the Ising model~\cite{CheSmi12,Smi10}\footnote{
See~\cite{IkhWesWhe13,IkhCar09} for examples of constructions,~\cite{Dum17} for a review, and~\cite{BefDumSmi15,DumSmi12} for other examples of simple mathematical applications.
}; the understanding of dimerization properties of the anti-ferromagnetic Heisenberg chain~\cite{AizDum20}; and the relation between Kardar--Parisi--Zhang equation and the stochastic six-vertex model~\cite{BorCorGor16}. 

While the use of the six-vertex model's integrable properties has been extraordinarily fruitful to understand its free energy, the analysis of the model's correlation functions and the associated stochastic processes have been particularly limited (with some notable exceptions like the free fermion point). For instance, the exact integrability provides strong evidence of a Berezinskii--Kosterlitz--Thouless phase transition of the antiferroelectric model between a regime in which correlations decay polynomially fast and a regime where they decay exponentially fast. However, proving mathematically that this is indeed the case remains an open problem with today's techniques.

Another example of a property of the six-vertex model that seemed to elude mathematicians for many years is the rigorous understanding of its height function representation (see definition below). Indeed, the six-vertex model produces one of the most natural models of random height functions. This interpretation of the model plays an important role for at least two reasons. First, special cases include the height function of the dimer model (when considering the free fermion point) and the uniformly chosen graph homomorphisms from $\bbZ^2 $ to $\bbZ$ (when considering the original square-ice model), which are models of independent interest. Second, the height function interpretation has been at the center of the bozonization of 2D lattice models, an extremely powerful tool introduced in the physics literature and enabling the use of the Coulomb gas formalism to understand (as of today non-rigorously) the behaviour of correlations (see e.g.~\cite{Dub11,Nie84,ZubItz77}). 

One of the most fundamental questions one can ask about a model of a random height function $h$ is whether the height function fluctuates or not. More precisely, does the height variance ${\rm Var}[h(x)-h(y)]$ between two points $x$ and $y$ remain bounded uniformly in $x$ and $y$, or does it on the contrary grow to infinity as the distance between $x$ and $y$ goes to infinity? In the former scenario, we say that the height function model is {\em localized} or  {\em smooth}, and in the latter one, that it is {\em delocalized} or  {\em rough}. On which side (localized/delocalized) of the dichotomy the model lies is a crucial question which can be understood as an analogue, for spin or percolation systems, of determining whether long-range order occurs or not at criticality. The answer can be quite subtle and seemingly similar models can exhibit very different behaviours.

As mentioned above, most of the currently known exact results on the six-vertex model seem to provide little rigorous information on the behaviour of the height function, in particular they do not directly answer the question of localization/delocalization. In this paper, we provide the first full description of which parameters $\mathbf c$ are such that the six-vertex height function is localized/delocalized, in the regime corresponding to Rys' model of hydrogen bonded ferroelectrics~\cite{Rys63} where the parameters of the six-vertex model, as defined in the next subsection, are $\mathbf a=\mathbf b=1$ and $\mathbf c\ge1$~\footnote{
Various predictions for the six-vertex model are formulated in terms of the parameter $\Delta=(\mathbf a^2+\mathbf b^2-\mathbf c^2)/(2\mathbf a\mathbf b)$. The six-vertex models with $\mathbf a=\mathbf b=1$ and $\mathbf c\ge1$ are equivalently determined by $\mathbf{a} = \mathbf{b}$ and $\Delta \le 1/2$; as we shall soon see, in terms of the latter formulation, we have localization for $\Delta < -1$ and delocalization for $\Delta \in [-1, 1/2]$.
}.

\subsection{Definitions and main result on the torus}

The six-vertex model on the torus is defined as follows. For $N>0$ even, let $\bbT_N := (V(\bbT_N), E(\bbT_N) )$ be the toroidal square grid graph with $N \times N$ vertices. An \emph{arrow configuration} $\omega$ on $\bbT_N$ is the choice of an orientation for every edge of $E(\bbT_N)$. We say that $\omega$ satisfies the \emph{ice rule}, or equivalently that it is a {\em six-vertex configuration}, if every vertex of $V(\bbT_N)$ has two incoming  and two outgoing incident edges in $\omega$. These edges can be arranged in six different ways around each vertex as depicted in Figure~\ref{fig:the_six_vertices}, hence the name of the model.
For parameters $\mathbf a_1,\mathbf a_2,\mathbf b_1,\mathbf b_2,\mathbf c_1,\mathbf c_2\ge0$, define the {\em weight} of a configuration $\omega$ to be 
\[
W_\mathrm{6V}(\omega) =\mathbf a_1^{n_1}\mathbf a_2^{n_2}\mathbf b_1^{n_3}\mathbf b_2^{n_4}\mathbf c_1^{n_5}\mathbf c_2^{n_6},
\]
where $n_i$ is the number of vertices of $V(\bbT_N)$ having type $i$ in $\omega$. 
In this paper, we will not study the model in its full generality of parameters, and focus on the special choice given by $\mathbf a_1=\mathbf a_2=\mathbf b_1=\mathbf b_2=1$ and $\mathbf c_1=\mathbf c_2=\mathbf c\ge1$, which corresponds to {\em isotropic\footnote{
The reader may verify from Figure~\ref{fig:the_six_vertices} that given $\mathbf a_1=\mathbf a_2=\mathbf b_1=\mathbf b_2$ and $\mathbf c_1=\mathbf c_2$, the weight of a vertex does not change under symmetries of the square lattice.
}} parameters. 

\begin{figure}[htb]
	\begin{center}
		\includegraphics[width=0.75\textwidth, page=1]{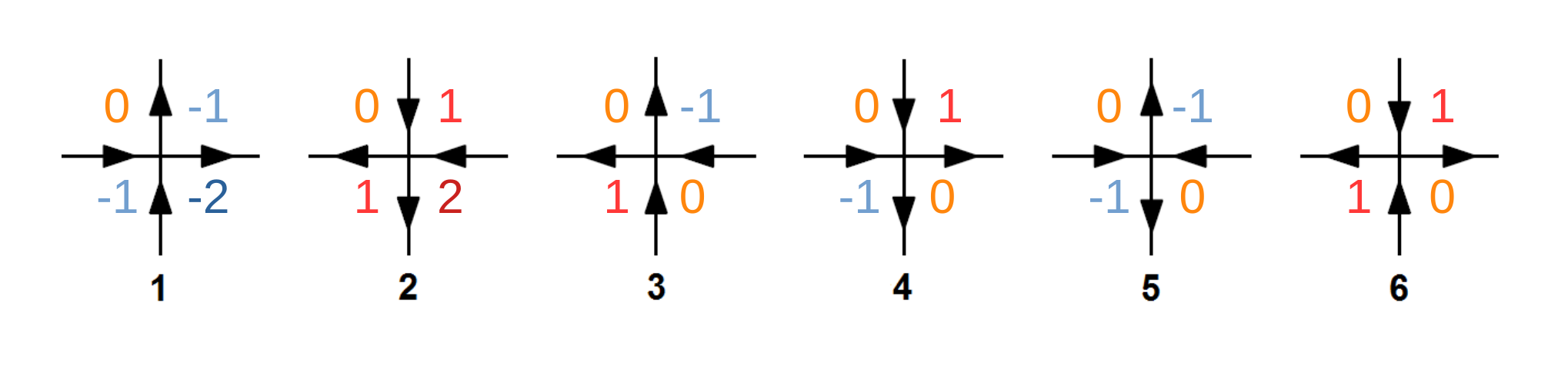}
	\caption{The $6$ \emph{types} of vertices in the six-vertex model together with the corresponding height function values on the four squares touching this vertex (we set $h=0$ on the upper-left square).
	Each type comes with a weight $\mathbf a_1,\mathbf a_2,\mathbf b_1,\mathbf b_2,\mathbf c_1,\mathbf c_2$. 
	}
	\label{fig:the_six_vertices}	
	\end{center}
\end{figure} 

The weights induce a probability measure on the set $\Iceconf(\bbT_N)$ of six-vertex configurations $\omega$ on $\bbT_N$ given by
\[
\bbP_{\bbT_N}[\{\omega\}] = \frac{W_\mathrm{6V}(\omega)}{Z(\bbT_N)},
\]
where $Z(\bbT_N):=\sum_{\omega\in\Iceconf(\bbT_N)} W_\mathrm{6V} (\omega)$ is the partition function. Below, we write $\bbE_{\bbT_N}$ for the associated expectation.

Write $\bbT_N^*$ for the dual graph of $\bbT_N$: its vertices are the faces of $\bbT_N$
and two vertices of $\bbT_N^*$ are connected by an edge of $\bbT_N^*$ if the corresponding faces of $\bbT_N$ share an edge. 
As mentioned in Section~\ref{subsubsec: intro - motivation}, the six-vertex model and its ice-rule naturally emerge when studying  {\em graph homomorphisms} from $\bbT_N^*$ into $\bbZ$, i.e.,~maps $h$ from the faces $F(\bbT_{N})$ of $\bbT_N$ to $\bbZ$ which satisfy $\vert h(x) -  h(y) \vert = 1$ for all neighbouring faces $x, y \in F(\bbT_{N})$. We call such graph homomorphisms {\em height functions}. To avoid certain technical difficulties, we will assume that $N$ is even and partition the faces of $\bbT_N$ in a bipartite fashion into odd and even faces, and will hereafter additionally impose that a height function $h$ is odd on odd faces and even on even faces.

For a given height function $h$, introduce the {\em six-vertex configuration $\omega$ associated to $h$} by orienting each edge $e$ so
that out of the two faces bordering $e$, the one on the left of $e$ (in the sense of this orientation) has the larger value of $h$. Note that two height functions $h$ and $h'$ give rise to the same six-vertex configuration $\omega$ if and only if $h -h'$ is a constant function. In the converse direction, it is not always true that a six-vertex configuration gives rise to a graph homomorphism on the faces of the torus (it only defines the gradient and may lead to inconsistencies when wrapping around the torus). However, for {\em balanced} configurations $\omega$ such that any row (resp.~column) of faces 
around $\bbT_N$ is crossed by as many up (resp.~right) as down (resp.~left) arrows, there exists a height function $h$ associated with $\omega$, which is unique up to additive constant.
From now on, let $\Iceconf^{(\mathrm{bal})}(\bbT_N)$ be the set of balanced configurations and 
\[\bbP_{\bbT_N}^{\mathrm{(\mathrm{bal})}} :=\mathbb P_{\bbT_N}[\,\cdot\,|\,\omega\in \Iceconf^{(\mathrm{bal})}(\bbT_N)].\]

When $\omega$ is chosen according to $\bbP_{\bbT_N}^{\mathrm{(\mathrm{bal})}}$, write $h$ for the height function associated to it (to fix the additive constant choose a root face and impose that $h$ is null on that face). 
Observe that the increments $h(x)-h(y)$ do not depend on the choice of the additive constant and thus on the root face. 
Also note that $\bbE_{\bbT_N}^{\mathrm{(\mathrm{bal})}} [ h(x)-h(y) ] = 0$ by symmetry under a global arrow flip and thus we have
\begin{align*}
\bbE_{\bbT_N}^{\mathrm{(\mathrm{bal})}} [ ( h(x)-h(y) )^2 ] 
= \mathrm{Var}_{\bbT_N}^{\mathrm{(\mathrm{bal})}} (h(x)-h(y)),
\end{align*}
where $\mathrm{Var}_{\bbT_N}^{\mathrm{(\mathrm{bal})}} $ denotes variance under $\bbP_{\bbT_N}^{\mathrm{(\mathrm{bal})}}$.
The goal of this paper is to study the behaviour of this variance as $x$ and $y$ are distant vertices in a large torus. 
Below, $d$ denotes the graph distance on the dual graph $\bbT_{N}^*$.

\begin{theorem}[Delocalized phase]\label{thm:variance_torus}
    Fix $1\le\mathbf c\le 2$. There exist $c,C > 0$ such that for every $N$ even 
    and every $x, y\in F(\bbT_{N})$ with $d(x, y) \geq 2$, we have
    \begin{align*}
    c \log d (x, y )
    \le    
    \bbE_{\bbT_{N}}^{\mathrm{(\mathrm{bal})}} [ ( h(x)-h(y) )^2 ] 
    \le C \log d (x, y ).
    \end{align*}
\end{theorem}

\begin{remark}
	Instead of an $N \times N$ square torus as here, one may more generally study the balanced six-vertex model on an $N \times M$ torus with even dimensions.
	For $M \geq N$ the variance of $h(x)-h(y)$ behaves like $\log d(x,y) + d(x,y)/N$ up to multiplicative constants, and the results of the present paper may be used to show this.  
\end{remark}

The previous result was known in three special cases, namely for square-ice, i.e.~$\mathbf c=1$~\cite{CPT18,Hom,She05}, for the free fermion point $\mathbf c=\sqrt2$~\cite{BK,FerSpo06,Ken00}, and for $\mathbf c=2$~\cite{GlaPel19}. During the writing of this paper, Marcin Lis produced a proof for $\sqrt{2+\sqrt 2} \leq \mathbf c \leq 2$ based on different techniques than ours~\cite{Lis20}. To the best of our knowledge, the result is new for other parameters $\mathbf c \ge 1$.
This result offers a complete picture of the behaviour of the height function of the six-vertex model in the range of parameters $\mathbf a=\mathbf b=1$ and $\mathbf c\ge1$ as it complements the following result for $\mathbf c>2$, proved in~\cite{GlaPel19} (as a consequence of~\cite{DGHMT,RS19}).

\begin{theorem}[Localized phase;~\cite{GlaPel19,DGHMT,RS19}]\label{thm:localisation}
	Fix $\mathbf c>2$. There exists $C\in(0,\infty)$ such that for every $N$ even and every $x,y\in F(\bbT_{N})$,
	$$ \bbE_{\bbT_{N}}^{\mathrm{(\mathrm{bal})}}[(h(x)-h(y))^2]\le C.$$
\end{theorem}

In the regime $1\le \mathbf c\le 2$, the model is predicted to have Gaussian behaviour and to converge in the sense of distributions in the scaling limit to (a scaling of) the Gaussian Free Field (GFF) on the two-dimensional torus.
The only case for which this is known is the free fermion point $\mathbf c=\sqrt 2$ (see~\cite{Ken00} and reference therein; cf. also~\cite{GMT17}). 
The logarithmic divergence of the variance proved here is consistent with this behaviour, but not sufficient to determine it. 
We do however mention~\cite{DKKMO20}, whose result may be loosely reformulated as ``any sub-sequential scaling limit of height functions obtained for  $\sqrt 3\le\mathbf c\le 2$ is invariant under rotation''. 

%

\subsection{Main results in planar domains}\label{subsec: main results in planar domains}

In this section, we develop the theory of the six-vertex model in finite domains and present our main results in this context.  Due to the six-vertex model's spatial Markov property (see Section~\ref{sec:SMP} for the precise definition), these results have strong implications for the six-vertex model on the torus discussed in the previous section.

Equip $\bbZ^2$ with the square grid graph structure and let $(\bbZ^2)^* = (\frac12,\frac12) + \bbZ^2$ denote the dual of $\bbZ^2$; its vertices are identified with the faces of $\bbZ^2$. As before, we partition $\mathbb Z^2$ in a bipartite fashion into odd and even faces.
Let $V$ be a finite connected set of vertices of the graph $\bbZ^2$, and let $E$ be the edges incident to them.
The height function model related to arrow configurations on $E$ with ice rule on $V$ is defined on the subgraph $D$ of $(\bbZ^2)^*$ consisting of the faces of $\bbZ^2$ with at least one corner in $V$ and the dual edges of $\bbZ^2$ that cross a primal edge in $E$. We say that such a subgraph $D$ is a \emph{discrete domain}, and denote $V=V(D)$ and $E=E(D)$. 
The faces of $\bbZ^2$ in $D$ with at least one corner not in $V$ are called the \emph{boundary $\partial D$ of $D$} (see Figure~\ref{fig:domains}).

A {\em boundary condition} on $D$ is given by a function $\xi:\partial D\rightarrow\bbZ$; we say that $\xi$ is {\em admissible} if there exists a graph homomorphism $h: D \to \bbZ$ with $h_{| \partial D} = \xi $ and if $\xi$ is odd on odd faces (and therefore even on even faces).
Let $\Iceconf(D,\xi)$ be the set of arrow configurations on $E$, associated with graph homomorphisms $h: D \to \bbZ$ with $h_{| \partial D} = \xi $. The map from these graph homomorphisms to the associated six-vertex configurations in $\Iceconf(D,\xi)$ is bijective, and hence we will often identify them, and call $h$ height functions.
Introduce the probability measure on $\Iceconf(D,\xi)$ given by 
\begin{align*}
	\bbP_D^{\xi}[\{\omega\}] = \frac{W_\mathrm{6V}(\omega)}{Z(D,\xi)},
\end{align*}
where $W_\mathrm{6V}$ is the six-vertex weight from the vertices of $V(D)$ and $Z(D,\xi )$ is the partition function defined by $\sum_{\omega\in\Iceconf(D,\xi)} W_{\rm 6V}(\omega)$ \footnote{
In later sections, we will use the notation $\bbP_D^{\xi}$ for height assignments $\xi$ on other supports than $\partial D$; the concepts above readily generalize to these cases.
}.

For integers $0<n<N$, define $\Lambda_n: =(-n,n)^2$ and $A(n, N) = \Lambda_N\setminus\Lambda_n$; denote $\Lambda_n \subset D$ if $\Lambda_n \cap \bbZ^2 \subset V(D)$ and similarly for $A(n,N) $.
Let $\calO_{h\ge k}(n,N)$ be the event that there exists a circuit of adjacent faces with $h(x)\ge k$ in $A(n,N)$ that surrounds $\Lambda_n$.

\begin{theorem}[Uniformly positive annulus circuit probabilities]\label{thm:bxp}
    Fix $\mathbf{c} \in [1, 2]$. For every $k,\ell>0$, there exists $c=c(\mathbf{c} ,\ell,k)>0$ such that for every $n$ large enough, 
    every discrete domain $D\supset \La_{2n}$, and every admissible boundary condition $\xi$ on $\partial D$ (or a subset of it) with $|\xi|\le \ell$, we have
    \begin{equation}\label{eq:bxp}
    	\bbP_D^{\xi}[\calO_{h\ge k}(n,2n)]\ge c.
    \end{equation}
\end{theorem}

An important aspect of the previous estimate is that it is uniform over the scales $n$ of the annulus in which the circuit occurs, as well as over the domains $D$. This allows one to combine it with the spatial Markov and FKG properties of the model (see Sections~\ref{sec:SMP}--\ref{sec:monotonicity}) to deduce the other main theorems of this paper. 
Note also that the ``flatness'' $|\xi|\le \ell$ of the boundary condition is crucial, and the statement above is expected to fail otherwise. An extreme example is given by ``sloped'' boundary conditions, that may be chosen in such a way as to completely freeze the configuration inside the domain (see Figure~\ref{fig:domains} for an example),
thus preventing the event $\calO_{h\ge k}(n,2n)$ from occurring. 

 
Theorem~\ref{thm:bxp} may be used to derive estimates similar to~\eqref{eq:bxp} for annuli with any outer to inner radius ratio, with the constant $c$ depending on this ratio.
This can be shown with standard RSW-type techniques, by building a big circuit out of many small ones. 
Two extensions of the result above will be discussed in Lemmas~\ref{lem:circ_increase} and~\ref{lem:circ_decrease}.
They are concerned with how the probability of events $\calO_{|h| \geq k}(n,N)$ evolves as 
$N/n$ tends to infinity, and as $k$ tends to infinity, respectively. 
The upshot is that the probability then converges to $1$ polynomially in $n/N$, and to $0$ exponentially in $k$, respectively. 

As a consequence of Theorem~\ref{thm:bxp}, we obtain the following bounds on the variance of the height function. 
Below, $d$ denotes the graph distance on $(\bbZ^2)^*$. 
\begin{corollary}[Logarithmic variance in planar domains]\label{cor:variance_D}
    Fix $1\le \mathbf c\le 2$. 
    There exist $c, C> 0$ such that for every discrete domain $D$, 
    every admissible boundary condition $\xi$ on $\partial D$, 
    and every face $x$ of $D \setminus \partial D$, 
    if we set $\max_{y \in \partial D } |\xi (y) | = \ell$, then 
    \begin{equation*}
    c\log d (x, \partial D)-4\ell^2 \le \mathrm{Var}_D^{\xi}  ( h ( x ) )  \le C\log d (x, \partial D) + 4\ell^2 .
    \end{equation*}
\end{corollary}

It is quite standard for percolation models that Theorem~\ref{thm:bxp} along with positive association and the spatial Markov property imply results such as Corollary~\ref{cor:variance_D} and Theorem~\ref{thm:variance_torus}.
However, we warn the reader of subtleties in their proofs
due to the particular forms of the spatial Markov property (Proposition~\ref{prop: SMP}) and the pushing of boundary conditions (Proposition~\ref{prop:pushing_1})
in this height-function model.

\subsection{Some core ideas of the proof}\label{sec:sketch}

As already mentioned, the key to all the results discussed so far is the circuit probability estimate of Theorem~\ref{thm:bxp}. Its proof relies on three main inputs.
First, Theorem~\ref{thm:Bethe} below estimates certain free energies associated to the six-vertex model on a cylinder;
it was obtained in~\cite{DKKMT20} using the Bethe ansatz\footnote{The central role of this input is highlighted by the fact that it is the only place in this paper that differentiates between the phases $\mathbf{c} > 2$ and $1 \leq \mathbf{c} \leq 2$.}.
The second is contained in the proof of Theorem~\ref{thm:RSW_origins}, and is way to relate the estimates of the free energy obtained above to a certain behaviour of the height function on the cylinder. This is the main innovation of the present work.
The third central input, also contained in the proof of Theorem~\ref{thm:RSW_origins}, is a Russo--Seymour--Welsh (RSW) type theory for the level sets of the height function. 
Below, we briefly introduce these three results in this order.

Let $\bbO_{N,M}$ denote the cylindrical square lattice with a height of $M$ faces and a perimeter of $N$ faces. The six-vertex configurations on (the $N\times(M-1)$ degree $4$ vertices of) $\bbO_{N,M}$ and their six-vertex weights are then defined as straightforward generalizations of the toroidal and finite planar cases. Let $N$ be even and, for $s \in [ -N/2, N/2]$, denote by $\Iceconf^{(s)}$ the set of six-vertex configurations on $\bbO_{N,M}$ such that every row of $N$ faces around $\bbO_{N,M}$ is crossed by $2 \lceil s\rceil$ more up arrows than down arrows. Let
\[
Z^{(s)}_{N,M}:=\sum_{\omega\in \Iceconf^{(s)}}W_\mathrm{6V}(\omega).
\]

\begin{theorem}[Free energy on the cylinder;~\cite{DKKMT20}]\label{thm:Bethe}
	Fix $\mathbf c>0$. There exists a function $f_\mathbf c:(-1/2,1/2)\to \bbR^+$ such that 
	\begin{align*}
	\lim_{\substack{N \to \infty\\ N\,\mathrm{even}}} \lim_{M \to \infty} \frac{1}{NM} \log 
	\Zcyl{ \alpha N }{N}{M}
	 = f_\mathbf c(\alpha).
	\end{align*}
	Moreover, 
	for $0<\mathbf c\le 2$ there exists $C= C(\mathbf{c}) > 0$ such that for every $\alpha \in (-1/2,1/2)$,
		\begin{align}\label{eq:BA}
			f_\mathbf c(\alpha) \geq f_\mathbf c(0)-C \alpha^2.
		\end{align}
\end{theorem}

The function $f_\mathbf c( \alpha )$ is called the \emph{free energy} of the cylindrical six-vertex model at unbalance $\alpha$.
The previous theorem has an appealing physical intuition: the free energy $f_\mathbf c(\alpha)$ is differentiable at 0 as a function of $\alpha$, for all $0<\mathbf c\le 2$. 

The objective of the second main ingredient, Theorem~\ref{thm:RSW_origins}, is to deduce the annulus circuit probabilities, and thus ultimately the delocalization of the height function, from the differentiability of $f_\mathbf c$. Let us mention that~\cite{DKKMT20} also shows that the free energy is non-differentiable at 0 when $\mathbf c>2$, 
which corresponds to the regime where the height function is localised (see Theorem~\ref{thm:localisation}).
As such, we have a direct correspondence between differentiability/non-differentiability of $f_\mathbf c(\alpha)$ at 0 and delocalization/localization of the height function with slope $0$; this correspondence is expected to apply in great generality, in particular for other slopes \cite{She05}.

\begin{theorem}[From free energy to annulus circuits]\label{thm:RSW_origins}\label{thm:crucial}
    There exist $\eta, c,C>0$ such that for all integers $k,r$ with $k$ large enough and $r > 2k / \eta $,
    and for all $\mathbf{c} \geq 1$, we have
    \begin{equation}\label{eq:RSW_origins}
	   \bbP_{\Lambda_{12 r}}^{0,1}[\calO_{h\ge ck}( 12 r , 6 r)]
    	\ge c \exp\big[C r^2\big(f_\mathbf c (\tfrac{k}{\eta r} )- f_\mathbf c(0)\big)\big],
    \end{equation}
    where $0,1$ denotes the admissible boundary condition on $\partial\Lambda_{12 r}$ taking values 0 and 1 only. 
\end{theorem}

Theorems~\ref{thm:Bethe} and~\ref{thm:RSW_origins} readily imply Theorem~\ref{thm:bxp}. 

Our third main step, the RSW theory, follows ideas that were created initially in the context of two-dimensional Bernoulli percolation~\cite{Rus78,SeyWel78}, and were instrumental for instance in the computation of its critical point. 
To date, RSW type results are understood as comparing crossing probabilities in domains of different shape but similar size scale.
In the past decade, the theory has been extended to a wide variety of percolation models~\cite{BefDum12,BolRio10,Tas15,DumHonNol11,DST,DGPS,KoSTas20}, and more recently to level sets of height function models on planar graphs, see e.g.~\cite{Hom,GM}.

In our main RSW type result, Theorem~\ref{thm:RSW}, the careful reader will observe a twist compared to the existing such statements on height function models: 
we bound the probability of having crossings of height larger than $c k$ of long domains by the probability of having crossings of height $k$ of short domains, 
where $c >0$ is a small constant. 
Such a loss in the height would prove very problematic for the renormalization arguments usually performed in percolation models. 
Indeed, Theorem~\ref{thm:RSW} does not a priori suffice to prove a dichotomy theorem as that of~\cite{DST} or~\cite{DT19}. 
In our setting, Theorem~\ref{thm:Bethe} provides an input which renders the renormalization superfluous.

\subsection{Further questions}
\label{subsec:intro - further questions}

\paragraph{Infinite volume limits and mixing}

A reader familiar with the random surface theory of~\cite{She05} will notice that the delocalization proven in this paper, together with an application of that theory\footnote{See~\cite{Piet} on the validity of~\cite{She05} on the six-vertex model.}, shows the local convergence of the balanced six-vertex arrow configurations, for $\mathbf{c} \in [ 1, 2]$, on the torus $\bbT_N$ as $N \to \infty$.  More delicate questions address the infinite-volume limit of the model in planar domains, and the rate at which the effect of different boundary conditions dies out, i.e., the mixing rate. Analogous infinite volume limits and mixing properties are fundamental, e.g., in the study of the Ising and FK Ising models.
For $\mathbf{c} \in [ \sqrt{3}, 2]$, such properties have been established also for the six-vertex model in~\cite{Lis20}.
We plan to discuss these topics in the full range $\mathbf{c} \in [ 1, 2]$ in a later publication.

\paragraph{Different model parameters}

The reader will notice that our main results, Theorems~\ref{thm:bxp} and its consequences, Theorem~\ref{thm:variance_torus} and Corollary~\ref{cor:variance_D}, are valid only for $\mathbf c \in [1,2]$. 
That $\mathbf c \leq 2$ is required is unsurprising since the model exhibits a different behaviour when $\mathbf c > 2$, 
as illustrated by Theorem~\ref{thm:localisation}.  The difference in behaviour may be traced back to the behaviour of the free energy of Theorem~\ref{thm:Bethe}; recall that this is the only point in our proof differentiating between $\mathbf c \in [1,2]$ and $\mathbf c > 2$.

When $\mathbf c \in (0,1)$, the Bethe ansatz computation of Theorem~\ref{thm:Bethe} still applies and provides a differentiable free energy at $\alpha = 0$. 
Moreover, the height model is expected to have a similar behaviour to when $\mathbf c \in [1,2]$. 
However, all the other main arguments of this paper fail in the range $\mathbf c \in (0,1)$, due to the lack of positive association which is ubiquitously applied in our proofs. Indeed, when $\mathbf c < 1$, the FKG property fails, both for the height function and its absolute value.

In a more general context, it is natural to consider the model with arbitrary positive weights 
$\mathbf a_1 = \mathbf a_2 = \mathbf a$, 
$\mathbf b_1 = \mathbf b_2 = \mathbf b$, and 
$\mathbf c_1 = \mathbf c_2 = \mathbf c$; recall that it is expected that the behaviour of the model depends only on $\Delta = (\mathbf a^2+\mathbf b^2-\mathbf c^2)/(2\mathbf a\mathbf b)$, and thus delocalization results similar to ours should hold for all parameters $(\mathbf a,\mathbf b,\mathbf c)$ with $\Delta \in [-1,1/2]$. As regards this case, we leave it to the reader to verify that our combinatorial tools of Section~\ref{sec: basic 6V height fcn ppties} and Appendix~\ref{appendix:FKG stuff}, in particular the positive association properties of the model, remain valid with analogous proofs whenever $\max \{ \mathbf{a}, \mathbf{b}\} \leq \mathbf{c}$. Consequently, it also holds that \textit{if} Theorem~\ref{thm:bxp} remains true for $\max \{ \mathbf{a}, \mathbf{b}\} \leq \mathbf{c}$, then so do Theorem~\ref{thm:variance_torus} and Corollary~\ref{cor:variance_D} (the proof of this implication is only based on the tools of Section~\ref{sec: basic 6V height fcn ppties}). Unfortunately, the geometric RSW theory (and more precisely, the proof of Proposition~\ref{prop:RSW1}) in this paper relies on the model being invariant under both vertical and diagonal reflections, hence requiring $\mathbf a = \mathbf b$.

\paragraph{Sloped boundary conditions}
Let $\xi : ( \bbZ^2 )^* \to \bbZ$ be a fixed height function and study the measures $\bbP_D^{\xi_{\vert \partial D}}$ in growing domains $D \nearrow \bbZ^2 $. Corollary~\ref{cor:variance_D} gives the height variance for flat enough boundary conditions: for instance if $\xi(x) - \xi(0) = O(1)$, we have $\Var_D^{\xi_{\vert \partial D}} ( h(x) ) \sim  \log d(x, \partial D)$.
One may also study boundary conditions that are not flat, most interestingly boundary conditions with a given slope: take a fixed ice-rule arrow configuration in an $N \times M$ torus, embed it periodically in the plane, and let $\xi$ be the corresponding height function on the faces of $\bbZ^2$. The \textit{slope} of $\xi$ is then the vector $s = ((\xi(y+(N,0))-\xi(y) ) / N, ( \xi(y+(0,M))- \xi(y) )/ M )$ (which is independent of $y \in( \bbZ^2 )^* $); note that
\begin{align*}
	\xi(x) - \xi(0) = \langle x, s \rangle + O(1).
\end{align*}
With different choices of $M,N, \xi$ in the above, the possible slopes $s$ are exactly the rational points of $[-1, 1] \times [-1, 1]$.

It is expected that a result similar to Corollary~\ref{cor:variance_D} holds under the measure $\bbP_D^{\xi_{\vert \partial D}}$, 
whenever the slope of $\xi$ is in the interior\footnote{For slopes on the boundary of $[-1, 1]^2$ one readily shows that the configuration inside $D$ freezes completely.}
of $[-1, 1]^2$ (Corollary~\ref{cor:variance_D} treats the zero-slope case).
For such boundary conditions, the RSW result is also expected to apply for $h -\xi$ instead of $h$, at sufficiently large scales.
Indeed, in a slightly different context, it was shown in  \cite{She05} that the height function delocalizes for non-zero slopes in $(-1,1)^2$. 
Then, the height function in finite domains is expected to converge to the unique infinite-volume one, and to delocalize logarithmically. 

\subsection*{Organization of the paper}
Section~\ref{sec: basic 6V height fcn ppties} introduces a toolbox of fundamental combinatorial properties of six-vertex height functions, which will be constantly applied in the sequel. 
Section~\ref{sec:RSW} presents crossing probability estimates, in particular the RSW-type result of Theorem~\ref{thm:RSW}; 
these do not rely on~\eqref{eq:BA} and are valid for all $\mathbf c \geq 1$. 
Section~\ref{sec:Annulus} contains the proofs of Theorems~\ref{thm:RSW_origins} and~\ref{thm:bxp}. 
The estimates on the variance of the height function of Theorem~\ref{thm:variance_torus} and Corollary~\ref{cor:variance_D} are proved in Section~\ref{sec:5}.

\section{Basic properties}
\label{sec: basic 6V height fcn ppties}

This section studies six-vertex height functions in (discrete domains embedded in) the plane, torus, or cylinder, with parameter $\mathbf{c} \geq 1$. As mentioned in the introduction, all the results in this section also hold for the three-parameter model when $\max \{ \mathbf{a}, \mathbf{b} \} \leq \mathbf{c}$.

\paragraph{On setup and notation}

We denote by $\bbG$ an ``ambient space graph'' that can be taken to be either the torus $\bbT_N$, the cylinder $\bbO_{N,M}$, or the whole plane $\bbZ^2$. By the terms vertex, edge, face, and dual edge we will always refer to those structures of $\bbG$. We will always assume that $N$ is even and hence the faces of $\bbG$ can be  bipartitioned into \textit{even} and \textit{odd} faces, such that no odd (resp.~even) face shares an edge with another odd (resp.~even) face.

A discrete domain $D \subset \bbG^*$ is defined for $\bbG=\bbT_N$ and $\bbG=\bbO_{N,M}$ similarly to the planar case in Section~\ref{subsec: main results in planar domains}.
Recall that a function $h: D \to \bbZ$ is a \textit{height function} if for any two adjacent faces $x $ and $y$ in $D$, we have $|h(x)-h(y)| =1$, and $h$ is even on even faces and odd on odd faces.
Let $\calH_D$ denote the set of such height functions on $D$.

Finally, recall from the introduction that a boundary condition $\xi$ (and thus its induced measure $\bbP^\xi_D$) may be defined on any nonempty set of faces $B \subset D$. 

\begin{figure}
	\begin{center}
		\includegraphics[width=0.45\textwidth]{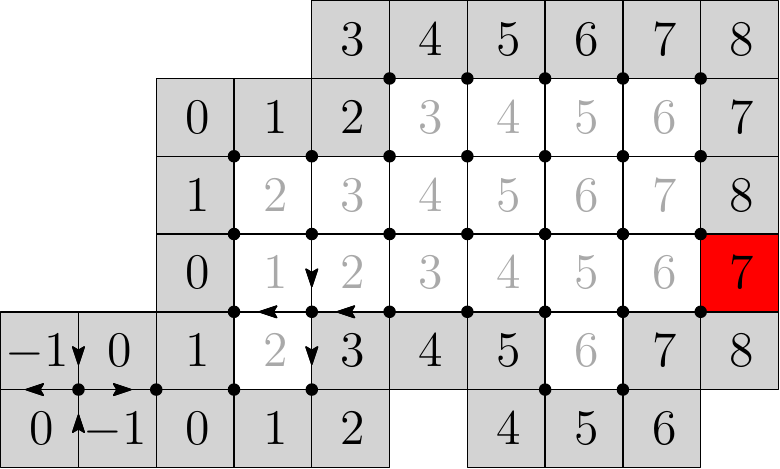}
	\caption{A discrete domain $D$ generated by the marked vertices. The boundary $\partial D$ is represented in grey and some of the arrows of the associated 6-vertex configuration are shown. The boundary condition is such that there exists a unique height function inside $D$ consistent with it (we say that the height function is frozen). If the red square were to contain a $9$ instead of a $7$, the boundary condition would become non-admissible.}
	\label{fig:domains}
	\end{center}
\end{figure}

\subsection{Spatial Markov property}\label{sec:SMP}

\begin{proposition}[Spatial Markov property (SMP)]\label{prop: SMP}
    Let $D \subset D'$ be two domains of $\bbG$ and $\xi$ be an admissible boundary condition on $\partial D'$. 
    Then for any realisation $\zeta$ of a height function chosen according to $\bbP_{D'}^\xi$, we have
    \begin{align}\label{eq:SMP}\tag{SMP}
    	\bbP_{D'}^{\xi} [ \; \cdot \; \vert \; h = \zeta \text{ on $D^c \cup \partial D$} ] = \bbP_D^{ \zeta_{
    \vert \partial D} }.
    \end{align}
\end{proposition}
Above, the left-hand side refers to the law of the height function restricted to $D$, written $h_{\vert D}$, when $h$ is sampled according to $\bbP_{D'}^{\xi}$.
Observe that the right-hand side of~\eqref{eq:SMP} only depends on the values of $\zeta$ on $\partial D$. 
In particular, this proves that conditionally on the values of $h$ on $\partial D$, the restrictions of the height function to $D$ and $D^c$ are independent.

\begin{proof}
For any height function $h$ equal to $\zeta$ on $D^c \cup \partial D$, 
\begin{align*}
	\bbP_{D'}^\xi[h]=\tfrac1{Z(D',\xi)}
	\prod_{v \in V(D)}\mathbf{c}^{\ind_{\{v \text{ is of type $5$ or $6$ in $h$}\}}} 
	\prod_{v \in V(D')\setminus V(D)}\mathbf{c}^{\ind_{\{v \text{ is of type $5$ or $6$ in $h$}\}}}.
\end{align*}
The second product above only depends on $\zeta$, 
since it only involves vertices for which all four surrounding faces have height prescribed by $\zeta$. 
Thus, the law of $h_{\vert D} $ under $\bbP_{D'}^{\xi} [ \; \cdot \; \vert \; h = \zeta \text{ on $D^c \cup \partial D$} ]$  has probabilities proportional to the first product above, and therefore to 
$\bbP_D^{ \zeta_{
    \vert \partial D} } [ h_{\vert D} ]$, with a factor of proportionality that depends on $\zeta_{
    \vert D^c \cup  \partial D}$. 
As these two measures are supported on the same set of height configurations, we conclude that they are equal.
\end{proof}

\subsection{Monotonicity properties and correlation inequalities}
\label{sec:monotonicity}

The six-vertex model enjoys  useful monotonicity properties with respect to its height function representation when $\mathbf a=\mathbf b=1$ and $\mathbf c\ge1$ (or in more general when $\max \{ \mathbf{a}, \mathbf{b} \} \leq \mathbf{c}$). We now state these properties. Proofs are given in Appendix~\ref{app: pf of FKG and CBC} since they are all not explicitly present in the literature\footnote{
It is also worthwhile to point out that the computations would yield counterexamples for these monotonicity results
when $\mathbf c\in(0,1)$, or $\max \{ \mathbf{a}, \mathbf{b} \} > \mathbf{c}$.
}.

An important concept in the study of both height functions and boundary conditions is the \emph{partial order relation} $\preceq$ on $\heightfcns_D$. For two height functions $h, h' \in \heightfcns_D$, we set $h \preceq h'$ if and only if $h(x) \leq h'(x)$ for all faces $x$ in $D$. An analogous partial order is defined between boundary conditions.

A function $F: \heightfcns_D \to \bbR$ is \emph{increasing} if
$h \preceq h' $ implies that $ F(h) \leq F(h')$.
An event $A$ is {\em increasing} if its indicator function $\mathbbm{1}_A$ is an increasing function, and \emph{decreasing} if its complement $A^c$ is increasing. The results below are stated in terms of expectations of increasing functions, but we will mostly apply them to probabilities of increasing events.

\begin{proposition}\label{prop: CBC and FKG}\label{prop:FKG}
    Fix a discrete domain $D$, any two admissible boundary conditions $\xi \preceq \xi'$ 
   	and any two increasing functions $F,G: \heightfcns_D \to \bbR$. Then, we have
    \begin{align}
    	\bbE^{\xi}_{D} [F(h)G(h)] &\geq \bbE^{\xi}_{D} [F(h)] \bbE^{\xi}_{D} [G(h)], \tag{FKG}\label{eq:FKG-h}\\
    	\bbE^{\xi'}_{D} [F (h)] &\geq \bbE^{\xi}_{D} [F (h)].\tag{CBC}\label{eq:CBC-h}    
    \end{align}
\end{proposition}

The proof of Propositions~\ref{prop:FKG} is in Appendix~\ref{app: pf of FKG and CBC}. 
For now, let us prove the following elementary corollary of~\eqref{eq:CBC-h}.

\begin{corollary}\label{cor:E[h]>0}
     Let $D$ be a discrete domain and $\xi$ an admissible boundary condition.
    If $\xi \ge m$ (resp. $\xi \le M$) then for any face $x$ of $D$, we have
    \begin{align*}
    \bbE^{\xi}_D [h (x)] \geq m \qquad (\text{resp. }
    \bbE^{\xi}_D [h (x)] \leq M).
    \end{align*}
\end{corollary}

\begin{proof}
	It suffices to prove the first bound for $m=0$ (the rest follows readily). 
	The comparison between boundary conditions and the invariance of weight under sign flip $W_\mathrm{6V}(h) = W_\mathrm{6V}(-h)$  give
    \begin{align*}
	    2\bbE^{\xi}_D [h (x)]\ge\bbE^{ \xi}_D [h (x)] + \bbE^{ -\xi}_D [h (x)] = 0,
    \end{align*}which is what we wanted to prove.
\end{proof}

Crucially, our model enjoys an additional monotonicity property for the absolute value of the height function. 

\begin{proposition}\label{prop:|h|-monotonicity}
	Fix a discrete domain $D$, $\xi' \succeq \xi \geq 0$ two admissible boundary conditions on $\partial D$,
	a (possibly empty) set of faces $B\subset D$, 
	and two height-functions $\zeta' \succeq \zeta \geq 0$ on $B$ achievable under $\bbP^{\xi'}_D$ and $\bbP^{\xi}_D$, respectively.
	Then, for any two increasing functions $F,G: \heightfcns_D \to \bbR$, we have
	\begin{align} 
		\bbE^{\xi'}_D \big[F( \vert h \vert) \,\big|\, |h| = \zeta' \text{ on $B$}\big]
		&\geq 
		\bbE^{\xi}_D \big[F( \vert h \vert) \,\big|\, |h| = \zeta \text{ on $B$}\big],
		\tag{CBC-|h|} \label{eq:CBC-|h|}\\
		 \bbE^{\xi}_D \big[F( \vert h \vert) G( \vert h \vert)\big] 
		&\geq \bbE^{\xi}_D \big[F( \vert h \vert)\big]  \bbE^{\xi}_D \big[ G( \vert h \vert)\big].
		 \tag{FKG-|h|} \label{eq:FKG-|h|}
	\end{align}
\end{proposition}

\begin{remark}
The inequality~\eqref{eq:FKG-|h|} also holds for the conditional measure $\bbP^{\xi}_D [. \,|\, |h| = \zeta \text{ on $B$}]$ by the same proof. The statements above also apply to boundary conditions $\xi, \xi'$ imposed on any non-empty set rather than just $\partial D$.
\end{remark}

\subsection{Boundary pulling and pushing}

In models with the spatial Markov property and monotonicity properties, 
a useful tool is the comparison of probabilities of certain events in different domains. 
This is sometimes referred to as the pushing/pulling of boundary conditions. 
In our model, it is achieved through the FKG inequality for the absolute value of the height function. 

In order to state the pushing/pulling property, we need the concept of minimal height functions. Let $D \subset \bbG^*$ be a discrete domain and $\xi$ be an admissible boundary condition defined on $B \subset D$.
The reader may verify that
\begin{align*}
\underline{h} (x) = \max_{y \in B} (\xi (y) - d_D(x, y)),
\end{align*}   
where $d_D$ denotes the graph distance on $D \subset \bbG^*$,
is the unique minimal height function $h$ with $h_{ \vert B} = \xi$. That is, for any other such $h$, we have $\underline h \preceq h$. Similarly, if $h_m$ is the height function taking only values $m$ and $m+1$, it holds that $h(\cdot)= \max\{ \underline{h}(\cdot), h_m (\cdot) \}$ is the unique minimal height function $h$ with $h_{ \vert B} = \xi$ with $h \geq m$. Maximal extensions can be constructed similarly.

\begin{proposition}\label{prop:pushing_1}
    Fix integers $k>m$. 
    Let $D \subset \bbG^*$ be a discrete domain, $\xi$ be an admissible boundary condition on $B \subset D$ with $\xi \geq m$
    and $\zeta \in \heightfcns_D$ the minimal height function with boundary conditions $\xi$ and with $\zeta \geq m-1$. 
    Then, for any $B' \supset B$, we have
    \begin{align}
	    \bbP^{\zeta_{ \vert B' } }_D [\exists C \in \calC \text{ with } h_{|C}\ge k ]
	    \leq 2\, \bbP^{\xi}_D [\exists C \in \calC \text{ with } h_{|C}\ge k],	\label{eq:pushing_0}
    \end{align}
    for any collection $\calC$ of connected subsets of $D$. 
	When each set in $\calC$  intersects $B$, then the factor $2$ may be removed. 
\end{proposition}

The above will mostly be used in the form of the following corollary. 

\begin{corollary}\label{cor:pushing_modified}
	Let $D \subset D'$ be two discrete domains, $\xi'$ be an admissible boundary condition for $\heightfcns_{D'}$ on $\partial D'$, with $\xi' \geq m$ for some $m$. 
	Let $\xi$ be the minimal admissible boundary condition for $\heightfcns_{D}$ on $\partial D$, that coincides with $\xi'$ on $\partial D \cap \partial D'$ and satisfies $\xi \geq m$. 
	Then, for any $k > m$, 
	\begin{align}	\label{eq:pushing_11}
		\bbP^{\xi}_{D} [\exists C \in \calC \text{ with } h_{|C}\ge k+2] \le 
    	2\,\bbP^{\xi'}_{D'} [\exists C \in \calC \text{ with } h_{|C}\ge k],
	\end{align}
	for any collection $\calC$ of connected subsets of $D$. 
	When each set in $\calC$ intersects $\partial D'$, then the factor $2$ may be removed. 
\end{corollary}

The corollary will be applied to the existence of certain paths, most commonly crossings of certain domains.
%
Two things should be kept in mind when applying Corollary~\ref{cor:pushing_modified}. 
First, due to~\eqref{eq:CBC-h},~\eqref{eq:pushing_11} also applies to pairs of boundary conditions $\tilde\xi,\tilde\xi'$ with 
$\tilde\xi \preceq \xi$ and $\xi' \preceq \tilde \xi'$. 
Second, even though the statement suggests that $\xi$ is chosen in terms of $\xi'$, we will sometimes start with a boundary condition $\xi$, 
then construct a boundary condition $\xi'$ for which~\eqref{eq:pushing_11} holds. The two cases correspond to boundary pushing and pulling.

\begin{proof}[Proposition~\ref{prop:pushing_1}]
	Since the model is invariant under the addition of a constant, we may limit ourselves to the case $m =0$. 
	Fix a set $\calC$ of connected subsets of $D$.
	Write $\calA = \{ h\,:\, \exists C \in \calC \text{ with }h_{|C} \geq k\}$.
	
	Since $k>0$, if $|h|\in \calA$, then there exits  $C \in \calC$ on which $|h|\ge k$ and in particular $h$ is of constant sign. 
	As a consequence
	$$
	\bbP^{\xi}_D [h\in \calA]+\bbP^{\xi}_D [-h\in \calA] \ge \bbP^{\xi}_D [|h|\in \calA].
	$$
	By sign flip symmetry and comparison between boundary conditions~\eqref{eq:CBC-h} 
	(recall that $\xi\ge 0$, and hence $-\xi \preceq \xi$), we find that
	\begin{align}\label{eq:22remove}
		2\bbP^{\xi}_D [h \in \calA]\ge \bbP^{\xi}_D [|h|\in \calA].
	\end{align}
	
	It remains to lower bound the right-hand side.
	Observe that the lowest possible values of $|h|$ on $B$ are given by $|\zeta|$. By~\eqref{eq:FKG-|h|},
	\begin{align*}
		\bbP^{\xi}_D [|h|\in \calA] 
		&\geq \bbP^{\xi}_D [|h|\in \calA \,|\, |h| = |\zeta| \text{ on $B$}]\\
		&\geq \min\big\{\bbP^{\xi}_D [h\in \calA \,|\, h = \zeta' \text{ on $B$}]\,:\, \zeta' \text{ admissible s.t. } |\zeta'| = |\zeta|\big\}.
	\end{align*}
	Due to~\eqref{eq:FKG-h} and to the fact that $\calA$ is increasing, the minimum above is realised by the lowest configuration $\zeta'$ satisfying the condition above. The choice of $\zeta$ as lowest among the realisations of $h$ on $B$ with $\zeta \geq -1$ guarantees that the minimum in the above is obtained when $\zeta' = \zeta$.
	Combining this observation with~\eqref{eq:22remove} provides the desired bound. 
	
	Finally, if $\calC$ is such that all $C \in \calC$ intersect $\partial D$, 
	then $\bbP^{\xi}_D [h \in \calA] =  \bbP^{\xi}_D [|h|\in \calA]$. 
	Indeed, when $|h| \in \calA$, the sign of $h$ on any set $C\in \calC$ realising $\calA$ is necessarily $+$, due to its intersection with $\partial D$ and to the fact that $\xi \geq 0$ (since $k > 0$, $C$ intersects the boundary only on faces where $\xi > 0$). 
	Thus, in this particular case, the factor $2$ may be removed from~\eqref{eq:pushing_0}.
\end{proof}

\begin{proof}[Corollary~\ref{cor:pushing_modified}]
	Fix $D\subset D'$, $\xi$ and $\xi'$ as in the statement. 
	Let $\zeta$ be the smallest realisation of a height function on $D'$ with boundary conditions $\xi'$ and with $\zeta \geq m-1$. 
	Then, due to Proposition~\ref{prop:pushing_1} and~\eqref{eq:SMP},
	\begin{align*}
	    \bbP^{\xi'}_{D'} [\exists C \in \calC \text{ with } h_{|C}\ge k]
	    \geq \tfrac12 \bbP^{\zeta_{\vert \partial D} }_{D} [\exists C \in \calC \text{ with } h_{|C}\ge k ].
	\end{align*}
	Notice now that, by choice of $\xi$, we have $\zeta_{\vert \partial D} \succeq \xi - 2$, and~\eqref{eq:CBC-h} thus gives
	\begin{align*}
	    \bbP^{\zeta_{\vert \partial D} }_{D} [\exists C \in \calC \text{ with } h_{|C}\ge k ]
	    \geq \bbP^{\xi-2}_{D} [\exists C \in \calC \text{ with } h_{|C}\ge k]
	   	= \bbP^{\xi}_{D} [\exists C \in \calC \text{ with } h_{|C}\ge k+2].
	\end{align*}
	The claim follows.
	When all sets in $\calC$ intersect $\partial D'$, the factor $1/2$ disappears in the first equation displayed above. 
\end{proof}

\section{RSW theory}\label{sec:RSW}

\newcommand{\Strip}{\mathrm{Strip}}
This section introduces tools of a geometric nature for the six-vertex height functions, 
related to crossings of domains by height function level sets. 
The main result is the Russo--Seymour--Welsh (RSW) Theorem~\ref{thm:RSW}. 
An intermediate result, Proposition~\ref{prop:RSW1} will also be used later, in Section~\ref{sec:4}. In this section, we only work in the plane.

The results in this section work for all $\mathbf{c} \geq 1$, i.e., they do not differentiate between the localized and delocalized phases. However, as discussed in Section~\ref{subsec:intro - further questions}, they do \textit{not} directly generalize for the six-vertex model with three parameters $\mathbf{a}, \mathbf{b}, \mathbf{c}$. The reader will also notice that various inexplicit constants appear in the statements of this section. Explicit values for these constants could be worked out by carefully tracing through the proofs, but this is not needed for the purpose of this paper. An interesting consequence is nevertheless that the inexplicit constants may be chosen uniformly in $\mathbf{c} \geq 1$. 

\subsection{The main RSW result}

Given a discrete domain $D$ and sets $A, B $ of faces of $ D$, write $A \xlra{h \geq k \text{ in $D$}} B$  for the event that there exists a path of faces
$u_0, \dots, u_n$ of $D$ with $u_0 \in A$, $u_n \in B$, $u_i$ adjacent to $u_{i+1}$ in $D$  for all $i$ and $h(u_i) \geq k$ for all $i$. 
When no ambiguity is possible, we remove the mention to $D$ from the notation. 
The same notation applies with $h \leq k$ and $|h|$ instead of $h$.  
 
For convenience, we will work here with the following measures in infinite horizontal strips.
Fix $n \geq 2$ and set $\Strip := \bbZ \times [0,n]$ seen as a subgraph of $\bbZ^2$. 
Its boundary $\partial \Strip$ is formed of the faces in $\bbZ \times ([0,1] \cup [n-1,n])$; 
the notion of admissible boundary condition on $\partial \Strip$ adapts readily from that on (finite) domains. 
Fix an admissible boundary condition $\xi$ on $\Strip$ with $|\xi| \preceq M$ for some $M \geq 1$. 
The measure $\bbP_{\Strip}^\xi$ is defined as the weak limit of measures $\bbP_{[-m,m]\times [0,n]}^{\xi_m}$ as $m \to \infty$,
where $\xi_m$ is the minimal boundary condition on $\partial [-m,m]\times [0,n]$ which is equal to $\xi$ on $\partial \Strip \cap \partial [-m,m]\times [0,n]$. It is an immediate consequence of the finite energy of the model 
that $\bbP_{\Strip}^\xi$ exists. Furthermore, by the same argument, $\bbP_{\Strip}^\xi$  is the limit of any sequence of measures $\bbP_{D_m}^{\zeta_m}$, where 
$D_m$ is any increasing sequence of domains with $\bigcup_{m \geq 1}D_m = \Strip$ 
and $\zeta_m$ is any sequence of boundary conditions on $\partial D_m$ that are equal to $\xi$ on $\partial \Strip \cap \partial D_m$. 

Note that as a consequence of this construction, the spatial Markov property~\eqref{eq:SMP}, the FKG inequalities~\eqref{eq:FKG-h} and~\eqref{eq:FKG-|h|}, the comparison of boundary conditions~\eqref{eq:CBC-h} and~\eqref{eq:CBC-|h|}
and the pushing of boundary conditions~\eqref{eq:pushing_11} apply to $\bbP_\Strip^\xi$.

We are now ready to state the main result of this section. 

\begin{theorem}[RSW]\label{thm:RSW}
	There exist absolute constants $\delta, c,C>0$ such that for any $k \geq 1/c$ and any $n$, we have
	\begin{align}\label{eq:RSW}
	\bbP_{\Lambda_{12 n}}^{0,1}[\calO_{h\ge ck}( 6n ,12 n)]
	\geq c \big(\bbP_{\bbZ \times [-n,2n]}^{0,1}\big[[0, \lfloor \delta n \rfloor ]\times \{0\} \xleftrightarrow{h \geq k \; \mathrm{in} \; \bbZ \times [0,n]}\bbZ\times \{ n\}\big]\big)^{C},
	\end{align}
where $0,1$ denote the admissible boundary conditions on the boundary of each domain that only take values 0 and 1.
\end{theorem}

\begin{figure}
    \begin{center}
    \includegraphics[width=0.35\textwidth]{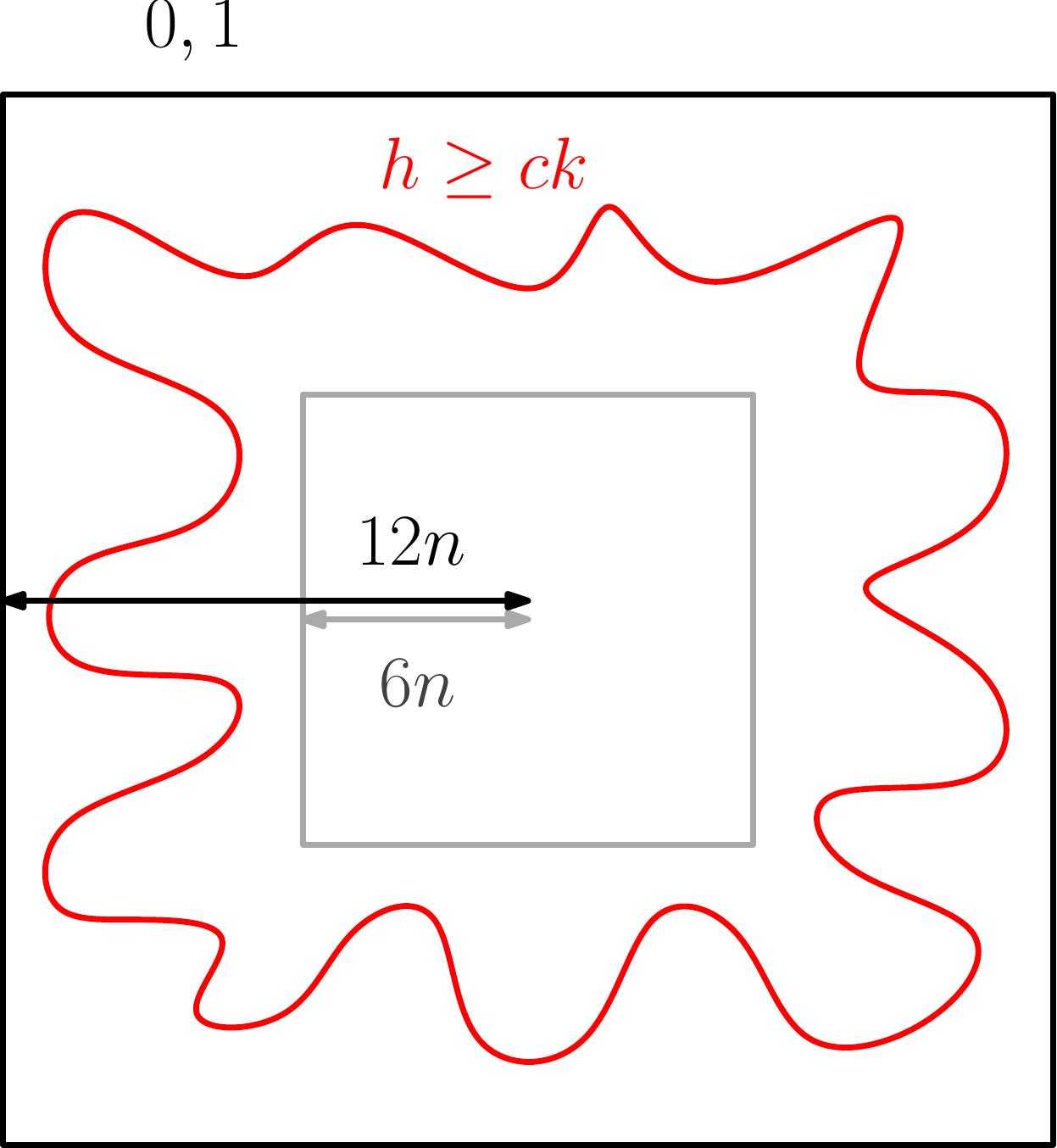} \quad \qquad
    \includegraphics[width=0.50\textwidth]{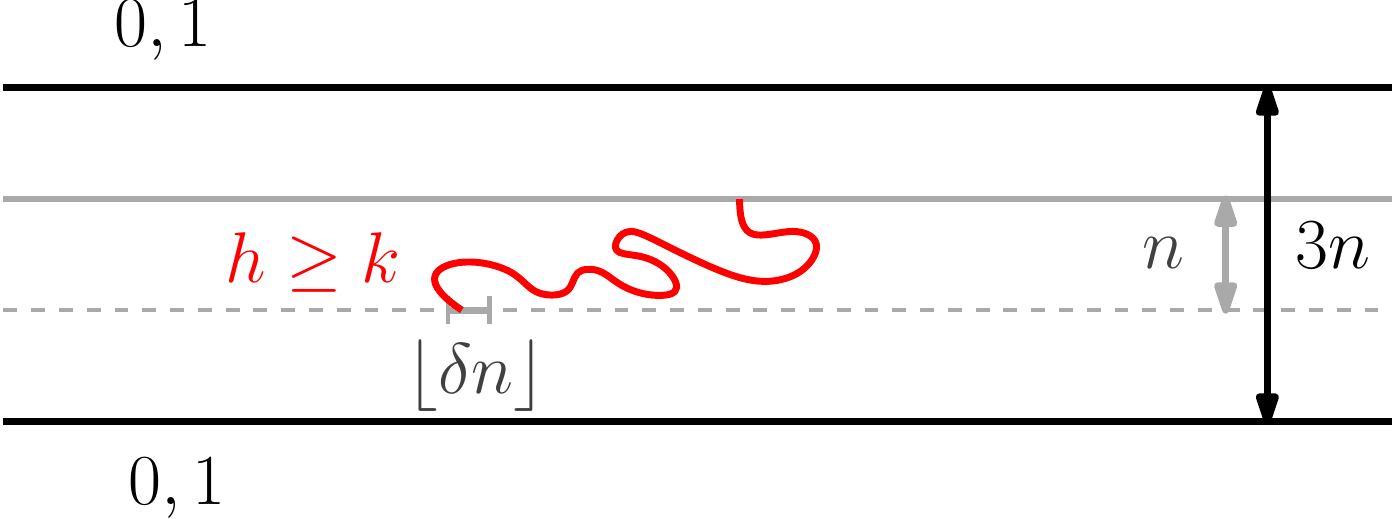}
    \caption{
    Theorem~\ref{thm:RSW} lower-bounds the probability of a large-height circuit around an annulus (left) in terms of that of a vertical crossing of the middle third of a strip, starting from a given narrow window of length $\lfloor\delta n\rfloor$ (right).}
    \end{center}
\end{figure}

The rest of this section is organised as follows. 
In Section~\ref{subsec: symmetric domain crossings}, we discuss duality properties and crossings of certain symmetric domains. In Section~\ref{sec:gluing_cross}, we prove a result about vertical crossings of a strip with endpoints contained in small intervals. 
This result is used to prove to Theorem~\ref{thm:RSW} and will also be used in Section~\ref{sec:4}. 
In Section~\ref{sec:long_in_stip} we use the result of Section~\ref{sec:gluing_cross} to bound the probability of horizontal crossings of long rectangles in a strip. Then, in Section~\ref{sec:strip_to_annulus}, the previous bounds are extended to circuits in annuli, thus proving Theorem~\ref{thm:RSW}. 

\subsection{Crossings of symmetric quadrilaterals}\label{subsec: symmetric domain crossings}

A discrete domain $D$ is said to be \textit{simply-connected} if it is the subgraph of $(\bbZ^2)^*$ bounded on or inside a simple loop on $(\bbZ^2)^*$. (The corresponding primal vertices $V(D) \subset \bbZ^2$ are hence those inside the loop.)
When four (different) faces $a,b,c,d$ in counter-clockwise order on the bounding loop are fixed, $(D;a,b,c,d)$ is called a \textit{(discrete) quad}. The boundary of a quad is divided into four arcs $(ab)$, $(bc)$, $(cd)$, and $(da)$, that are paths on $(\bbZ^2)^*$ intersecting at their extremities.

For a discrete domain $D$, we say that two faces $u$ and $v$ are \textit{$\times$-adjacent} in $D$ if $d_D(u, v) = 2$ and $u$ and $v$ share a corner; a \textit{$\times$-path} in $D$ is a sequence of $\times$-adjacent faces.
For sets $A, B $ of faces of $ D$ we write $A \xlra{h \geq k \text{ in $D$}}_{ \times } B$  for the event that there exists a $\times$-path $u_0, \dots, u_n \in D$ with $u_0 \in A$, $u_n \in B$ and $h(u_i) \geq k$ for all $i$; similar notations are used for $h \geq k$ and $\vert h \vert$, and ``in $D$'' is omitted if clear.

\begin{remark}\label{rem:dual_crossings}
	The $\times$-paths are the dual of ordinary paths, in the sense that for a quad $(D; a, b, c, d)$, we have
	\begin{align}
	\label{eq:dual_crossings}
 \{ (ab) \xlra{h \geq k } (cd) \}^c
 =  \{(bc) \xlra{h < k }_{ \times } (da)\}.
	\end{align}
	See Figure~\ref{fig:symmetric domains} for an explanation. Furthermore, we have
	\begin{align}
	\label{eq:from times-crossings to crossings}
		\{(bc) \xlra{h \leq k-2 } (da)\}
		\subset 
		\{(bc) \xlra{h < k }_{ \times } (da)\}
		\subset
		\{(bc) \xlra{h \leq k } (da)\}.
	\end{align}
\end{remark}

A \emph{symmetry} $\sigma$ of $\bbZ^2$ is a graph isomorphism from $\bbZ^2$ to itself that fixes the bipartition of $F(\bbZ^2)$ (even faces are sent to even faces, and odd faces to odd faces). Given an admissible boundary condition $\xi$ on $D$, we denote by $\sigma \xi$ the admissible boundary condition $\xi \circ \sigma^{-1}$ on $\sigma (D)$. A discrete quad $(D;a,b,c,d)$ is said to be \emph{symmetric} if there exists a symmetry $\sigma: \bbZ^2 \to \bbZ^2$ such that $\sigma(D) = D$ and $\sigma$ maps the boundary arcs $(ab)$ and $(cd)$ of $D$ to $(bc)$ and $(da)$.
 
\begin{figure}
\begin{center}
\includegraphics[height=0.45\textwidth]{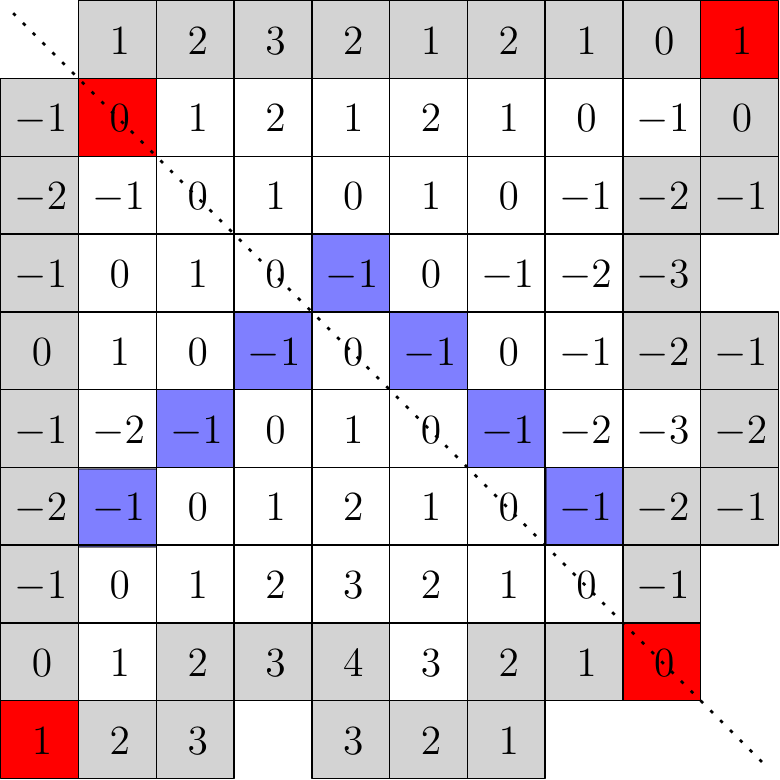}
\caption{A quad (with marked faces at its corners) which is symmetric with respect to the reflection $\sigma$ along the diagonal. 
The boundary condition $\xi$ is such that $-\xi \preceq \sigma \xi$, as required in Lemma~\ref{lem:symmetric_domain}. 
This domain contains no path of non-negative height from top to bottom, but contains a negative $\times$-crossing from left to right.
Such $\times$-crossings are not necessarily unique; one may be found by following the contour of the connected component of non-negative faces of the bottom arc.}
\label{fig:symmetric domains}
\end{center}
\end{figure}

\begin{lemma}[Crossing probability in symmetric domains]\label{lem:symmetric_domain}
	Let $(D;a,b,c,d)$ be a discrete quad which is symmetric with respect to a symmetry $\sigma$.
	For any boundary condition $\xi$ on $\partial D$ such that $\sigma \xi  \succeq -\xi $, we have 
	\begin{align*}
		\bbP_D^\xi[  (ab) \xlra{h \geq 0 } (cd) ] \ge \tfrac{1}{2}.
	\end{align*}
\end{lemma}

\begin{proof}
	Using first~\eqref{eq:dual_crossings} and then~\eqref{eq:from times-crossings to crossings}, we deduce that 
	\begin{align*}
	1 - \bbP_D^\xi[  (ab) \xlra{h \geq 0 } (cd)  ] = \bbP_D^\xi[  (bc) \xlra{h <0 }_{ \times } (da) ] \leq \bbP_D^\xi[  (bc) \xlra{h \leq 0 } (da)].
	\end{align*}
	Applying the symmetry $\sigma$ (and the fact that symmetries preserve six-vertex weights in our parameter range) we get
	\begin{align*}
    	\bbP_D^\xi[ (bc) \xlra{h \leq 0 } (da)]
    	= \bbP_{D}^{\sigma \xi}[ (ab) \xlra{h \leq 0 } (cd)] 
    	\le \bbP_{D}^{-\xi}[  (ab) \xlra{h \leq 0 } (cd) ] 
    	= \bbP_{D}^{\xi}[  (ab) \xlra{h \geq 0 } (cd)],
	\end{align*}
	where the inequality follows from the comparison of boundary conditions of Proposition~\ref{prop: CBC and FKG} since $\sigma\xi \succeq - \xi$ and the event $(ab) \xlra{h \leq 0 } (cd) $ is decreasing. The claim follows by combining the two displayed equations above.
\end{proof}

\subsection{No crossings between slits}\label{sec:gluing_cross}

For the rest of this section, aiming to prove Theorem~\ref{thm:RSW}, we omit the integer roundings in $\lfloor \delta n \rfloor$ to streamline the notation; $\delta n$ thus always represents an integer.

\begin{figure}
\begin{center}
\includegraphics[width=0.45\textwidth, page=1]{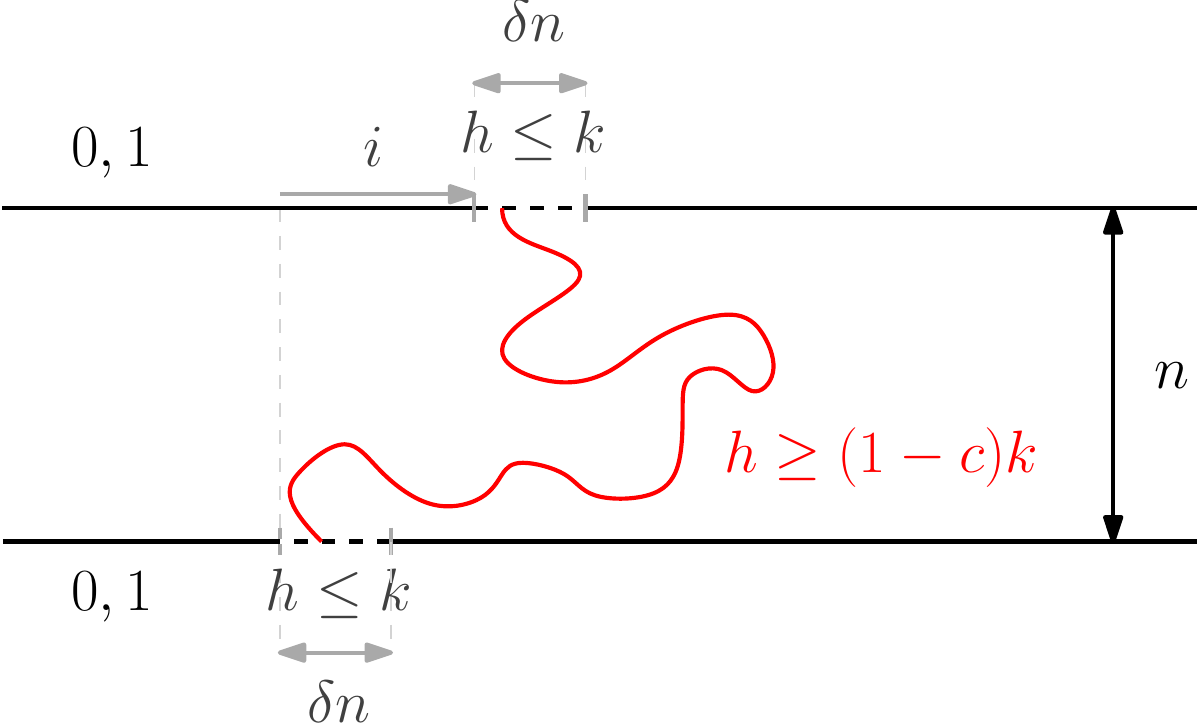} \quad
\includegraphics[width=0.51\textwidth, page=1]{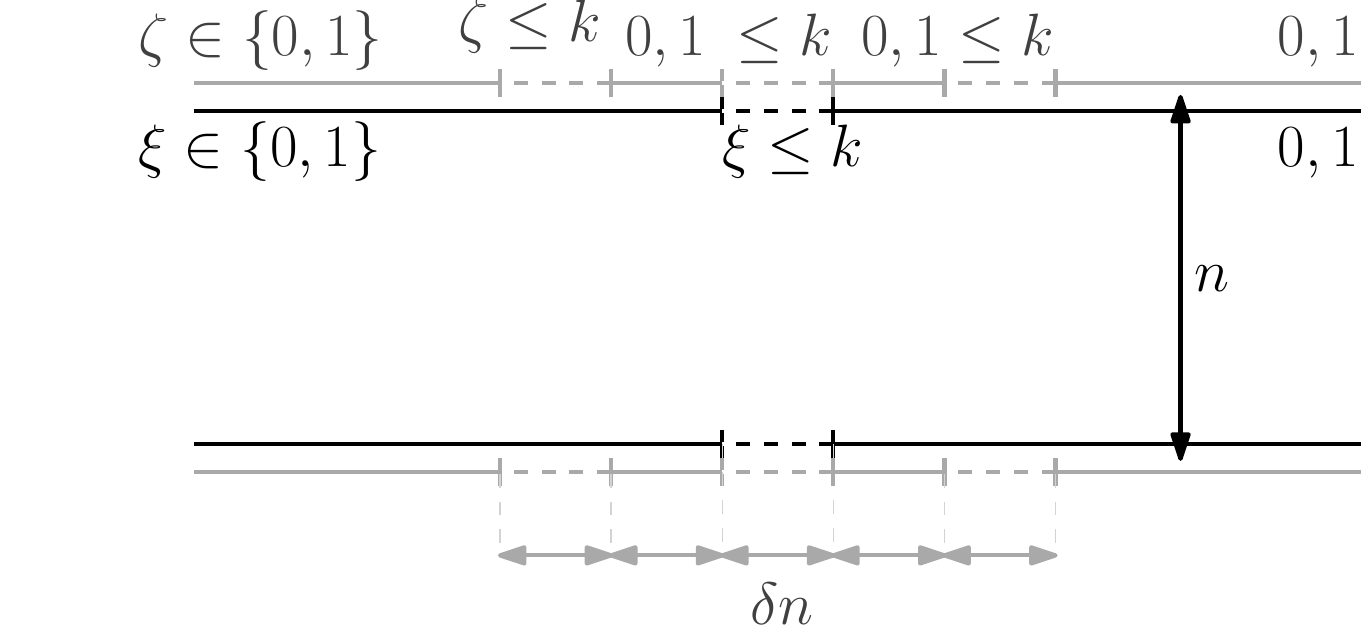}
\caption{
\label{fig:no vertical crossings}
{\em Left:} an illustration of the boundary condition and the crossing event in Proposition~\ref{prop:RSW1}.
In the top and bottom intervals of length $\delta n$ (which we call slits) the boundary conditions are $k,k-1$, except at their ends, where they progressively decrease to $0$. One should think of $\delta n$ as being much larger than $k$. 
{\em Right:} illustration for the proof in the case $i=0$: the boundary conditions $\xi$ (black) and $\zeta$ (gray) -- the same color code as in the left picture applies. The geometry and the definition of boundary conditions are similar on the lower and upper boundaries of the strip.
}
\end{center}
\end{figure}

\begin{proposition}\label{prop:RSW1}
	There exist constants $\delta, c >0$ such that the following holds. 
	For any $k \geq 1/c$, any $n$
	and any $i \in \bbZ$
	\begin{align}\label{eq:RSW1}
		\bbP_{\bbZ \times [0,n]}^{\xi}\big[[0,\delta n]\times \{0\} \xleftrightarrow{h \geq (1-c) k}[i,i+\delta n]\times \{n\}\big] \leq 1-c,
	\end{align}
	where $\xi$ is the largest boundary condition on $\partial \bbZ \times [0,n]$ which is at most $k$ and has
	 values $0,1$ on $(\bbZ \setminus [0,\delta n])\times \{0\}$ and $(\bbZ \setminus [i, i+\delta n])\times \{n\}$.
\end{proposition}
See Figure~\ref{fig:no vertical crossings}(left) for an illustration.

The proposition above will be used twice: once as the key step in the proof of Theorem~\ref{thm:RSW} 
and again to build so-called ``fences'' in Section~\ref{sec:4}.
It may be useful to adopt a dual view of the result above. 
Indeed, due to Remark~\ref{rem:dual_crossings}, the above shows that, 
in spite of the large boundary conditions (roughly) $k$ on the slits $(\bbZ \setminus [0,\delta n])\times \{0\}$ and $(\bbZ \setminus [i, i+\delta n])\times \{n\}$, 
one may construct with positive probability a path of at most $(1-c)k$ disconnecting these slits from each other.

\begin{proof}[Proposition~\ref{prop:RSW1}]
We will start by proving the statement for $i = 0$; the statement for general $i$ follows by a simple manipulation. For the rest of this subsection, we omit the subscripts $\bbZ \times [0,n]$ in the strip measures.
\smallskip

\noindent
{\bf Case $i = 0$:}
Fix $\delta = 1/17$
and integers $k$ and $n$ with $kc \geq 1$, where $c >0 $ is a constant whose value will be specified later and will not depend on $k$ or $n$.
For $j \in \bbZ$, write $L_j$ for the vertical line $\{ j \delta n \} \times \bbR$, 
\begin{align*}
	I_j:=[ j \delta n , (j+1) \delta n  ] \times\{ 0\} \quad \text{ and }\quad 
	\tilde I_j:=[ j \delta n , (j+1) \delta n   ] \times\{ n\}.
\end{align*}
Then, for $\alpha \in \{0,+,-\}$ and $k \geq 1$, define 
$\calE_{h\ge k}(j ,\alpha)$ as the event that there exists a path of $h\ge {k}$
from $I_j$ to $\tilde I_j$ in the strip $\bbR \times [0, n]$, and furthermore 
\begin{itemize}[noitemsep,nolistsep]
	\item if $\alpha=0$, the path intersects neither $L_{j-5}$ nor $L_{j+6}$,
	\item if $\alpha=+$, the path intersects $L_{j+6}$; 
	\item if $\alpha=-$, the path intersects $L_{j-5}$.
\end{itemize}
Similar notations apply for $h \leq k$ and for $\vert h \vert$.
Notice that the events $\calE_{h\ge k}(j ,\alpha)$ for $\alpha \in \{0,+,-\}$ are all increasing, they are not mutually exclusive, 
and we have
\begin{align}
\label{eq:union of cross events}
\bigcup_{ \alpha \in \{-,0,+\} } \calE_{h\ge k}(j ,\alpha)
\; = \; \{I_j \xlra{h\ge k \text{ in $\bbR \times [0, n]$}} \tilde I_j\}.
\end{align}

Write $\zeta$ for the largest boundary condition for the strip $\bbZ \times [0,n]$, 
which takes values at most $k$ and values in $\{ 0,1 \}$ outside of $I_{-2}$, $I_0$, $I_{2}$ and their top counterparts $\tilde I_{-2}$, $\tilde I_0$ and $\tilde I_{2}$; see Figure~\ref{fig:no vertical crossings}(right) for an illustration.
Next we state a lemma that will quickly imply the desired result. 

\begin{lemma}\label{lem:tricky}
    For $c =1/4$, any $k \geq 44$, and any $\alpha \in \{0,+,-\}$, we have
    \begin{align}\label{eq:tricky}
        \bbP^{\xi}[I_0 \xleftrightarrow{h \geq (1-c)k} \tilde I_0]
        \leq 1-\tfrac{1}{6}\bbP^\zeta[\calE_{h \leq c k}(-3,\alpha) \cap \calE_{h \leq c k}(-1,\alpha)  \cap \calE_{h \leq c k}(1,\alpha) \cap \calE_{h \leq c k}(3,\alpha)],
    \end{align}
    where $\xi$ the boundary condition defined in Proposition~\ref{prop:RSW1}, for $i = 0$. 
\end{lemma}

Before proving the lemma, let us see how it allows us to conclude the proof of Proposition~\ref{prop:RSW1}. 
For an integer $j$, let $\tau_j \xi$ be the horizontal shift of the boundary condition $\xi$ by $j\delta n$, and remark that by definition of $\xi$ and $\zeta$, if $j$ is an odd integer we have $\tau_j \xi \le k-\zeta$.
Using first the sign flip invariance of the height function and then the~\eqref{eq:CBC-h} inequality and the horizontal shift symmetry\footnote{
Note that we shift the boundary condition by $k$, which swaps the odd and even faces in case $k$ is odd. They can be swapped back by shifting the strip horizontally by $(1,0)$; alternatively, the reader may observe that it suffices to prove the claim for even $k$ here. 
We will keep these manipulations implicit in the subsequent parity swaps occurring in the rest of the article.
	}, we have, for all $j \in \{ \pm 3, \pm 1\}$ and $\alpha \in \{-,0,+\}$,
\begin{align*}
	\bbP^\zeta[\calE_{h \leq c k}(j,\alpha)]
	= 	\bbP^{k-\zeta}[\calE_{h \geq (1-c) k}(j,\alpha)] 
	 \geq \bbP^\xi[\calE_{h \geq (1-c) k}(0,\alpha)].
\end{align*}
By~\eqref{eq:union of cross events}, we conclude that there exists $\alpha_0 \in \{-,0,+\}$  such that
\begin{align*}
	\bbP^\xi[\calE_{h \geq (1-c) k}(0,\alpha_0)] \geq \tfrac13 \bbP^{\xi}[I_0 \xleftrightarrow{h \geq (1-c)k } \tilde I_0].
\end{align*}
Using first the FKG inequality for the decreasing events $\calE_{h \leq c k}(j,\alpha_0)$, and then the two previously displayed equations, we get
\begin{align*}
	\bbP^\zeta  [\calE_{h \leq c k}(-3,\alpha_0 ) \cap \calE_{h \leq c k}(-1,\alpha_0 )  \cap \calE_{h \leq c k}(1,\alpha_0 ) \cap \calE_{h \leq c k}	(3,\alpha_0 )] 
 &\geq 
	\prod_{j \in \{ \pm 3, \pm 1 \} } \bbP^\zeta  [\calE_{h \leq c k}(j,\alpha_0 )  ]
	\\
	& \geq 
	\Big(\tfrac{1}3\bbP^{\xi}[I_0 \xleftrightarrow{h \geq (1-c) k} \tilde I_0]\Big)^4.
\end{align*}
Injecting this into~\eqref{eq:tricky}, we conclude that for $c=1/4$, $\bbP^{\xi}[I_0 \xleftrightarrow{h \geq (1-c) k} \tilde I_0]$ is bounded above by an absolute constant strictly smaller than $1$. 
Adjusting the value of $c$ to a smaller constant if need be, we find~\eqref{eq:RSW1} with $i =0$. 
\medskip 

\noindent{\bf Case $ i \neq 0$:}
    Fix constants $\delta$ and $c$ so that~\eqref{eq:RSW1} holds for $i = 0$ with these constants. 
    Define $\delta' = \delta /3$; we will prove Proposition~\ref{prop:RSW1} for $\delta'$ instead of $\delta$, and omit integer roundings also in $\lfloor \delta' n \rfloor$. 
    By the~\eqref{eq:CBC-h} inequality, for any $i \in [-2 \delta' n, 2 \delta' n]$, 
	  \begin{align*}
		\bbP^{\xi'}\big[[0,\delta' n]\times \{0\} \xleftrightarrow{h \geq (1-c) k}&[i,i+\delta' n]\times \{n\}\big]  \\
		& \leq \bbP^{\xi}\big[[0,\delta n]\times \{0\} \xleftrightarrow{h \geq (1-c) k}[0,\delta n]\times \{n\}\big]
		\leq (1-c),
	\end{align*}
    where $\xi'\le \xi$ are the boundary conditions on $\partial \bbZ \times [0,n]$ 
    defined in the statement of Proposition~\ref{prop:RSW1} for $\delta'$ and $\delta$, respectively, and for $i=0$ for the latter.

	It thus remains to prove Proposition~\ref{prop:RSW1} for $|i| > 2 \delta' n$. By vertical reflection symmetry of the model, we may assume $i > 2 \delta' n$.
	Now, using first the sign flip invariance of the height function and then the~\eqref{eq:CBC-h} inequality 
	and the vertical reflection symmetry (see Figure~\ref{fig:contradictory crossings}), we compute
	\begin{align*}
		\bbP^{\xi'} & \big[[\delta' n, 2\delta'n]\times \{0\} \xleftrightarrow{h \leq c k}[-i+ \delta'n,-i + 2\delta'n]
		\times \{n\}\big]  	
		\\
		&= \bbP^{k-\xi'}  \big[[\delta' n, 2\delta'n]\times \{0\} \xleftrightarrow{h \geq (1-c) k}[-i+ \delta'n,-i + 2\delta'n]
		\times \{n\}\big]  
		\\
		& \geq \bbP^{\xi'}\big[[0,\delta' n]\times \{0\} \xleftrightarrow{h \geq (1-c) k}[i,i+\delta' n]\times \{n\}\big]. 
	\end{align*}
	Moreover, notice that as $c \leq 1/4$ (due to the assumption of previous lemma) and $k \geq 1/c \geq 4$,
	the event on the left hand side above excludes the one on the right-hand side (see Figure~\ref{fig:contradictory crossings} again).   
	We conclude that 
	\begin{align*}
		\bbP^{\xi'}\big[[0,\delta' n]\times \{0\} \xleftrightarrow{h \geq (1-c) k}[i,i+\delta' n]\times \{n\}\big] \leq \tfrac12.
	\end{align*}
\end{proof}

\begin{figure}
\begin{center}
\includegraphics[width=0.5\textwidth]{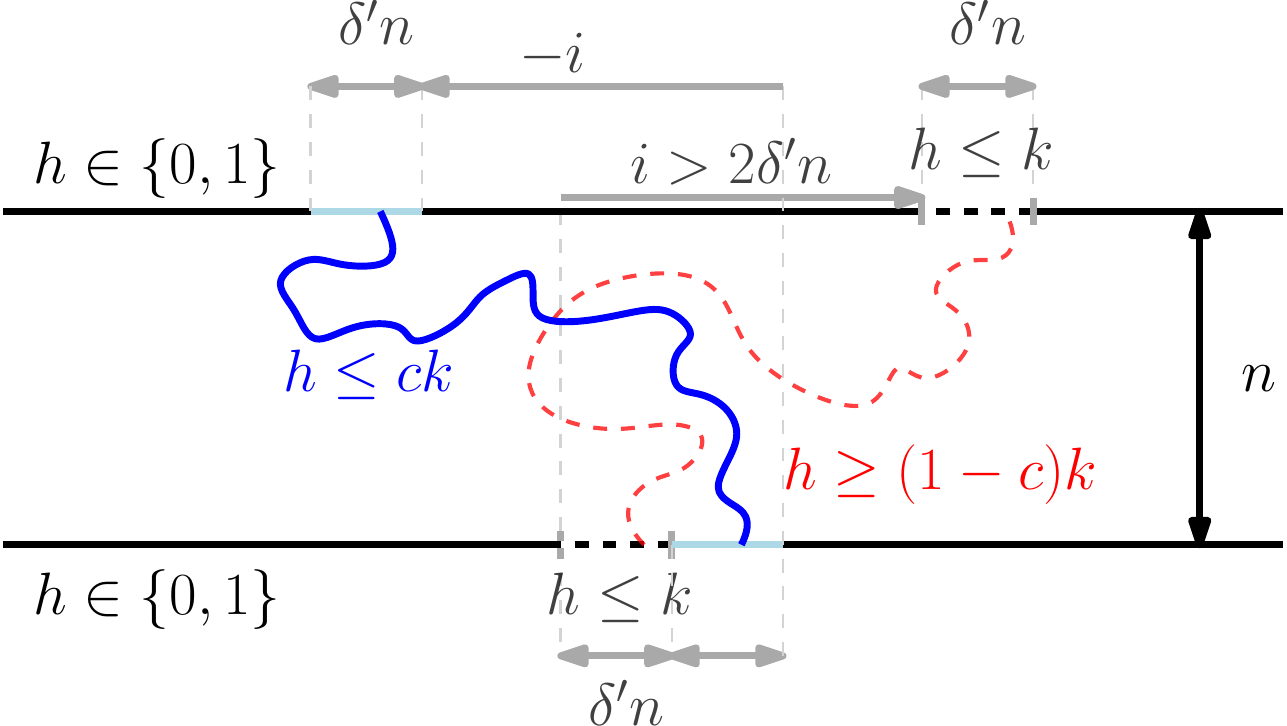}
\caption{
\label{fig:contradictory crossings}
A geometric argument in the end of the proof of Proposition~\ref{prop:RSW1}. The light-blue and dashed boundary segments are reflection symmetric. The small-height crossing between the light-blue boundary segments (in blue) excludes the large-height crossing between the dashed boundary segments (in dashed red). Solid-line boundaries represent here boundary conditions $0,1$ and dashed boundaries represent their maximal extensions smaller or equal to $k$.
}
\end{center}
\end{figure}
 
We now give the proof of Lemma~\ref{lem:tricky}. 

\begin{proof}[Lemma~\ref{lem:tricky}]
	Fix $\alpha \in \{0,-,+\}$. Set 
	\begin{equation*}
		\calS:=\calE_{h \leq c k}(-3,\alpha) \cap \calE_{h \leq c k}(-1,\alpha)  \cap \calE_{h \leq c k}(1,\alpha) \cap \calE_{h \leq c k}	(3,\alpha)
		\cap  \{I_{-2} \xleftrightarrow{h \geq (1-c)k } \tilde I_{-2}\} \cap  \{I_2 \xleftrightarrow{h \geq (1-c)k } \tilde I_2\}.
	\end{equation*}
	We first claim that it suffices to prove that
	\begin{align}\label{eq:tricky2}
    	\bbP^{\zeta}[I_0 \xleftrightarrow{h \geq (1-c)k } \tilde I_0 \,|\, \calS ] 
    	\leq 1/2.
	\end{align}
	Indeed,~\eqref{eq:tricky2} implies that 
	\begin{align*}
	\bbP^{\zeta}  \Big[ 
	\Big(
	\bigcap_{j \in \{ 0, \pm 2 \} } \{I_j \xlra{h \geq (1-c) k } \tilde I_j \}
	\Big)^c
	 \Big]
    	&\geq \tfrac12	\bbP^{\zeta}[ \calE_{h \leq c k}(-3,\alpha) \cap \calE_{h \leq c k}(-1,\alpha)  \cap \calE_{h \leq c k}(1,\alpha) \cap \calE_{h \leq c k}	(3,\alpha) ],
	\end{align*}
	and therefore there exists $j_0 \in \{ 0, \pm 2 \}$ such that
	\begin{align*}
	\bbP^{\zeta}  \Big[ I_{j_0} \xlra{h \geq (1-c) k } \tilde I_{j_0} \Big]
	\leq 
	1-\tfrac{1}{6}\bbP^\zeta[\calE_{h \leq c k}(-3,\alpha) \cap \calE_{h \leq c k}(-1,\alpha)  \cap \calE_{h \leq c k}(1,\alpha) \cap \calE_{h \leq c k}(3,\alpha)].
	\end{align*}
	Finally, observe that the boundary conditions $\zeta$ are such that
	\begin{align*}
    	\bbP^{\xi}[I_0 \xleftrightarrow{h \geq (1-c) k } \tilde I_0] 
    	\leq \bbP^{\zeta }[I_j \xleftrightarrow{h \geq (1-c) k } \tilde I_j] \quad \text{ for every $j \in \{-2,0,2\}$}. 
	\end{align*}
	In conclusion,~\eqref{eq:tricky} is indeed implied by~\eqref{eq:tricky2}, and we will focus on proving the latter. 
	
	For $h\in \calS$, the fact that
	$\calE_{h \leq c k}(-3,\alpha)$, $\calE_{h \leq c k}(-1,\alpha)$ and $I_{-2} \xleftrightarrow{h \geq (1-c)k} \tilde I_{-2}$ occur, 
	induces the existence of a leftmost crossing $\gamma_L$ of $h \le ck$ from $I_{-1}$ to $\tilde I_{-1}$. 
	Similarly, there exists a rightmost crossing $\gamma_R$ of $h \le ck$ from $I_1$ to $\tilde I_{1}$.
	
	Let $D$ be the discrete domain made of faces of $\bbZ\times [0,n]$ 
	that are between $\gamma_L$ and $\gamma_R$, or on $\gamma_L$ and $\gamma_R$. 
	Notice that the event $\calS$ and the paths $\gamma_L$ and $\gamma_R$ are measurable in terms of 
	the values of the height function on $D^c \cup \partial D$.
	Moreover, when $\calS$ occurs, all faces on $\gamma_L$ and $\gamma_R$ have height $ck$ or $ck-1$~\footnote{
In this sentence and for the rest of this proof we suppress for the sake of streamlined writing two minor details. First, we omit the integer roundings from $\lfloor ck \rfloor$ when it appears in exact height function values.	Second, as $\gamma_L$ ends on $I_{-1}$ and $\tilde{I}_{-1}$ where the boundary condition is $0,1$ the height value at the endpoints and close to them is not $ck$ or $ck-1$. Due to being leftmost, $\gamma_L$ actually slides directly to the left from both end points, to reach the height $ck$ or $ck-1$ in $I_{-2}$ and $\tilde{I}_{-2}$, respectively, and then connects these ``left-pushed endpoints'' by a curve on which the height indeed is $ck$ or $ck-1$.
	}. See Figure~\ref{fig:tricky quad}.
	
\begin{figure}
	\begin{center}
		\includegraphics[width=0.65\textwidth]{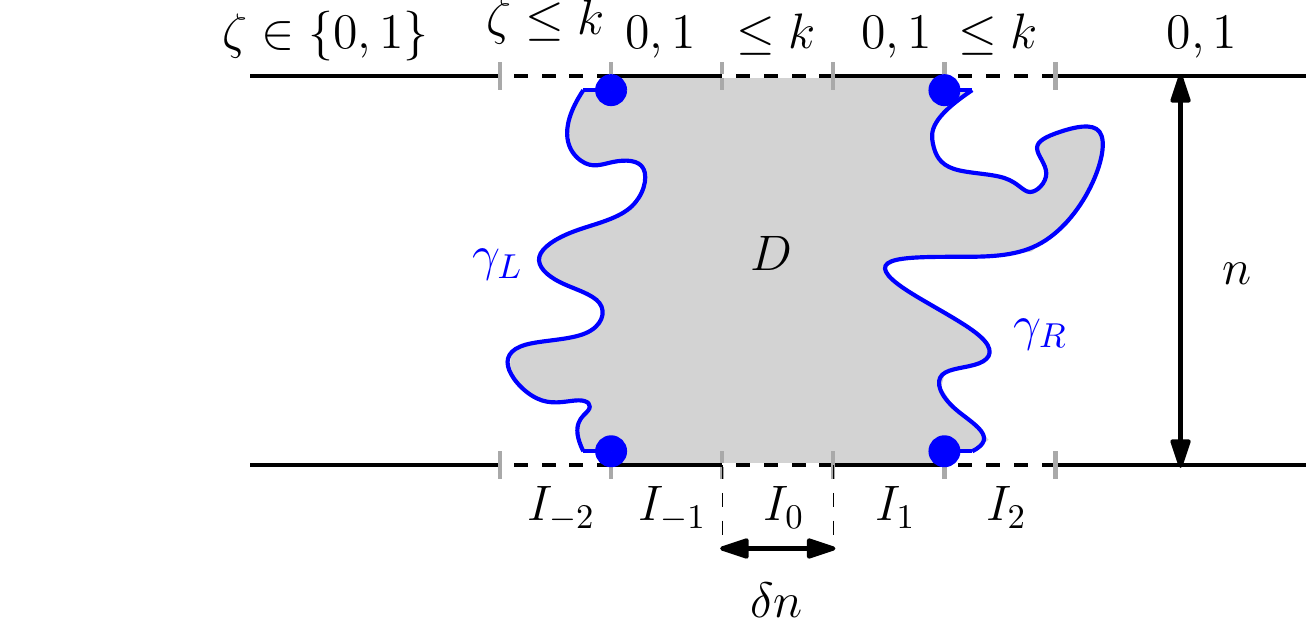}
	\caption{
	\label{fig:tricky quad}
	Illustration for the proof of Lemma~\ref{lem:tricky}. The curves $\gamma_L$ and $\gamma_R$ (blue) and the boundary segments of the strip between them define the quad $D$. Solid-line boundaries represent here boundary conditions $\zeta \in \{ 0,1 \}$, and dashed lines represent maximal extensions of boundary conditions remaining at most $k$.}
	\end{center}
\end{figure}

	Thus, conditionally on a realisation of $\gamma_L$, $\gamma_R$ and $h$ on $D^c \cup \partial D$, 
	the height function in $D$ is distributed according to $\bbP_D^{\chi}$, 
	where $\chi$ is identical to $\zeta$ on the boundary of the strip $ \bbZ \times [0,n]$ and is equal to $ck$ or $ck-1$ on $\gamma_L$ and $\gamma_R$. 
	Let $\calX$ be the set of possible realisations of $(D, \chi)$ such that $\calS$ occurs.
	Then 
	 \begin{align*}
		\bbP^{\zeta}[\calS ]  &= \sum_{(D,\chi)\in \calX}\bbP^{\zeta}[\gamma_L, \gamma_R \text{ bound $D$} ]\qquad \text{ and}  \\
		\bbP^{\zeta}[ \{ I_0 \xlra{h \geq (1-c)k } \tilde I_0 \} \cap \calS ]
		&= \sum_{(D,\chi)\in \calX}\bbP^{\chi}_{D}[I_0 \xlra{h \geq (1-c)k } \tilde I_0 ]\,\bbP^{\zeta}[\gamma_L, \gamma_R \text{ bound $D$}].
	\end{align*}
	To prove~\eqref{eq:tricky2}, it thus suffices to show that for every $(D,\chi)\in \calX$,
	\begin{align}\label{eq:tricky3}
		\bbP^{\chi}_{D}[I_0 \xlra{h \geq (1-c)k } \tilde I_0 ] \leq \tfrac12 .
	\end{align}
	If $\gamma_L$ and $\gamma_R$ intersect, then the left-hand side is equal to $0$ and there is nothing to do, 
	so we restrict ourselves to domains $D$ for which $\gamma_L$ and $\gamma_R$ do not intersect. 
	In the rest of the proof we show~\eqref{eq:tricky3} by distinguishing between the different values of $\alpha$. 
	We only describe the proof for $\alpha=0$ and $\alpha=+$; the proof for $\alpha=-$ is the same as that for $\alpha = +$. 

	\paragraph{Case of $\alpha = 0$:} 
	Since $\calS$ occurs, 
	$\gamma_L$ and $\gamma_R$ are contained between any path of height $h \leq c k$ 
	from $I_{-3}$ to $\tilde I_{-3}$  and from $I_{3}$ to $\tilde I_{3}$. 
	As we are in the case $\alpha = 0$, such paths exist left of $L_{-8}$ and right of $L_9$. 
	Thus $D$ is necessarily contained in a $n\times n$ square $D' \supset [-8\delta n, 9\delta n]\times [0,n]$ (recall that $17 \delta n \leq n$). 
	
Denote by $\{ck, ck-1\}$ the boundary condition on $\partial D$ taking only values $ck$ and $ck-1$, 
and by $\chi \vee \{ck, ck-1\}$ the pointwise maximum of the two; by~\eqref{eq:CBC-h},
\begin{align*}
\bbP^{\chi}_{D}[I_0 \xlra{h \geq (1-c)k} \tilde I_0 ]
\leq
\bbP^{\chi \vee \{ck, ck-1\} }_{D}[I_0 \xlra{h \geq (1-c)k} \tilde I_0 ].
\end{align*}
Note that on the left and right sides of $D$, $\chi$ (and thus also $\chi \vee \{ck, ck-1\}$) takes values $ck$ and $ck-1$.  Let $\chi'$ be the boundary condition for $\heightfcns_{D'}$ on $\partial D'$ which is equal to $ \chi \vee \{ck, ck-1\}$ on $\partial D \cap \partial D'$
	and on the faces of $\partial D'$ left or right of $D$ takes values $ck$ and $ck-1$. 
	By definition, $ck - 1 \leq \chi' \leq k $, and we can apply the boundary pushing of Corollary~\ref{cor:pushing_modified} (recall that $c=1/4$ so $ck - 1 < (1-c)k$) to get
	\begin{align*}
		\bbP^{ \chi \vee \{ck, ck-1\} }_{D}[I_0 \xlra{h \geq (1-c)k} \tilde I_0 ] \leq
		\bbP^{\chi'}_{D'}[I_0 \xlra{h \geq (1-c)k -2 \text{ in $D$}} \tilde I_0 ]
		\leq \bbP^{\chi'}_{D'}[\calV_{h\ge (1-c)k-2} (D')],
	\end{align*}
	where $\calV_{h\ge (1-c)k - 2} (D')$ denotes the existence of a path of $h \geq (1-c)k-2$ crossing $D'$ vertically.
	Recall from Remark~\ref{rem:dual_crossings} that $\calV_{h\ge (1-c)k-2} (D') \subset \calH_{h \leq (1-c)k-3} (D')^c$, 
	where $\calH$ refers to horizontal crossings. 
	Apply now Lemma~\ref{lem:symmetric_domain} to the symmetric domain $D'$, with the boundary conditions 
	 $\lfloor (1-c)k \rfloor - 3 - \chi'$, 
	to conclude that 
	\begin{align*}
		\bbP^{\chi'}_{D'}[\calV_{h \ge (1-c)k - 2 } (D')] 
		\leq 1 - \bbP^{\chi'}_{D'}[\calH_{h \le (1-c)k -3} (D')] 
		\leq \tfrac12,
	\end{align*}
	for all $k \geq 28$ \footnote{{This threshold ensures that the maximal boundary value $\lfloor (1-c)k \rfloor - 3 -( \lfloor ck \rfloor -1)$ of the boundary condition $\lfloor (1-c)k \rfloor - 3 - \chi'$ has larger modulus than the minimal one, namely $\lfloor (1-c)k \rfloor - 3 -k$.}}.
	
\begin{figure}
	\begin{center}
	\includegraphics[width = 0.98\textwidth]{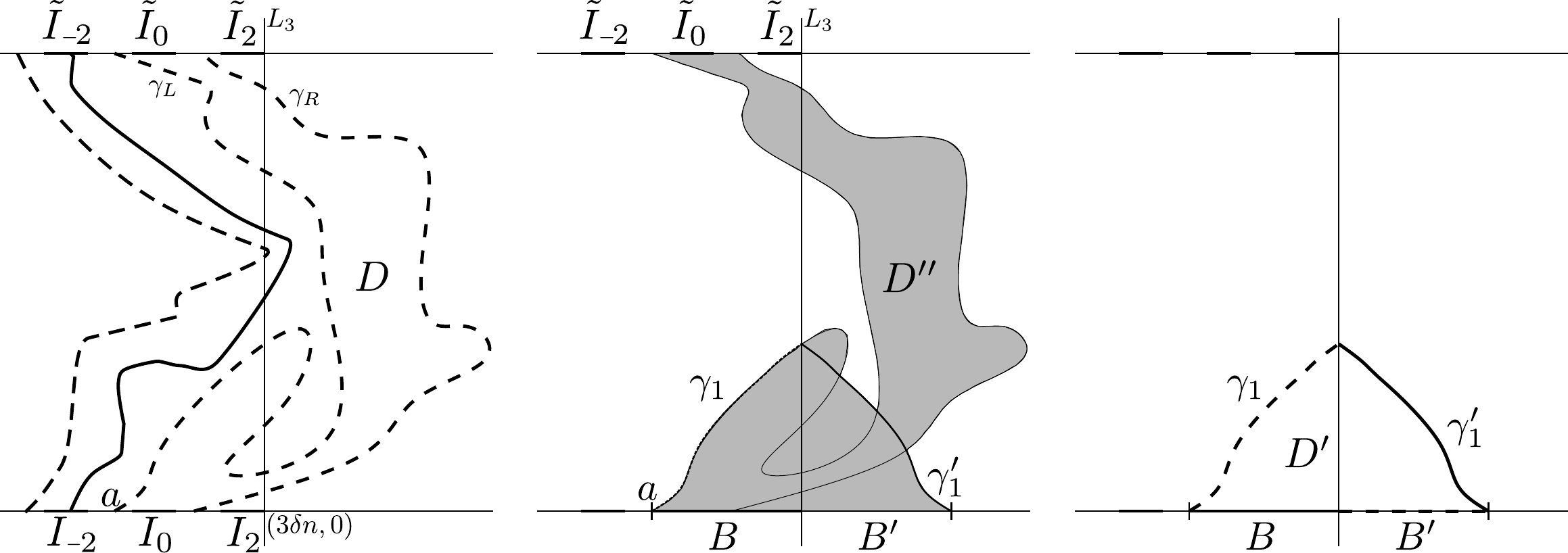}
	\caption{The dashed lines represent boundary conditions $ck, ck-1$, while the thick ones are boundary conditions $k,k-1$ (up to interpolations at the ends of the intervals). The segments at the bottom and top have length $\delta n$. 
	{\em Left:} the domain $D$ formed of faces contained between $\gamma_L$ and $\gamma_R$. 
	{\em Middle:} the domain $D$ is extended into $D''$ by pushing away the (small) boundary conditions $ck,ck-1$. 
	The probability of existence of a crossing from $I_0$ to $\tilde I_0$ of large height increases.
	{\em Right:} the domain $D''$ is shrunk to $D'$ by pulling closer the (large) boundary conditions $k,k-1$. 
	The probability of existence of a crossing from $\gamma_1$ to $B'$ of low height decreases.}
	\label{fig:tricky}
	\end{center}
\end{figure}

\paragraph{Case of $\alpha=+$:} 
	Since $\calE_{h \leq c k}(-3,+)$ and  $I_{-2} \xleftrightarrow{h \geq (1-c)k} \tilde I_{-2}$ occur, $\gamma_L$ necessarily intersects the vertical line $L_{3}$ (see Figure~\ref{fig:tricky}).
	Orient $\gamma_L$ from bottom to top and let $\gamma_1$ be its sub-path up to its first intersection with $L_3$.
	Write $a$ for the starting point of $\gamma_L$ and $B$ for the segment of $\bbZ \times \{0\}$ between $a$ and $(3\delta n,0)$. 
	Let $\sigma$ be the reflection with respect to the line of faces touching $L_3$ on the left ($\sigma$ thus preserves parity); define $\gamma_1' = \sigma( \gamma_1 )$ and $B' = \sigma(B)$.
	Set $D'$ to be the simply connected domain bounded by $B$, $B'$, $\gamma_1'$ and $\gamma_1$;
	let $D'' = D' \cup D$; see Figure~\ref{fig:tricky} for an illustration. 
	
	Let $\chi \vee \{ck, ck-1\}$ be as above and write $\chi''$ for the lowest boundary conditions on $\partial D''$ 
	which are identical to $\chi \vee \{ck, ck-1\}$ on $\partial D'' \cap \partial D$ and $\{ck, ck-1\}$ on the right or left of $D$.
		Applying~\eqref{eq:CBC-h} and~\eqref{eq:pushing_11} we find, 
	\begin{align*}
		\bbP^{\chi}_{D}[I_0 \xlra{h \geq (1-c)k} \tilde I_0 ] 
		\leq 
		\bbP^{\chi \vee \{ck, ck-1\} }_{D}[I_0 \xlra{h \geq (1-c)k} \tilde I_0 ] 
		\leq \bbP^{\chi''}_{D''}[I_0 \xlra{h \geq (1-c)k - 2} \tilde I_0 ].
	\end{align*}
	As stated in Remark~\ref{rem:dual_crossings}, the event $I_0 \xlra{h \geq (1-c)k-2} \tilde I_0$ 
	is incompatible with the event $\gamma_1 \xlra{h \leq (1-c)k-3 \text{ in $D''$}} B'$, so
		\begin{align*}
		\bbP^{\chi''}_{D''}[I_0 \xlra{h \geq (1-c)k - 2} \tilde I_0 ]
		\leq 1-	\bbP^{\chi''}_{D''}[\gamma_1 \xlra{h \leq (1-c)k-3 } B'].	
	\end{align*}
		
Next, write $\chi'$ for the largest boundary conditions on $\partial D'$ which is smaller or equal to $k$
	and equal to $\chi''$ on $\partial D' \cap \partial D''$. 
	Applying~\eqref{eq:pushing_11} to $-h$, we get
	\begin{align*}
	\bbP^{\chi''}_{D''}[\gamma_1 \xlra{h \leq (1-c)k-3 } B'] 
		\geq \bbP^{\chi'}_{D'}[\gamma_1 \xlra{h \leq (1-c)k-5 } B'] .	
	\end{align*}
	Observe that, {by construction, $D'$ is invariant under $\sigma$. Furthermore, $\gamma_1$ and $B'$ are contained in $\partial D' \cap \partial D''$, and therefore $\chi' = \chi'' \in \{ ck, ck-1\}$ on $\gamma_1 \cup B'$.} 
	Apply Lemma~\ref{lem:symmetric_domain} to the boundary condition {$\lfloor (1-c)k \rfloor -5 - \chi'$} to find that 
	\begin{align*}
		 \bbP^{\chi'}_{D'}[\gamma_L \xlra{h \leq (1-c)k-5 \text{ in $D'$}} B'] \geq \tfrac12,			
	\end{align*}
	for $k \geq 44$ \footnote{{This lower bound on $k$ ensures that $\lfloor (1-c)k \rfloor -5 - (\lfloor ck \rfloor-1)$ has larger modulus than $\lfloor (1-c)k \rfloor -5 - k$}.}. 
	The four equations displayed above imply~\eqref{eq:tricky3}.
\end{proof}

\subsection{Long crossings in a strip}\label{sec:long_in_stip}

In our proof of Theorem~\ref{thm:RSW}, the intermediate result below refers to crossings of long rectangles in a strip. We remind the reader that we still omit the integer roundings in~$\lfloor \delta n \rfloor$.

\begin{proposition}\label{prop:RSW3}
	There exist constants $\delta, c,C >0$ such that the following holds. 
	For any $k \geq 1/c$, any $n$, and any $\rho \geq 1$, 
	\begin{align}\label{eq:RSW3}\nonumber
	&\bbP_{\bbZ \times [-n,2n]}^{0,1}\big[\{0\}\times [-n,2n] \xleftrightarrow{h \geq c k \text{ in $\bbZ \times [0,n]$}}\{\rho \delta n\}\times [-n,2n]\big]\\
	&\qquad\qquad\qquad \geq  \big(c\, \bbP_{\bbZ \times [-n,2n]}^{0,1}\big[[0,\delta n]\times \{0\} \xleftrightarrow{h \geq k\text{ in $\bbZ \times [0,n]$}}\bbZ\times \{ n\}\big]\big)^{C \rho}.
	\end{align}
\end{proposition}

Let us prepare for the proof of Proposition~\ref{prop:RSW3} by fixing $\delta>0$ such that Proposition~\ref{prop:RSW1} applies for $3 \delta$,
and the constant $c_0>0$ appearing soon, which is given by the same proposition so that~\eqref{eq:height shift for slit crossings} below holds.
For notation, as in the proof of  Proposition~\ref{prop:RSW1}, we denote
\begin{align*}
    I_j:=[ j \delta n , (j+1) \delta n  ] \times\{ 0\} \quad \text{ and }\quad 
    \tilde I_j:=[ j \delta n , (j+1) \delta n   ] \times\{ n\}.
\end{align*}
Moreover, let $\calB_{h\ge \ell}(j)$ denote the ``bridging'' between $I_{j-1}$ and $I_{j+1}$ in $\bbZ \times [0,n]$:
\begin{align*}
	\calB_{ h \ge \ell}(j)& := \big\{ I_{j-1} \xlra{h \geq \ell \text{ in $\bbZ \times [0,n]$}} I_{j+1} \big\},
\end{align*}
and define $\calB_{ \vert h \vert \ge \ell}(j)$ analogously.
The following lemma is the key step in the proof of Proposition~\ref{prop:RSW3}.

\begin{lemma}\label{lem:bridges}
	With the notation above, 
	\begin{align}\label{eq:bridges}
		\bbP_{\bbZ \times [-n,2n]}^{0,1}[\calB_{h\ge c_0 k}(0)]
		\geq \tfrac{c_0}8 \big(\bbP_{\bbZ \times [-n,2n]}^{0,1}\big[I_0 \xleftrightarrow{h \geq k\text{ in $\bbZ \times [0,n]$}} \tilde I_0\big]\big)^{2}.
	\end{align}
\end{lemma}

The idea behind the proof of this lemma is simple:
condition on the left- and right-most crossings of height greater than $k$ from $I_{-1}$ to $\tilde I_{-1}$ 
and from $I_{1}$ to $\tilde I_{1}$, respectively, 
then use Proposition~\ref{prop:RSW1} to connect these two paths by a path of height at least $c k$. 
However, there are problems arising when pushing away boundary conditions; 
to overcome these we will need to use the FKG property for the absolute value of the height function~\eqref{eq:FKG-|h|}.

\begin{proof}
	We start by transferring the question to crossings in the absolute value of the height function. 
	First, by the comparison between boundary conditions for $h$, we have 
	\begin{align*}
		\bbP_{\bbZ \times [-n,2n]}^{0,1}[\calB_{h\ge c_0k}(0)]  
		\geq \tfrac12 \bbP_{\bbZ \times [-n,2n]}^{0,1}[\calB_{|h|\ge c_0k} (0)] .
	\end{align*}
	Now, if we define
	\begin{align*}
	\calT & := \{I_{-1}\xlra{|h| \geq k\text{ in $\bbZ \times [0,n]$}} \tilde I_{-1}\}
	\cap \{I_{1}\xlra{|h| \geq k\text{ in $\bbZ \times [0,n]$}} \tilde I_{1}\},
	\end{align*}
 	we have, due to the inclusion of events and the FKG inequality
	\begin{align*}
		\bbP_{\bbZ \times [-n,2n]}^{0,1}[\calT] 
		\geq \big(\bbP_{\bbZ \times [-n,2n]}^{0,1}\big[I_{0}\xlra{h \geq k\text{ in $\bbZ \times [0,n]$}} \tilde I_{0}\big]\big)^{2}.
	\end{align*}	
	It thus suffices to prove
	\begin{align}
	\label{eq:bridges2}
	\bbP_{\bbZ \times [-n,2n]}^{0,1} [ \calB_{|h|\ge c_0k} (0) \; \vert \; \calT ] \geq c_0 / 4.
	\end{align}
	
	When $\calT$ occurs, write $\gamma_L$ for the left-most path in $\bbZ \times [0,n]$ with $|h| \geq k$ connecting $I_{-1}$ to $\tilde I_{-1}$. 
	Similarly, let $\gamma_R$ be the right-most path in $\bbZ \times [0,n]$  with $|h| \geq k$	connecting $I_{1}$ to $\tilde I_{1}$. (By finite energy, such paths exist almost surely.)
	Write $D$ for the discrete sub-domain $\bbZ \times [0,n]$ of faces between or on the paths $\gamma_L$ and $\gamma_R$. (See Figure~\ref{fig:bridging domain} for an illustration.)	Notice that $\gamma_L$ and $\gamma_R$ are measurable in terms of the absolute value of the height function on 
	$D^{\text{out}} = (\bbZ \times [-n,2n] \setminus D) \cup \partial D$. Equip $D$ with the structure of a quad with $\gamma_L$ and $\gamma_R$ being two sides and the remaining two contained in $\bbZ \times [0,1]$ and $\bbZ \times [n-1,n]$, respectively, and denote as earlier $\calV(D)$ and $\calH(D)$ for vertical and horizontal crossing events, respectively. 
	
	\begin{figure}
	\begin{center}
	\includegraphics[width=0.5\textwidth]{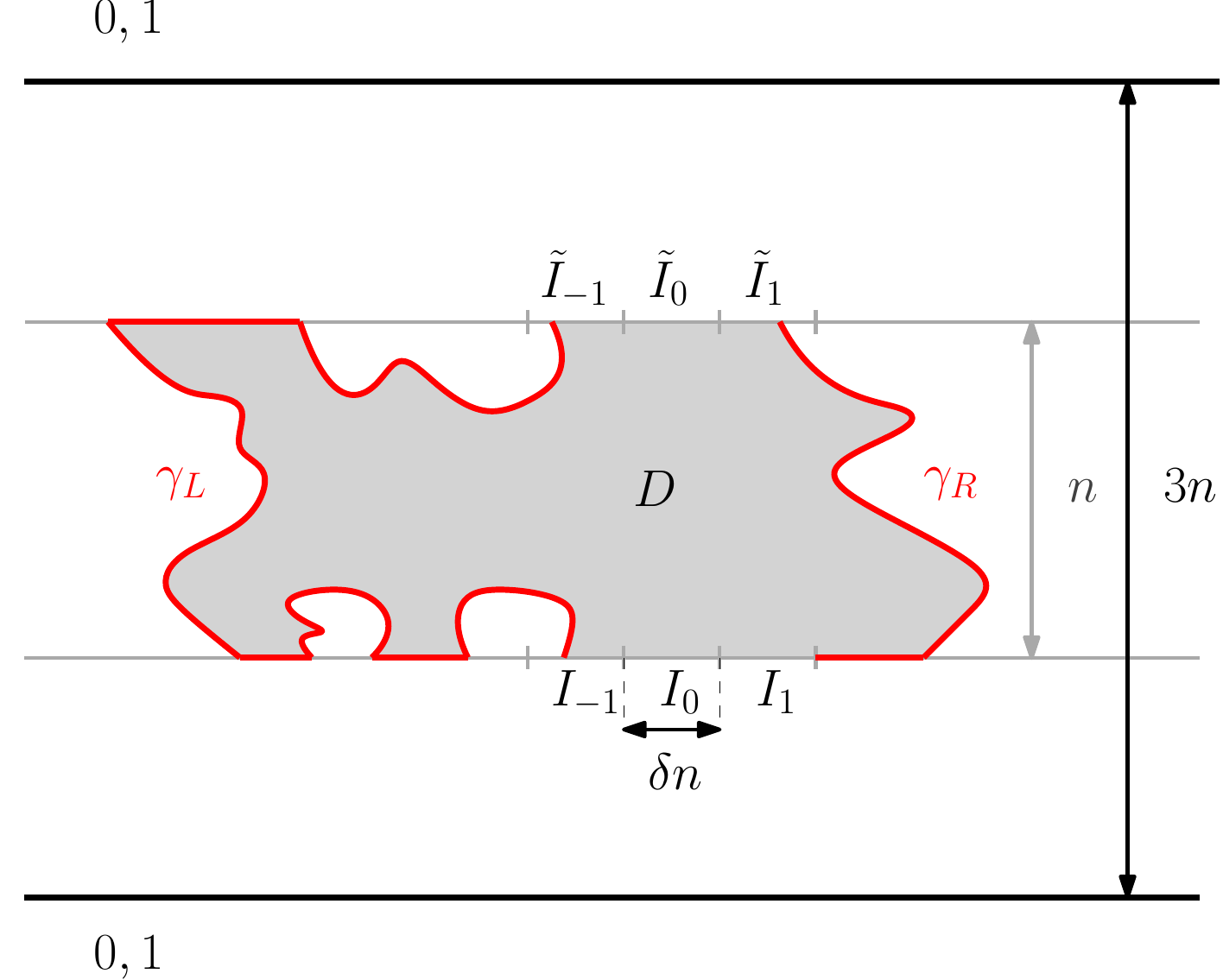}
	\caption{An illustration for the proof of Lemma~\ref{lem:bridges}.
	The red paths $\gamma_L$ and $\gamma_R$ have $|h|\geq k$ and are measurable in terms of the value of $|h|$ 
	on them and on the faces to their left and right, respectively.}
	\label{fig:bridging domain}	
	\end{center}	
	\end{figure}

	By inclusion of events, the~\eqref{eq:CBC-|h|} inequality, and inclusion again, we have
		\begin{align}
		\nonumber
		\bbP_{\bbZ \times [-n,2n]}^{0,1}\big[\calB_{|h|\ge c_0k} (0) \,\big|\,|h| \text{ on $D^{\text{out}}$} \big] 
		&\geq \bbP_{\bbZ \times [-n,2n]}^{0,1}\big[\calH_{|h| \geq c_0k } (D) \,\big|\,|h| \text{ on $D^{\text{out}}$} \big] \\
		&\geq \bbP_{\bbZ \times [-n,2n]}^{0,1}\big[ \calH_{|h| \geq c_0k } (D)  \,\big|\,|h| = \zeta \text{ on $D^{\text{out}}$} \big]\nonumber
		\\
		&\geq \bbP_{\bbZ \times [-n,2n]}^{0,1}\big[ \calH_{h \geq c_0k } (D)  \,\big|\,|h| = \zeta \text{ on $D^{\text{out}}$} \big],\nonumber
	\end{align}
	where $\zeta \geq 0$ is the minimal configuration on $D^c$ which is equal to $k$ and $k+1$ on~$\gamma_L$ and~$\gamma_R$.
	
	Now, it holds true that  
	conditionally on the fact that $|h| = \zeta$ on $D^{\text{out}}$, 
	there is a probability at least $1/4$ that $h$ is equal to $k$ and $k+1$ on both $\gamma_L$ and $\gamma_R$~\footnote{
This follows from the observation in Section~\ref{subsubsec: Ising tools} that given the absolute value, the signs of a height function are given by a (ferromagnetic) Ising model. The positive association of the Ising model and the positive boundary signs in $\bbP_{\bbZ \times [-n,2n]}^{0,1}$ thus make two plus signs the most probable one out of the four possible sign combination on the curves $\gamma_L$ and $\gamma_R$.
	}.
	In that case, the boundary condition for $h$ on $\partial D$, induced by $|h| = \zeta$, dominates the minimal boundary condition $\xi$ on $\partial D$ with $\xi \geq -1$ and which is equal to $k$ and $k+1$ on $\gamma_L$ and $\gamma_R$. 
	Thus, 
	\begin{align}
	\nonumber
		\bbP_{\bbZ \times [-n,2n]}^{0,1}\big[ \calH_{h \geq c_0k } (D)  \,\big|\,|h| = \zeta \text{ on $D^{\text{out}}$} \big]
		&\geq \tfrac14 \bbP_{D}^{\xi}\big[ \calH_{h \geq c_0k } (D)  \big]\\
		& =  \tfrac14 \bbP_{D}^{k - \xi}\big[ \calH_{h \leq (1-c_0)k } (D) \big]\nonumber\\
		& \geq \tfrac14\big( 1-\bbP_{D}^{k - \xi}\big[ \calV_{h \geq (1-c_0)k} (D) \big]\big),\nonumber
	\end{align}
	where the last line is due to Remark~\ref{rem:dual_crossings}.

	Notice now that the boundary conditions $k - \xi$ are bounded above by $k+1$ and are equal to $0$ and $-1$ on $\gamma_R$ and $\gamma_L$.
	Using boundary pushing (Proposition~\ref{prop:pushing_1}), we may now compare to $\bbP_{\bbZ\times [0,n]}^{\xi'}$ where $\xi'$ is the largest boundary condition smaller than $k+1$ and which is equal to $0$ and $1$ on $(\bbZ \setminus [-\delta n, 2\delta n])\times \{0,n\}$. We obtain
	\begin{align*}
		\bbP_{D}^{k - \xi}\big[ \calV_{h \geq (1-c_0)k} (D) \big]
		& \leq 
		\bbP_{\bbZ\times [0,n]}^{\xi'}\big[ \calV_{h \geq (1-c_0)k} (D) \big] \\
			& \leq 
				\bbP_{\bbZ\times [0,n]}^{\xi'}\big[[-\delta n, 2\delta n]\times \{0\} \xlra{h \geq (1-c_0)k } [-\delta n, 2\delta n]\times \{n\}\big],
	\end{align*}
	where the latter inequality used the fact that the bottom and top boundary segments of the quad
	$ D $ {are contained in the intervals $ [-\delta n, 2\delta n]\times \{0\}$ and $ [-\delta n, 2\delta n]\times \{n\}$, respectively.}
	
	Finally, Proposition~\ref{prop:RSW1} proves that
	\begin{align}
	\label{eq:height shift for slit crossings}
	\bbP_{\bbZ\times [0,n]}^{\xi'}\big[[-\delta n, 2\delta n]\times \{0\} \xlra{h \geq (1-c_0)k } [-\delta n, 2\delta n]\times \{n\}\big]
	\leq 1-c_0
	\end{align}
	(as such, with the boundary conditions $\xi'$, Proposition~\ref{prop:RSW1} addresses crossings of $h \geq (1-c_0)(k+1)$ but the choice of $c_0$ allows us to ignore this difference). 
	
	The four previously displayed inequalities yield for any $D^{\text{out}}$
	\begin{align*}
		\bbP_{\bbZ \times [-n,2n]}^{0,1}\big[\calB_{|h|\ge c_0k} \,\big|\,|h| \text{ on $D^{\text{out}}$} \big]  
		\geq \tfrac14 c_0.
	\end{align*}
This finishes the proof of~\eqref{eq:bridges2} and the entire lemma.
\end{proof}

\begin{proof}[Proposition ~\ref{prop:RSW3}]
It suffices to give the proof when $\rho$ is an integer.
	When the events $\calB_{h\ge c_0k}(j)$ with $0\le j<\rho$ occur, 
	they induce the existence of a path from $\{0\}\times [-n,2n]$ to $\{\rho \delta n\}\times [-n,2n]$
	of height at least $c_0k$. Moreover, this path is contained in the central strip $\bbZ \times [0,n]$. 
	Due to the FKG inequality, the invariance of $\bbP_{\bbZ \times [-n,2n]}^{0,1}$ under  horizontal translations and Lemma~\ref{lem:bridges},
	we find
	\begin{align}
		&\bbP_{\bbZ \times [-n,2n]}^{0,1}\big[\{0\}\times [-n,2n] \xlra{h \geq c_0k \text{ in $\bbZ \times [0,n]$}}\{\rho \delta n\}\times [-n,2n]\big]\nonumber\\
		&\hspace{4cm}\geq \bbP_{\bbZ \times [-n,2n]}^{0,1}[\calB_{h\ge c_0k}(j)]^{\rho} \nonumber\\
		&\hspace{4cm}\geq 
		\Big[ \tfrac{c_0}8 \bbP_{\bbZ \times [-n,2n]}^{0,1}\big[I_0 \xlra{h \geq k\text{ in $\bbZ \times [0,n]$}} \tilde I_0\big]^{2}
\Big]^{\rho}.
\label{eq:aalade}
	\end{align}

We now claim that 
	\begin{align}
	\label{eq:targeting a vertical crossing to opposite slit}
		\bbP_{\bbZ \times [-n,2n]}^{0,1}\big[I_0 \xlra{h \geq k\text{ in $\bbZ \times [0,n]$}} \tilde I_0\big]
		\geq \Big(\tfrac13\bbP_{\bbZ \times [-n,2n]}^{0,1}\big[I_0 \xlra{h \geq k\text{ in $\bbZ \times [0,n]$}} \bbZ \times \{n\}\big]\Big)^2,
	\end{align}
	which together with~\eqref{eq:aalade} completes the proof.
	To prove~\eqref{eq:targeting a vertical crossing to opposite slit}, observe that if the event on the right-hand side occurs, then $I_0$ is either connected by $h \geq k$ to either to $\tilde I_0$, to $(-\infty, 0] \times \{n\}$ or to $[\delta n, \infty) \times \{n\}$.
	It follows that at least one of these connections has probability $\tfrac13\bbP_{\bbZ \times [-n,2n]}^{0,1}\big[I_0 \xlra{h \geq k\text{ in $\bbZ \times [0,n]$}} \bbZ \times \{n\}\big]$; if it is the connection to $\tilde I_0$,~\eqref{eq:targeting a vertical crossing to opposite slit} follows immediately, so assume next that it is the connection to $[\delta n, \infty) \times \{n\}$ (the third case is symmetric).

	Now, if $I_0$ is connected to $[\delta n, \infty) \times \{n\}$ and 
	$\tilde I_0$ to $[\delta n, \infty) \times \{0\}$
	by paths of height at least $k$  simultaneously, 
	then $I_0$ and $\tilde I_0$ are connected to each other by such paths. 
	Thus, using the vertical symmetry and the FKG inequality, 
	\begin{align*}
		\bbP_{\bbZ \times [-n,2n]}^{0,1}\big[I_0 \xlra{h \geq k\text{ in $\bbZ \times [0,n]$}} \tilde I_0\big]
		\geq \bbP_{\bbZ \times [-n,2n]}^{0,1}\big[I_0 \xlra{h \geq k\text{ in $\bbZ \times [0,n]$}} [\delta n, \infty) \times \{n\}\big]^2,
	\end{align*}
and~\eqref{eq:targeting a vertical crossing to opposite slit} follows from the assumption of the previous paragraph.
\end{proof}

\subsection{From strip to annulus}\label{sec:strip_to_annulus}

In this section we conclude the proof of Theorem~\ref{thm:RSW}. The fairly classic argument consists in combining different crossings in rectangles and using the proper comparison between boundary conditions.

\begin{proof}[Theorem ~\ref{thm:RSW}]
	Write $\calH_{h \geq \ell}(R)$ for the event that the rectangle $R := [0,18n] \times [-n,2n]$ 
	contains a horizontal crossing of height at least $\ell$, that is a path of $h\ge \ell$ from $\{0\}\times [-n,2n]$ to $\{18 n\}\times [-n,2n]$. 
	By Proposition~\ref{prop:RSW3}, we may fix constants $c, C, \delta$ such that 
	\begin{align*}
		\bbP_{\bbZ \times [-n,2n]}^{0,1}[\calH_{h \geq c k}(R)]
 \geq  c \,\bbP_{\bbZ \times [-n,2n]}^{0,1}\big[[0,\delta n]\times \{0\} \xlra{h \geq k\text{ in $\bbZ \times [0,n]$}}\bbZ\times \{ n\}\big]^{C}.
	\end{align*}
	Let $S_L = [-3n, 0 ] \times [-n,2n]$ and $S_R=  [18n, 21n ] \times [-n,2n]$ be the two squares to the left and right of $R$, respectively. 
	Write $\calV_{h \leq \ell}(S_L)$ for the event that there exists a path from the top $[-3n, 0 ] \times \{2n\}$ 
	to the bottom $[-3n, 0 ] \times \{-n\}$ of $S_L$ formed of faces with height at most $\ell$. The same notation applies to $S_R$. See Figure~\ref{fig:RSW from strip to rectangle}.
	
\begin{figure}
\begin{center}
\includegraphics[width=0.8\textwidth]{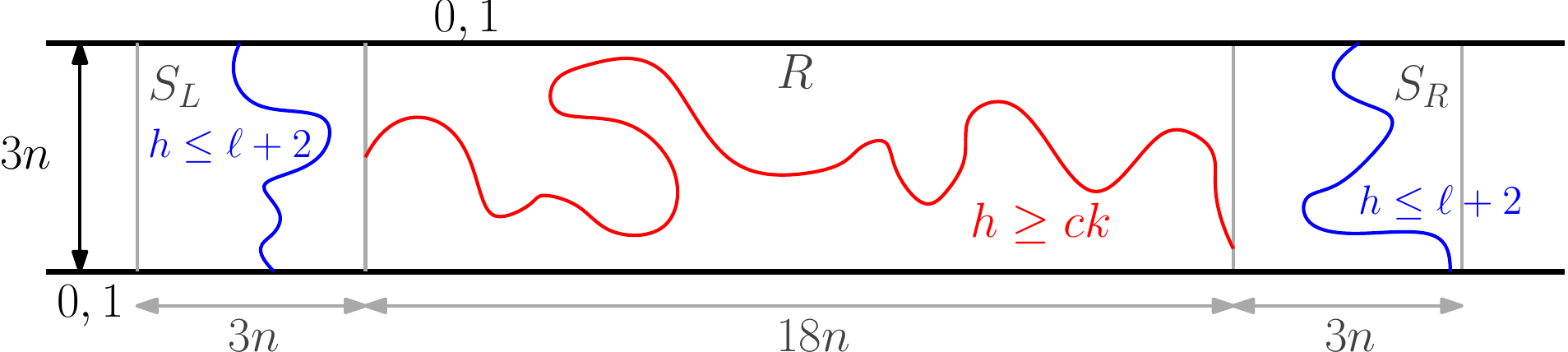}
\caption{The last step in the proof of Theorem~\ref{thm:RSW} is, {roughly speaking,} based on the Bayes formula for the events of the horizontal crossing $\calH_{h \geq c k}(R)$ and the two vertical crossings {$ \calV_{h \leq \ell+2}(S_L) \cap \calV_{h \leq \ell+2}(S_R)$.}}
\label{fig:RSW from strip to rectangle}
\end{center}
\end{figure}	
	
	When $\calH_{h \geq c k}(R)$ occurs, let $\Gamma$ be the lowest path connecting the left and right sides of $R$ which is of height greater or equal to $c k$. 
	Notice that $\Gamma$ may be explored by revealing a random set of faces $F\subset R$, all of whose heights are at most $ck+1$ (here and below, we omit integer roundings of $\lceil ck \rceil$ and treat $ck$ as an integer). 
	 Denote also $\ell = \lceil (c k + 1)/2 \rceil $. We have for any possible realisation $\gamma$ of $\Gamma$, using~\eqref{eq:pushing_11} for $-h$,  
	\begin{align*}
		\bbP_{\bbZ \times [-n,2n]}^{0,1}[\calV_{h \leq \ell + 2}(S_L) \cap \calV_{h \leq \ell + 2}(S_R)|\,\Gamma = \gamma]
		\geq \bbP_{S_L}^{\xi_L}[\calV_{h \leq \ell }(S_L)]\bbP_{S_R}^{\xi_R}[ \calV_{h \leq \ell }(S_R)],
	\end{align*}
	where $\xi_L$ and $\xi_R$ are the largest boundary conditions on $\partial S_L$ and $\partial S_R$, respectively, that are everywhere at most $c k + 1$ and equal to $0$ and $1$ on $\bbZ \times \{-n,3n\}$.
	Now, due to Lemma~\ref{lem:symmetric_domain}, each of the probabilities of the right-hand side above is at least $1/2$. 
	In conclusion, 
	\begin{align*}
		&\bbP_{\bbZ \times [-n,2n]}^{0,1}[\calH_{h \geq c k}(R) \,|\, \calV_{h \leq \ell + 2}(S_L) \cap \calV_{h \leq \ell+ 2}(S_R)]\\		
		& \qquad \geq \bbP_{\bbZ \times [-n,2n]}^{0,1}[\calH_{h \geq c k}(R) \cap \calV_{h \leq \ell + 2}(S_L) \cap \calV_{h \leq \ell + 2}(S_R)]\\
	&\qquad \geq  \tfrac{c}4 \bbP_{\bbZ \times [-n,2n]}^{0,1}\big[[0,\delta n]\times \{0\} \xlra{h \geq k \text{ in $\bbZ \times [0,n]$}}\bbZ\times \{ n\}\big]^{C}.
	\end{align*}
	
	When $\calV_{h \leq \ell + 2}(S_L) \cap \calV_{h \leq \ell + 2}(S_R)$ occurs, 
	consider the discrete domain $D$ on and between the vertical crossings of height at most $\ell + 2$
	that are left-most in $S_L$ and right-most in $S_R$, respectively. By the spatial Markov property, the conditional measure on the left-hand side above can be seen as a convex combination of measures on such domains $D$, with boundary conditions which are at most $\ell+2$. 
	By the previous display and the~\eqref{eq:CBC-h} inequality, we conclude the existence of such a domain $D_0$ with 
	$[0,18n] \times [-n,2n] \subset D_0 \subset [-3n,21n] \times [-n,2n]$ such that 
	\begin{align}\label{eq:long_rect}
		\bbP_{D_0}^{ \ell+1, \ell +2}[\calH_{h \geq c k}(R)]  
		\geq  \tfrac{c}4 \bbP_{\bbZ \times [-n,2n]}^{0,1}\big[[0,\delta n]\times \{0\} \xlra{h \geq k\text{ in $\bbZ \times [0,n]$}}\bbZ\times \{ n\}\big]^{C}.
	\end{align}
	
	Finally, consider the rectangle $R_N = [-9n,9n] \times [6n, 9n]$ and its rotations $R_W$, $R_S$ and $R_E$ around the origin by $\tfrac\pi2$, $\pi$ and $\tfrac{3\pi}2$, respectively. 
	Note that $R_N$ is a translate of the rectangle $R$ considered above. 
	 By~\eqref{eq:pushing_11}, we deduce that 
	\begin{align*}
		\bbP_{\Lambda_{12n}}^{0,1}[\calH_{h \geq ck - \ell - 3}(R_N)]
		\geq \tfrac12 \bbP_{D_0}^{0,1}[\calH_{h \geq ck - \ell -1}(R)]
		= \tfrac12 \bbP_{D_0}^{\ell +1, \ell +2}[\calH_{h \geq c k}(R)].
	\end{align*}
	By rotational invariance, the same lower bound holds for probabilities of crossing $R_W$, $R_S$ and $R_E$ in the ``long'' direction.
	If all these crossing events occur simultaneously, 
	then $\Lambda_{9n}\setminus \Lambda_{6n}$ contains a circuit of height at least $ck - \ell - 3 \geq ck/2 - 6$, 
	and thus $\calO_{h\ge ck/2 - 6}( 6 n ,12 n)$ occurs. 
	Applying the FKG inequality, we find 
	\begin{align*}
		\bbP_{\Lambda_{12n}}^{0,1}[\calO_{h\ge ck/2 -6}( 6 n ,12 n)]
		\geq \big(\tfrac{c}{8}\big)^4 \, \bbP_{\bbZ \times [-n,2n]}^{0,1}\big[[0,\delta n]\times \{0\} \xlra{h \geq k\text{ in $\bbZ \times [0,n]$}}\bbZ\times \{ n\}\big]^{4C}.
	\end{align*}
	The above implies~\eqref{eq:RSW} after adjustment of the constants $c, C$.
\end{proof}

\section{From free energy to circuit probability estimate}\label{sec:6V-cylinder}\label{sec:4}\label{sec:Annulus}

In this section, let $N$ be even and let $\bbP_{\bbO_{N,M}}^{(0)}$ denote the six-vertex measure on the cylinder graph $\bbO_{N,M}$ conditioned on the event that each row of $N$ faces around $\bbO_{N,M}$ is crossed by as many up arrows as down arrows. Under $\bbP_{\bbO_{N,M}}^{(0)}$, each six-vertex configuration defines a height function on the cylindrical dual graph which is unique up to additive constant. When describing events in terms of height function, we will mean that the associated equivalence class of height functions contains a representative having the property of interest.

\subsection{A probabilistic interpretation of free energy increments}

 For $k,n \geq 1$ and a set $S=\{s_0,\dots,s_{2n-1}\}$ of $2n$ faces on the bottom of $\bbO_{N,M}$ (indexed from left to right), let $\calA(S,n,k)$ be the event that for each $0\le i<2n$, there exists a vertical $\times$-crossing of the cylinder, starting at $s_i$, and on which $h=0$ if $i$ is even, and $h=k$ if $i$ is odd. The core of this section is the proof of the following result (recall the definition of the free energy $f_{\mathbf c}$ from Theorem~\ref{thm:Bethe}).
\begin{proposition}
\label{prop: Mano}\label{lem: anchored clusters}
For every $\alpha\in (0,1/2)$ and $k\ge1$, for $n = \lfloor  \lceil \alpha N \rceil / k \rfloor$ we have 
\begin{align*}
	\liminf_{N \to \infty} \liminf_{M \to \infty} \tfrac{1}{NM} \log \max_{S}\bbP^{(0)}_{\bbO_{N,M}} [ \calA(S,n,k) ]
	\geq f_{\mathbf c}(\alpha) - f_{\mathbf c}(0),
\end{align*}
where the maximum is over sets $S$ of $2n$ faces on the bottom of $\bbO_{N,M}$.
\end{proposition}

Relating the probability of the events $\calA(S,n,k)$ to $f_\mathbf c$ will be done in two steps. We start by relating the free energy to the probability of the event $\calB(L)$  that  $h$ contains two vertical $\times$-crossings of $h=0$ and $h=L$ respectively.
\begin{lemma}
\label{lem: Mano}
For every $\alpha\in (0,1/2)$, we have
\begin{align*}
\liminf_{N \to \infty} \liminf_{M \to \infty} \tfrac{1}{NM} \log 
\PRcyl{0}{N}{M} [ \calB(\lceil\alpha N\rceil) ]
\geq  f_{\mathbf c}( \alpha ) - f_{\mathbf c}( 0 ).
\end{align*}
\end{lemma}

\begin{proof}In what follows, set $L=\lceil \alpha N \rceil$. The strategy of the proof is to construct a map 
\[
\mathbf T: \Omega^{(2L)}_{6V}(\bbO_{N,M}) \longrightarrow \Omega^{(0)}_{6V}(\bbO_{N,M}) \cap \calB(L)
\] such that 
\begin{itemize}[noitemsep,nolistsep]
\item[(i)] for any $\omega \in \Omega^{(2L)}_{6V}(\bbO_{N,M}) $,  we have
$W_{\rm 6V}( \mathbf T (\omega) ) \geq \mathbf c^{-2M/\alpha} W_{\rm 6V}(\omega)$,
\item[(ii)]  for any $\omega' \in \Omega^{(0)}_{6V}(\bbO_{N,M}) \cap \calB(L) $, the number of preimages $|\mathbf T^{-1}(\{\omega'\})|$  is bounded by $ N^2 2^{2M / \alpha}$.
\end{itemize}
Assuming for a moment that such a map $\mathbf T$ is constructed and using the definition of the free energy $f_\mathbf c$ in Theorem~\ref{thm:Bethe}, we find
\begin{align*} 
\sum_{ \omega' \in \calB(L) } W_{\rm 6V}( \omega ' )& \stackrel{(ii)}\geq 
\!\!\!\!\!\!\!\sum_{ \omega \in \Omega_{\rm 6V}^{( 2 L)}(\bbO_{N,M})} \!\!\!\!\frac{W_{\rm 6V}( \mathbf T ( \omega ) )}{N^2 2^{2M / \alpha}} \stackrel{(i)}\geq 
 \frac{Z^{(2L)}_{N,M}\mathbf c^{-2M/\alpha}}{N^2 2^{2M / \alpha}}
=\exp[ f_\mathbf c(\alpha) MN(1+o(1))],
\end{align*}
where $o(1)$ denotes a quantity tending to 0 as $M$ and then $N$ tend to infinity. The claim thus
follows by using the definition of the free energy again to give
\[
Z_{N,M}^{(0)}=\exp[ f_\mathbf c(0)MN(1+o(1))].
\]

\begin{figure}
	\begin{center}
	\includegraphics[width=0.87\textwidth]{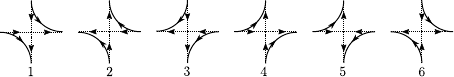}
	\caption{A deterministic map from local 6-vertex configuration to oriented loops and paths. Vertices of types 5 and 6 could be split into non-crossing paths in two ways, but we always choose the left-turning splitting.}
	\label{fig:loops}
	\end{center}
\end{figure}

We therefore turn to the construction of $\mathbf T$ (see Figures~\ref{fig:loops}--\ref{fig:loop reversal}). Consider $\omega \in \Omega^{(2L)}_{6V}(\bbO_{N,M})$.
Define the associated configuration $\overrightarrow{\omega}$ of fully-packed, noncrossing oriented loops and paths on $\bbO_{N,M}$, the paths starting and ending at the bottom or top of the cylinder: $\overrightarrow{\omega}$ is obtained from $\omega$ by splitting the arrows at each vertex into noncrossing loop/path segments. This splitting done so that $\overrightarrow{\omega}$ is a deterministic function of $\omega$ (there is only one noncrossing way to split type 1--4 vertices, while for type 5--6 vertices that could be split into two left or two right-turns, we fix an arbitrary rule, say for definiteness the left-turning splitting depicted in Figure~\ref{fig:loops}). 
Note that $\overrightarrow{\omega}$ must contain at least $2L$ paths between the bottom and the top of the cylinder, and that among all such paths, there are exactly $2L$ more that are oriented upward than downward.

 Let $\gamma_1, \ldots, \gamma_{2L}$ be upward vertical crossing paths of $\overrightarrow{\omega}$ (indexes running from left to right) such that for $1 \leq i \leq 2L- 1$, the connected component of $\bbO_{N,M} \setminus (\gamma_i \cup \gamma_{i+ 1})$ to the right of $\gamma_i$ has an equal number of up and down vertical directed paths of $\overrightarrow{\omega}$. It is not hard to check that such crossings $\gamma_1, \ldots, \gamma_{2 L}$ exist.
Such a family of paths $\gamma_1, \ldots, \gamma_{2L}$ may not be unique, so in order for $\mathbf T$ to be well-defined, we again fix some arbitrary deterministic way to choose them.

The six-vertex configurations on $\bbO_{N,M} $ have $\leq 2NM$ oriented edges and  $\gamma_1, \ldots, \gamma_{2 L}$ are edge-disjoint, so for some $1 \leq i \leq L$, we must have
\begin{align*}
\len (\gamma_i ) + \len (\gamma_{L+i} ) \leq 2MN/L\le 2M / \alpha.
\end{align*}
Let $i^*$ be the integer $i$ minimizing the left-hand side above. We finally define $\mathbf T(\omega)$ to be the six-vertex configuration obtained by reversing the arrows of $\omega$ that are either on the path $\gamma_{i^*}$ or in the connected component $C$ of $\bbO_{N,M}  \setminus (\gamma_{i^*} \cup \gamma_{L+i^*})$ to the right of $\gamma_{i^*}$\footnote{This can be seen as reversing some loops and paths of $\overrightarrow{\omega}$, which directly implies that $\mathbf T(\omega)$ indeed satisfies the ice rule.}.

\begin{figure}
	\begin{center}
	\includegraphics[width=.75\textwidth]{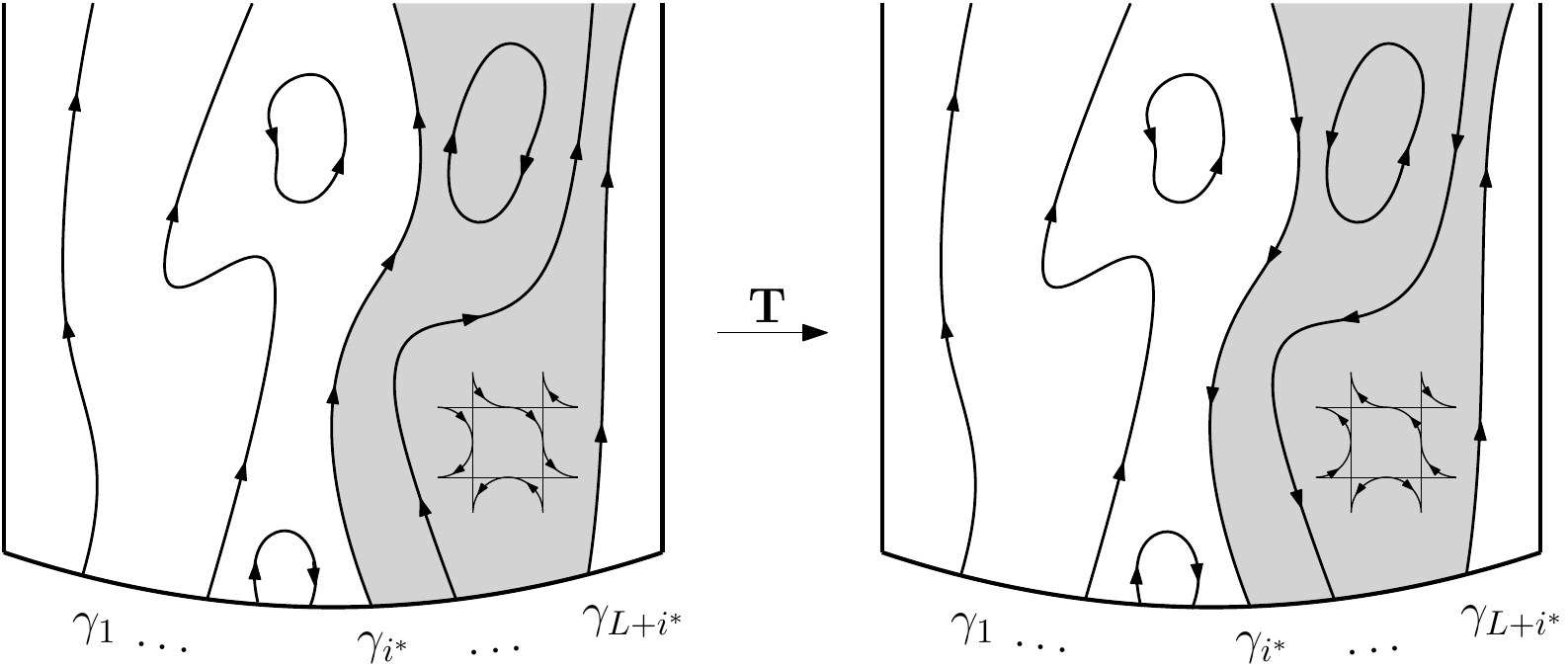}
	\caption{The construction of $\mathbf T$ by reversing the orientation of all the loops and paths (and hence the arrows of 6-vertex configuration) that are between the paths $\gamma_{i^*} $ and $ \gamma_{L+i^*}$ (grey region) or on the former.}
	\label{fig:loop reversal}
	\end{center}
\end{figure}

We now verify that $\mathbf T$ sends any configuration in $\Omega^{(2L)}_{6V}(\bbO_{N,M})$ to $ \Omega^{(0)}_{6V}(\bbO_{N,M}) \cap \calB(L)$, and that it has the desired properties (i) and (ii):
\begin{itemize}
\item 
$\mathbf T(\omega) \in \Omega^{(0)}_{6V}(\bbO_{N,M}) \cap \calB(L)$: look at the directed loops and paths of $\overrightarrow{\omega}$, after the reversal performed by $\mathbf T$.
Among the paths between the top and bottom of $\bbO_{N,M}$, on $\bbO_{N,M}  \setminus (C \cup \gamma_{i^*})$ there are $L$ more upward than downward paths, and on $(C \cup \gamma_{i^*})$ there are $L$ more downward than upward paths. It follows that $\mathbf T(\omega)$ has as many upward as downward paths, so $\mathbf T(\omega) \in \Omega^{(0)}_{6V}(\bbO_{N,M})$. The refinement $\mathbf T(\omega) \in \Omega^{(0)}_{6V}(\bbO_{N,M}) \cap \calB(L)$ follows from the same argument, as the directed loops and paths are level lines of the height function of $\mathbf T(\omega) $, and vertical $\times$-paths of constant height are formed by the faces on both sides of each path  $\gamma_1, \ldots, \gamma_{2 L}$. 

\item Property (i): when changing from $\omega$ to $\mathbf T(\omega)$, only the vertices on the paths $\gamma_{i^*}$ and $ \gamma_{L+i^*}$ may change weight. There are at most $2M/\alpha$ such vertices, each changing the six-vertex weight by a factor  at most  $\mathbf c$.

\item Property (ii): if $\omega'= \mathbf T(\omega)$ and we know $\gamma_{i^*} $ and $ \gamma_{L+i^*}$, we can reconstruct $\omega$. Regardless of $\omega'$, there are at most $N^2 2^{M / \alpha}$ possible pairs of paths $( \gamma_{i^*}, \gamma_{L+i^*} )$: at most $N^2$ pairs of first edges, and at most $2^{2M / \alpha}$ ways to choose the next at most $2M/\alpha$ edges of $\gamma_{i^*} $ and $ \gamma_{L+i^*}$ (the paths turn at every vertex). 
\end{itemize}
This finishes the proof.
\end{proof}

We now turn to the proof of Proposition~\ref{prop: Mano}.

\begin{proof}[Proposition~\ref{prop: Mano}]
Fix a root face $\rho$ on the bottom of $\bbO_{N,M}$ and, for integers $j \geq 0$ and $n, k \geq 1$, let $\calA_j(n,k)$ be the event that there are $2j+2$ vertical $\times$-crossings of the cylinder $(\gamma_i,-j\le i\le j)$ and $\gamma'$, around the cylinder in this order and such that 
\begin{itemize}[noitemsep,nolistsep]
\item $\gamma_0$ starts from  $\rho$,
\item the height $h$ on $\gamma_i$ is $0$ if $|i|\le j$ is even, and $k$ if $|i|\le j$ is odd, 
\item the height $h$ on $\gamma'$ is $(n-j)k$ if $j$ is even and $-(n-j-1)k$ if $j$ is odd.
\end{itemize}
We start by proving that for $0 \le j \leq n-2$,
\begin{equation}\label{eq:ahah}
\bbP^{(0)}_{\bbO_{N,M}}[\calA_{j}(n,k)]\le\bbP^{(0)}_{\bbO_{N,M}}[\calA_{j+1}(n,k)].
\end{equation}
In order to see this, fix $0 \le j \leq n-2$; let us assume for definiteness that $j$ is even (the case of odd $j$ can be treated in a similar fashion). For $h \in \calA_j(n,k)$ or $h \in \calA_{j+1}(n,k)$, suppose (as the choice of the vertical $\times$-crossings inducing this event may not be unique) in the following that for $i > 0$ (resp.~$i < 0$) $\gamma_i$ is taken to be the left-most (resp.~right-most) appropriate $\times$-crossings of $h=0$ or $h=k$ from the root $\rho$. Observe that the crossings $\gamma_{j+1}$ and $\gamma_{-j-1}$ thus defined exist even for $h \in \calA_j(n,k)$ due to the existence of the crossing $\gamma'$ on which $h \geq 2 k$.
Let $X(h)$ be the portion of the cylinder on or right of $\gamma_{-j-1}$ and on or left of $\gamma_{j+1}$ and $\mathcal X$ be the set of possible values of $(X(\omega),h_{|X(\omega)})$ for $\omega$ such that $\gamma_{-j-1}, \dots, \gamma_{j+1}$ exist. We can write 
\begin{align*}
\bbP^{(0)}_{\bbO_{N,M}}[\calA_j(n,k)]&=\sum_{ (X, \xi) \in \mathcal{X} } \bbP^{(0)}_{\bbO_{N,M}}[\calV^\times_{h\ge (n-j)k}(Y)\; \vert \; h_{|X}=\xi] \bbP^{(0)}_{\bbO_{N,M}}[h_{|X}=\xi],\\
\bbP^{(0)}_{\bbO_{N,M}}[\calA_{j+1}(n,k)]&=\sum_{ (X, \xi) \in \mathcal{X} } \bbP^{(0)}_{\bbO_{N,M}}[\calV^\times_{h\le -(n-(j+1)-1)k}(Y)\; \vert \; h_{|X}=\xi] \bbP^{(0)}_{\bbO_{N,M}}[ h_{|X}=\xi],
\end{align*}
where the notation $\calV^\times(Y)$ denotes the occurrence of a vertical $\times$-crossing of the discrete domain $Y$ formed of faces that are in or share a corner with a face in $\bbO_{N,M}\setminus X$. Observe that by the spatial Markov property, $h_{ \vert Y}$ under $\bbP^{(0)}_{\bbO_{N,M}}[\cdot \; \vert \; h_{|X}=\xi] $ has the law $\bbP^{B,\{k-1,k\}}_{Y} $, where the superscript denotes the boundary condition $\{k-1,k\}$ on the union $B$ of the left and right sides of $Y$. From this observation, the comparison between boundary conditions and the invariance of the height function distribution between $h$ and $2k-h$, we deduce that 
\begin{align*}
\bbP^{(0)}_{\bbO_{N,M}}[\calV^\times_{h\ge (n-j)k}(Y) \vert  h_{|X}=\xi]&=\bbP^{B,\{k-1,k\}}_{Y}[\calV^\times_{h\ge (n-j)k}(Y)]\\
&\le \bbP^{B,\{k+1,k\}}_{Y}[\calV^\times_{h\ge (n-j)k}(Y)]\\
&=\bbP^{B,\{k-1,k\}}_{Y}[\calV^\times_{h\le -(n-(j+1)-1)k}(Y)]\\
&= \bbP^{(0)}_{\bbO_{N,M}}[\calV^\times_{h\le -(n-(j+1)-1)k}(Y)\vert  h_{|X}=\xi] ,
\end{align*}
from which~\eqref{eq:ahah} follows.

We now conclude the proof of the proposition. Set $n = \lfloor  \lceil \alpha N \rceil / k \rfloor$, $\rho=s_0$ and observe that $\calB (\lceil \alpha N \rceil) \subset \calB ( nk )$. By the rotational symmetry of the measure around the cylinder, we find that $$\tfrac1N\bbP^{(0)}_{\bbO_{N,M}}[B(nk)]\le\bbP^{(0)}_{\bbO_{N,M}}[\calA_0(n,k)].$$ Using first this observation, then~\eqref{eq:ahah} iteratively $n-1$ times, and then the fact that $\calA_{n-1}(n,k)$ is contained in the union of the $\calA(S,n,k)$ over $S$, where $S$ can be chosen in $\binom{N}{2n} \leq 2^N$ ways, we find
\begin{align*}
\tfrac1N\bbP^{(0)}_{\bbO_{N,M}}[B( \lceil \alpha N \rceil  )]\le\bbP^{(0)}_{\bbO_{N,M}}[\calA_0 (n,k)]\le \bbP^{(0)}_{\bbO_{N,M}}[\calA_{n-1} (n,k)]\le 2^N \max_S\bbP^{(0)}_{\bbO_{N,M}}[\calA(S,n,k)].
\end{align*}
The claim now follows from Lemma~\ref{lem: Mano}.
\end{proof}

\subsection{Proof of Theorem~\ref{thm:RSW_origins}}\label{sec:crossing to annulus}

\paragraph{Parameters and their relations}
We fix the following parameters for the rest of this subsection.
\begin{itemize}[noitemsep,nolistsep]
\item[(i)] Let $\delta > 0$ be an absolute constant so that both Proposition~\ref{prop:RSW1} and Theorem~\ref{thm:RSW} hold\footnote{
The inequalities in Proposition~\ref{prop:RSW1} and Theorem~\ref{thm:RSW} both trivially remain true if we adjust $\delta$ smaller, so there exists such $\delta$ that both hold.
} for this value of $\delta$.
Set $\eta = \delta /12$ and let $ c_0 > 0$ be the absolute constant given by Proposition~\ref{prop:RSW1} in~\eqref{eq:first RSW input to main thm} below.
\item[(ii)] Fix integers $k$ and $r$ with $k$ large and $r > 2k/\eta$; they correspond to those appearing
in Theorem 1.7.
\item[(iii)] Introduce the additional parameters $N, M \in \bbN$, with $N $ even. 
We will ultimately take $M$ and $N$ to infinity (in this order). 
Given $k$ and $r$, we will only work with pairs $N, M$ and their subsequential limits such that
\begin{align}\label{eq:MN}
n := \tfrac{N}{\eta r} \qquad \text{ and } \qquad m := \tfrac{M}{r},
\end{align}
are integers\footnote{By the previous footnote, we may assume $\delta /12 = \eta \in \bbQ$.} 
and $m$ is divisible by $3$.
Finally, we set $\alpha = \frac{ k }{ \eta r}$ so that the relation $\lfloor  \lceil \alpha N \rceil / k \rfloor = n$  of Proposition~\ref{prop: Mano} holds (note also that $\alpha \in (0, 1/2)$ as required).
\end{itemize}
In spite of this hierarchy, we will treat below $k$ and $r$ as ``any integers'', re-stating separately any assumptions on them. In particular, this will clarify the fact that when $k> k_0$ is chosen large enough, the lower bound $k_0$ only depends on parameters that are absolute constants.

\paragraph{The setup for the proof}
Let $S = \{s_0,\dots, s_{2n-1}\}$ be the set of faces, as in Proposition~\ref{prop: Mano}, that maximizes the probability $\bbP^{(0)}_{\bbO_{N,M}} [ \calA(S,n,k) ]$.
Let $\calX$ be the union of the clusters of $h\le 1$ of $s_0,s_2,\dots,s_{2n-2}$ and their bounding $\times$-paths of $h=2$.
 Since  
\begin{align*}
	\bbP^{(0)}_{\bbO_{N,M}} [ \calA(S,n,k) ] 
	= \sum_{X} \bbP^{(0)}_{\bbO_{N,M}} [ \calA(S,n,k) \,|\, \calX = X]\bbP^{(0)}_{\bbO_{N,M}} [\calX = X],
\end{align*}
one may find a realisation $X$ of $\calX$ such that 
\begin{align}\label{eq:hu1}
	\bbP^{(0)}_{\bbO_{N,M}} [ \calA(S,n,k) \,|\, \calX = X] \geq \bbP^{(0)}_{\bbO_{N,M}} [ \calA(S,n,k) ].
\end{align}
Fix $X$ to be such a realisation.

\begin{figure}[t]
\begin{center}
\includegraphics[width=0.36\textwidth]{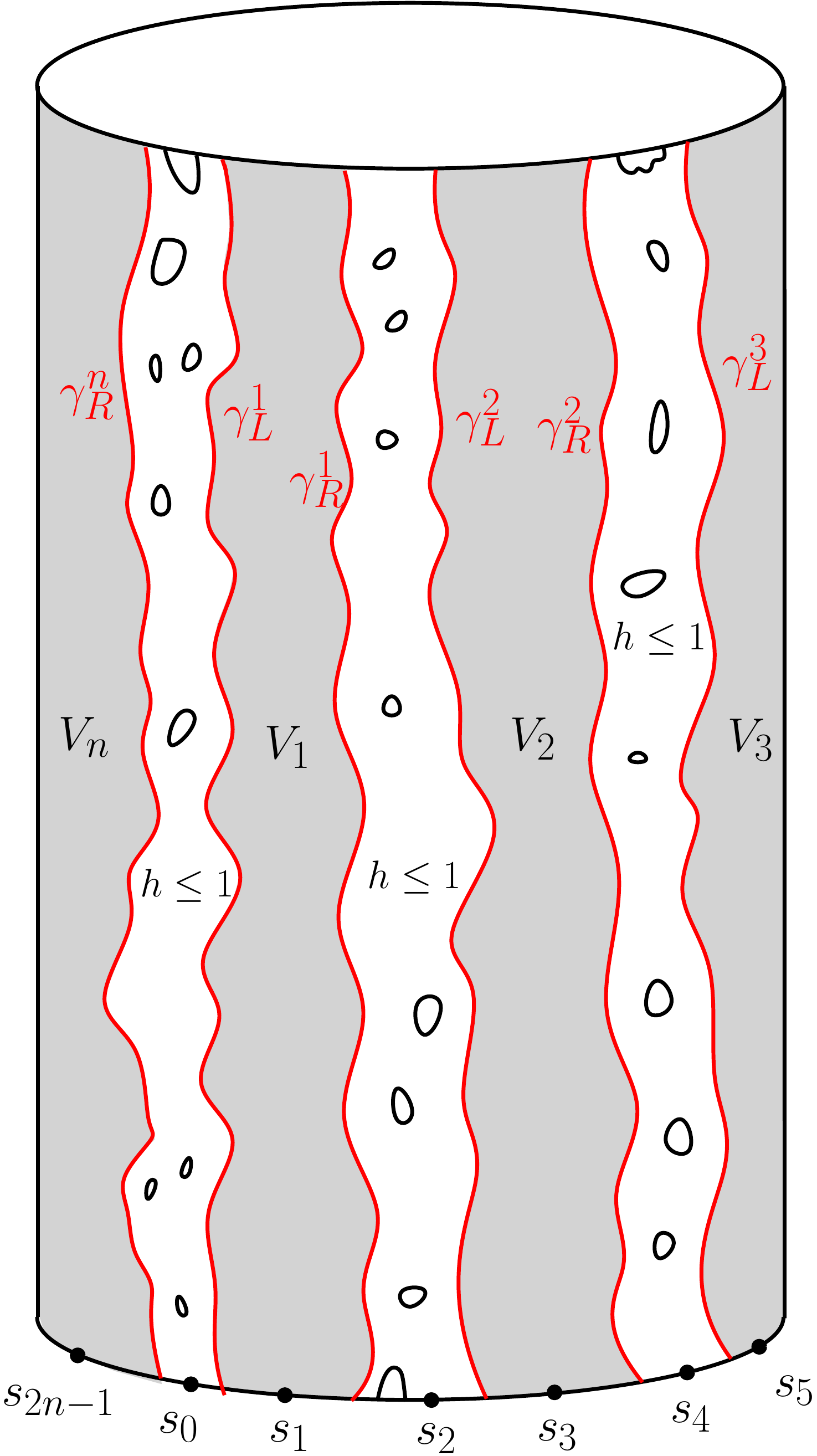}
\caption{An illustration of the setup of the proof of Theorem~\ref{thm:RSW_origins}. }
\label{fig:vertical domains}
\end{center}
\end{figure}

Now, $X$ is such that it does not exclude $\calA(S,n,k)$. In particular, $\mathbb O_{N,M} \setminus X$ contains $n$ regions $\tilde V_1,\dots, \tilde V_n$ (containing the faces $s_1,s_3,\dots s_{2n - 1} $, respectively) and $\calA(S,n,k)$ means that each of them contains a vertical $\times$-crossing of $\mathbb O_{N,M}$ of height at least $k$. 
Write $V_i$ for the discrete domain formed of faces of $\mathbb O_{N,M}$ in $\tilde V_i$ or sharing a corner with a face in $\tilde V_i$. Note that $V_i$ has a natural quad structure, with top and bottom sides on the top and bottom of $\bbO_{N,M}$ and left and right sides given by the faces of $X \cap V_i$ with height in $\{1, 2\}$; denote the latter two paths by $\gamma_L^i$ and $\gamma_R^i$, respectively, and orient them from bottom to top. See Figure~\ref{fig:vertical domains}.

For $1 \leq y \leq m$ let $\Slice_y$ be the translate by $(0,(y-1) r)$ of $\bbO_{N,r}$, seen as a subset of $\bbO_{N,M}$.
These horizontal slices form a partition of  $\bbO_{N,M}$. We next will construct several sub-domains of $V_x$ defined in terms of its intersection with certain slices. 

Notice that $V_x \cap \Slice_y$ may be formed of several domains (i.e. connected components). 
Call a domain of $V_x \cap \Slice_y$ that intersects both the top and bottom of $\Slice_y$ a {\em valid} domain.
The boundary of any valid domain is formed of segments from the top and bottom of $\Slice_y$ and sub-paths of $\gamma_L^x$ and $\gamma_R^x$. Among the latter, there exists exactly two sub-paths with one endpoint on the bottom of $\Slice_y$ and one on the top of $\Slice_y$. 
These two sub-paths bound the domain on the left and on the right; we will call then the left and right boundary of the domain. 
Generally the left and right boundary of a valid domain may be both part of the same path $\gamma_x^L$ or $\gamma_x^R$, or may be one part of $\gamma_x^L$ and the other part of $\gamma_x^R$. In the latter case, we call the valid domain {\em traversing}. 

Several geometric observations follow. First, any $V_x \cap \Slice_y$ contains at least one traversing domain. Second, any path running inside $V_x$ from the bottom to the top of  $\bbO_{N,M}$ intersects all traversing domains.
Finally, the traversing domains of $V_x \cap \Slice_y$ may be naturally ordered from bottom to top, with one traversing domain being considered below another, if the former may be connected to the bottom of $\bbO_{N,M}$ by a path inside $V_x$ which avoids the latter. 

Fix $y \equiv 2 $ mod~$3$. Then $\Slice_{y - 1} \cup \Slice_{y } \cup \Slice_{y + 1}$ also forms a slice of $\bbO_{N,M}$, and the denominations above apply to $V_x \bigcap (\Slice_{y - 1} \cup \Slice_{y } \cup \Slice_{y + 1} )$. 
Let  $U_{x, y}$ denote the bottom-most traversing domain of $V_x \bigcap (\Slice_{y - 1} \cup \Slice_{y } \cup \Slice_{y + 1} )$; 
write $\vartheta_L^{x,y}$ and $\vartheta_R^{x,y}$ for its left and right boundary, respectively. 
One may observe that, due to $U_{x,y}$ being bottom-most, $\vartheta_L^{x,y}$ is contained in $\gamma_L^x$ 
and $\vartheta_R^{x,y}$  is contained in $\gamma_L^x$. 
Write $\overline{U}_{x, y}$ for the domain of $(\Slice_{y - 1} \cup \Slice_{y } \cup \Slice_{y + 1})$ contained between $\vartheta_L^{x,y}$ and $\vartheta_R^{x,y}$, so that $U_{x, y} \subset \overline{U}_{x, y}$. 
See Figure~\ref{fig:first quad fig}~(left).

For $y \equiv 0$ or $1$ mod~$3$, define $Q_{x, y}$ as the union of valid domains of $U_{x, y} \cap \Slice_{y}$.
Also, let $\gamma_L^{x, y}$ (resp. $\gamma_R^{x, y}$) be the unique sub-path of $\vartheta_L^{x,y}$ (resp. $\vartheta_R^{x,y}$) between the top and bottom of $\Slice_y$ -- the uniqueness comes from the fact that, when $y \equiv 0$ mod~$3$, $\vartheta_L^{x, y}$ and $\vartheta_R^{x, y}$ intersect the top of $\Slice_y$ exactly once, while when $y \equiv 1$ mod~$3$, they intersect the bottom of $\Slice_y$ exactly once.
Then, $Q_{x,y}$ is contained between $\gamma_L^{x, y}$ and $\gamma_R^{x, y}$; see Figure~\ref{fig:first quad fig} (right). 

For $y \equiv 2 $ mod~$3$, defined $Q_{x,y}$ as the bottom-most traversing domain of $V_x \cap \Slice_y$ which is contained in $U_{x,y}$ 
(as above, it may be observed that $U_{x,y}$ contains at least one traversing domain of  $V_x \cap \Slice_y$). 
Write $\gamma_L^{x, y}$ and $\gamma_R^{x, y}$ for the left and right boundary of $Q_{x,y}$, respectively.
Note that $\gamma_L^{x, y}$ is {\em not} necessarily contained in $\vartheta_L^{x,y}$ or  $\vartheta_R^{x,y}$ -- see Figure~\ref{fig:first quad fig} (right) for an example. 
A key feature of this construction is that any path contained in $V_x$, linking $Q_{x,y}$ to the bottom of $\Slice_{y-1}$ or the top of $\Slice_{y+1}$, necessarily contains a vertical crossing of $Q_{x,y-1}$ or $Q_{x,y+1}$, respectively. 

\begin{figure}[t]
    \begin{center}
    \includegraphics[width=0.75\textwidth]{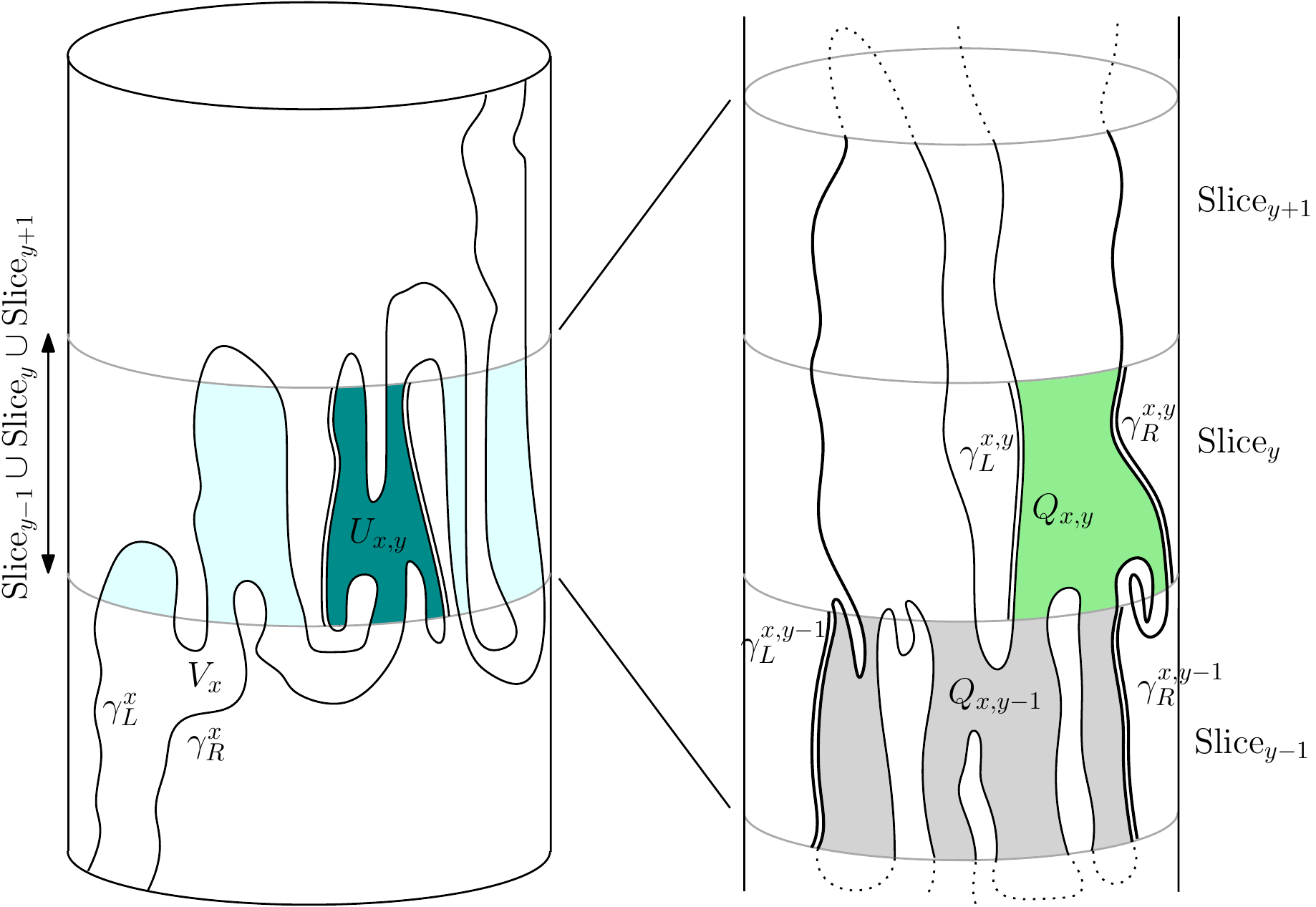}
    \caption{{\em Left:} $V_x \bigcap (\Slice_{y - 1} \cup \Slice_{y } \cup \Slice_{y + 1} )$ may consist of several domains; $U_{x, y}$ is the bottom-most traversing one. The curves $\vartheta_L^{x,y}$ and $\vartheta_R^{x,y}$ are highlighted by doubled lines.
    {\em Right:} $U_{x, y} \bigcap \Slice_{y-1}$ may consist of several domains; those that cross the slice vertically form $Q_{x, y-1}$.
    For the middle slice, $Q_{x,y}$ is defined in the same way as $U_{x,y}$. It is separated inside $V_x$ from the bottom and top of the cylinder by $Q_{x,y-1}$ and $Q_{x,y+1}$, respectively. 
    The curves $\gamma_L^{x,y-1}$, $\gamma_R^{x,y-1}$, $\gamma_L^{x,y}$ and $\gamma_R^{x,y}$ are highlighted.}
    \label{fig:first quad fig}
    \end{center}
\end{figure}

For all $x,y$, set $\overline{Q}_{x, y}$ to be the domain of $\Slice_y$ contained between $\gamma_L^{x,y}$ and $\gamma_R^{x,y}$.
Thus $Q_{x, y} \subset \overline{Q}_{x, y}$ and the latter has a natural quad structure, 
with two arcs formed by $\gamma_L^{x,y}$ and $\gamma_R^{x,y}$ and the two others formed by parts of the top and bottom of $\Slice_{y}$, respectively. 
Denote the top and bottom boundary arcs of $ \overline{Q}_{x, y} $ by ${\rm Top}_{x,y}$ and ${\rm Bottom}_{x,y}$, respectively. 
We call $\overline Q_{x,y}$ {\em tight} if ${\rm Top}_{x,y}$ and ${\rm Bottom}_{x,y}$ each consist of at most $\lfloor \delta r \rfloor $ faces (where $\delta >0$ is the absolute constant described above).
Furthermore, for $(x,y)$ with $y \equiv 2 $ mod~$3$, we say that $(x,y)$ is {\em good} if $\overline Q_{x,y-1}$, $\overline Q_{x,y}$ and $\overline Q_{x,y+1}$ are all tight. 

\begin{lemma}\label{lem:good_Q}
At least half of the pairs $(x,y)$, with $1 \leq x \leq n$ and $1 \leq y \leq m$ and $y \equiv 2 $ mod~$3$, are good.
\end{lemma}

\begin{proof}
	We will actually prove a slightly stronger claim: namely, at least half of the $n$ pairs $(x,y)$, 
	with fixed $1 \leq y \leq m$, $y \equiv 2 $ mod~$3$, are good. Fix thus one such $y$ for the rest of the proof.
		
	We first claim that the quads $\overline U_{x,y}$ with $1 \leq x \leq n$ are disjoint from each other.
	Fix~$x \neq x'$. Note that the paths $\vartheta_L^{x',y}$ and $\vartheta_R^{x',y}$ do not intersect $V_x$, and in particular not $\vartheta_L^{x,y}, \vartheta_R^{x,y} \subset V_x$. 
	Furthermore, $\overline U_{x,y} \cap V_x$ contains a horizontal path between $\vartheta_L^{x,y}$ and $\vartheta_R^{x,y}$;
	it follows that each of $\vartheta_L^{x',y}$ and $\vartheta_R^{x',y}$ lies entirely outside of $\overline U_{x,y}$. In particular, either $\overline U_{x,y}$ and $\overline U_{x',y}$ are disjoint, or $\overline U_{x,y} \subset \overline U_{x',y}$.	
	Symmetry allows us to discard the latter scenario, and we conclude that the quads $\overline U_{x,y}$ with $1 \leq x \leq n$ are disjoint. 
	
	Next, we claim that for any fixed $y' \in \{y, y \pm 1\}$, the quads $\overline  Q_{x,y'}$, with $1 \leq x \leq n$ are also disjoint. 
	(It is worth observing that that $\overline  Q_{x,y'}$ is not necessarily contained in  $\overline U_{x,y}$ 
	(see Figure~\ref{fig:first quad fig}, right diagram), so this is not immediate from the above.)
	For $y' = y$, the same proof as for $\overline U_{x,y}$ applies. 
	For $y' = y \pm 1$, a slight alteration of the argument above is necessary. 
	Indeed, for $x\neq x'$, since $\overline U_{x,y}\cap \overline U_{x',y} = \emptyset$,
	the paths $\gamma_L^{x',y'}$ and $\gamma_R^{x',y'}$ do not intersect $\overline U_{x,y}$, in particular not $\gamma_L^{x,y'}$ nor $\gamma_R^{x,y'}$. 
	But $\overline Q_{x,y'} \cap \overline U_{x,y}$ contains a path from $\gamma_L^{x,y'}$ to $\gamma_R^{x,y'}$, 
	namely its bottom or top, so both $\gamma_L^{x',y'}$ and $\gamma_R^{x',y'}$ are entirely outside of $\overline Q_{x,y'}$. Thus,  either $\overline Q_{x,y'}$ and $\overline Q_{x',y'}$ are disjoint, or $\overline Q_{x,y'} \subset \overline Q_{x',y'}$, and the latter is again excluded by symmetry.

	From the above, we conclude that for any $y' \in \{y, y \pm 1\}$,  
	the disjoint union of ${\rm Bottom}_{x,y'}$ for $x = 1,\dots, n$ is contained in one row of $N$ faces of $\bbO_{N,M}$.
	By the pigeon hole principle, at least a proportion $11/12$ of the quad bottoms $({\rm Bottom}_{x,y'})_{1\leq x\leq n}$ contain at most
	$12 N/n =  12 \eta r = \delta r$ faces (using the relations of the various parameters).
	The same holds for the tops of the quads $(\overline Q_{x,y'})_{1\leq x\leq n}$, 
	and we conclude that out of the $n$ quads $(\overline Q_{x,y'})_{1\leq x\leq n}$, there are at most $n/6$ quads that are not tight.
	Finally, out of the $n$ triplets of quads 
	$(\overline Q_{x,y-1}, \overline Q_{x,y}, \overline Q_{x,y+1})$, at least $n/2$ are formed exclusively of tight quads. 
\end{proof}

Let now $\calR$ be the ``ridge event'' that each $Q_{x,y}$ with $y \equiv 2 $ mod~$3$ contains a $\times$-path of $h \geq k$ between 
 ${\rm Top}_{x,y}$ and ${\rm Bottom}_{x,y}$. 
Then, we have $\calA(S,n,k) \subset \calR $ and 
\begin{align}\label{eq:hu2}
	\bbP^{(0)}_{\bbO_{N,M}} [ \calR \,|\, \calX = X] \geq \bbP^{(0)}_{\bbO_{N,M}} [ \calA(S,n,k) |\calX=X].
\end{align}

Moreover, define the ``fencing event'' $\calF$ that for each $(x,y)$ with $y \equiv 2 $ mod~$3$ which is good, 
$Q_{x,y-1}$ and $Q_{x,y+1}$ \textit{do not} contain $\times$-paths of $h \geq (1-c_0)k + 1$ (where $c_0$ is the absolute constant defined above) between the top and bottom of $\Slice_{y-1}$ and $\Slice_{y+1}$, respectively. (Equivalently, by Remark~\ref{rem:dual_crossings}, each component of $Q_{x,y-1}$ and $Q_{x,y+1}$ is crossed horizontally by a path of $h \leq (1-c_0)k $.)

\paragraph{The key lemmas}
The proof of Theorem~\ref{thm:RSW_origins} hinges on two lemmas which we now state and prove. 

\begin{lemma}[Building fences]\label{lem:fences}
With the parameters and notations above, we have for all $r > 0$ and all $k >k_0(c_0)$ large enough
	\begin{align*}
		\bbP^{(0)}_{\bbO_{N,M}} [\calF \,|\,\calR \text{ and } \calX = X] \geq c_0^{n m}.
	\end{align*}
\end{lemma}

\begin{proof}
	The occurrence of $\calR$ may be determined by exploring, for each $y$ with $y \equiv 2 $ mod~$3$, the component of $h \leq k-1$ in $Q_{x, y}$ that contains $\gamma_L^{x, y}$, and the $\times$-paths of $h=k$ bounding this component. Indeed, either the component of $h \leq k-1$ reaches $\gamma_R^{x, y}$, hence preventing any vertical $\times$-path of $h \geq k$, or $\gamma_L^{x, y}$ and $\gamma_R^{x, y}$ are separated in $Q_{x, y}$ by a $\times$-path of $h=k$, which due to the boundary conditions traverses from ${\rm Top}_{x,y}$ to ${\rm Bottom}_{x,y}$.
	This exploration only reveals faces in $Q_{x,y}$ with height at most $k$. Let ${\sf Exp}$ denote the random pair of faces and heights explored in this procedure. 
	Then 
	\begin{align}
	\label{eq:explore fences}
		\bbP^{(0)}_{\bbO_{N,M}} & [\calF \,|\,\calR \text{ and } \calX = X] \\
		& = \sum_{(E, h_{ \vert E} ) }\bbP^{(0)}_{\bbO_{N,M}} [\calF \,|\,{\sf Exp} = (E, h_{ \vert E} )  \text{ and } \calX = X] \bbP^{(0)}_{\bbO_{N,M}} [ {\sf Exp} = (E, h_{ \vert E} )  \, | \, \calR \text{ and } \calX = X],
			\nonumber
	\end{align}
	where the sum is over all possible realisations $(E, h_{ \vert E} ) $ of ${\sf Exp}$ such that $\calR$ occurs. 
	
\begin{figure}
\begin{center}
\includegraphics[width=0.55\textwidth]{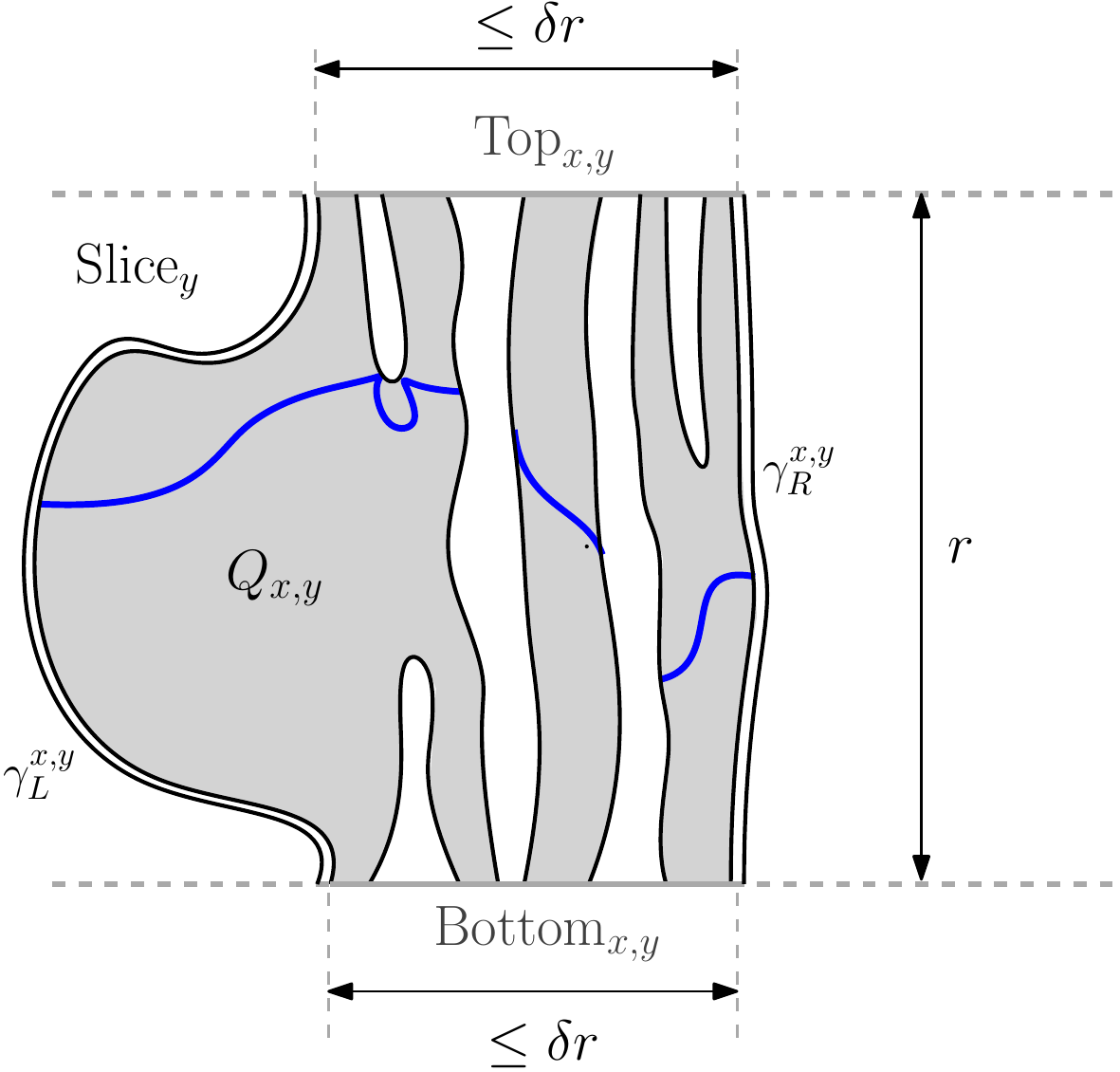}
\caption{
\label{fig:fence lemma}
An illustration for the proof of Lemma~\ref{lem:fences}. In every domain of $Q_{x, y}$, exactly two boundary segments in $\gamma_L^x \cup \gamma_R^x$ (in black) cross $\Slice_y$ vertically; the left- and right-most such crossings are $\gamma_L^{x,y}$ and $\gamma_{R}^{x,y}$, respectively.
 The boundary condition $\xi$ on $Q_{x,y}$ is $\{1, 2\}$  on the black parts of $\partial Q_{x, y}$, and their maximal extension which is smaller or equal to $k$ on the gray parts. 
The blue paths form a fence: they separate ${\rm Top}_{x,y}$ from ${\rm Bottom}_{x,y}$ inside $Q_{x,y}$ and have~$h \leq (1-c_0)k$.
}
\end{center}
\end{figure}

	Fix some $(E, h_{ \vert E} ) $ as above such that $\calR$ occurs, and fix $(x,y)$ with $y \not \equiv 2 $ mod~$3$, such that $\overline Q_{x,y}$ is tight. Recall the ``dual formulation'' of $\calF$
	and denote 
	\begin{align*}
	\calF^{x, y}_\ell = \{ \text{there is a left-to-right crossings of $h\leq \ell $ in each component of $Q_{x,y}$} \}
	\end{align*}
	(the meaning of ``left-to-right'' is explained with Figure~\ref{fig:fence lemma}).
	Recall that $E$ contains no face in $Q_{x,y}$.
	Due to~\eqref{eq:pushing_11} applied to $-h$, we now have
	\begin{equation}
	\label{eq:starting chain of inequalities here}
		\bbP^{(0)}_{\bbO_{N,M}}  [ \calF^{x, y}_{ (1-c_0)k } \,|\,{\sf Exp}  = (E, h_{ \vert E} )  \text{ and } \calX = X]
		\geq \bbP^{\xi}_{Q_{x,y}} [ \calF^{x, y}_{ (1-c_0)k -2 }   ] ,
	\end{equation}
	where $\xi$ is the largest boundary conditions on $\partial Q_{x,y}$ 
	that is at most $k$ and takes values in $\{1, 2\}$ on $\partial Q_{x,y} \setminus ({\rm Top}_{x,y}\cup {\rm Bottom}_{x,y})$.

	Now, going back to the ``primal formulation'' of $\calF$ and using Remark~\ref{rem:dual_crossings}, we have 
	\begin{align*}
		\bbP^{\xi}_{ Q_{x,y}} [ \calF^{x, y}_{ (1-c_0)k -2 }  ] 
		&\geq 1  - \bbP^{\xi}_{Q_{x,y}} [ {\rm Bottom}_{x,y}\xlra{h \geq (1-c_0)k - 2 \text{ in } Q_{x,y}}  {\rm Top}_{x,y}];
	\end{align*}
here and the rest of this proof, we omit integer roundings in $\lfloor (1-c_0 ) k \rfloor$.
	Then, by Corollary~\ref{cor:pushing_modified} and then inclusion of events,
		\begin{align*}
		\bbP^{\xi}_{Q_{x,y}} [ &{\rm Top}_{x,y}\xlra{h \geq (1-c_0)k - 2 \text{ in } Q_{x,y}}  {\rm Bottom}_{x,y}]\\
		& \leq \bbP^{\xi'}_{\bbZ \times [yr, (y+1) r]}	[ {\rm Top}_{x,y}\xlra{h \geq (1-c_0)k -4 \text{ in } Q_{x,y}}  {\rm Bottom}_{x,y}]\\
		& \leq 
		\bbP^{\xi'}_{\bbZ \times [yr, (y+1) r]}
		[ {\rm Top}_{x,y}\xlra{h \geq (1-c_0)k -4 }  {\rm Bottom}_{x,y}],
	\end{align*}
	where $\xi'$ is the boundary condition on $\partial \bbZ \times [yr, (y+1) r]$ that takes values in $\{ 1, 2\}$ outside of ${\rm Top}_{x,y}$ and ${\rm Bottom}_{x,y}$, where it is given by the maximal extension smaller or equal to~$k$.
	
	Now, recall that $ {\rm Top}_{x,y}$ and ${\rm Bottom}_{x,y}$ both contain at most $\lfloor \delta r \rfloor$ faces.
	Proposition~\ref{prop:RSW1} (and our original choice of $c_0$ and $\delta$ to match the below equation) guarantees~\footnote{Strictly speaking, for the boundary condition $\xi$, Proposition~\ref{prop:RSW1} addresses crossings of $h \geq (1-c_0) (k-1) + 1$ but the above holds by adjusting $c_0$ suitably smaller and taking $k > k_0 (c_0)$ large enough.}
	that
	\begin{align}
	\label{eq:first RSW input to main thm}
			\bbP^{\xi'}_{\bbZ \times [yr, (y+1) r]}
		[ {\rm Top}_{x,y}\xlra{h \geq (1-c_0)k -4 }  {\rm Bottom}_{x,y}] \leq 1-c_0 .
	\end{align}
	Tracing through the chain of inequalities that started from~\eqref{eq:starting chain of inequalities here}, we have
	\begin{align*}
			\bbP^{(0)}_{\bbO_{N,M}} [ \gamma_L^{x,y} \xlra{h\leq (1-c_0)k \text{ in $Q_{x,y}$}} \gamma_R^{x,y}\,|\,{\sf Exp} = (E, h_{ \vert E} )  \text{ and } \calX = X] \geq c_0.
	\end{align*}
	
	Finally, due to~\eqref{eq:SMP},~\eqref{eq:FKG-h} applies to the conditional measure $\bbP^{(0)}_{\bbO_{N,M}} [ \cdot\,|\,{\sf Exp} = (E, h_{ \vert E} )  \text{ and } \calX = X]$. 
	As there are at most $mn$ collections $Q_{x,y}$ needing to be crossed in order for $\calF$ to occur, we conclude that 
	\begin{align*}
		\bbP^{(0)}_{\bbO_{N,M}} [\calF \,|\,{\sf Exp} = (E, h_{ \vert E} ) \text{ and } \calX = X]
		\geq c_0^{mn}.
	\end{align*}
	The statement then follows from~\eqref{eq:explore fences}.
\end{proof}

\begin{lemma}[Ridges given fences]\label{lem:ridges}
With the parameters and notations above, we have for all $r > 0$ and all $k>k_0(c_0 )$ large enough
	\begin{align*}
		\bbP^{(0)}_{\bbO_{N,M}} [\calR \,|\, \calF \text{ and } \calX = X] \leq 
		\Big(2\bbP_{\bbZ \times [-r,2r]}^{0,1}\big[[0,\delta r]\times \{0\} \xleftrightarrow{ h \geq c_0k - 2\text{ in $\bbZ \times [0,r]$}}\bbZ\times \{ r\}\big]\Big)^{\tfrac{mn}{6}}.
	\end{align*}
\end{lemma}

\begin{proof}
	When $\calF$ occurs, for each good pair $(x,y)$, 
	let $\chi_T^{x,y}$ be the collection of top-most paths, in each connected component of $Q_{x,y+1}$, of height at most $(1-c_0)k$ that disconnect the bottom and top of $\Slice_{y+1}$.
	Similarly, let $\chi_B^{x,y}$ be the bottom-most paths in $Q_{x,y-1}$, of height at most $(1-c_0)k$	(here and for the rest of this proof, we again omit integer roundings in $\lfloor (1-c_0) k \rfloor$). 
	Write $D_{x,y}$ the connected component of $Q_{x,y}$ in the union of the faces of $U_{x, y}$ 
	contained on or between the curves of $\chi_B^{x,y}$ and $\chi_T^{x,y}$. 
	
	Notice that the domains $(D_{x,y}\, :\,(x,y) \text{ good})$ are measurable in terms of the height function outside of them
	and on their boundaries.
	Thus, conditionally on any realisation of these domains and on a realisation $\zeta$ of the height function outside of them
	and on their boundaries,  
	the height functions inside the different domains $D_{x,y}$ are independent of each other and follow laws $\bbP_{D_{x,y}}^{\zeta}$.

	The definition of $D_{x,y}$ is such that the values of $\zeta$ on $\partial D_{x,y}$ are at most $(1-c_0)k$ (when $k>k_0(c_0)$ is large enough so that $(1-c_0)k \geq 2$). 
	By~\eqref{eq:CBC-h}, each measure $\bbP_{D_{x,y}}^{\zeta}$ is stochastically dominated by $\bbP_{D_{x,y}}^{(1-c_0)k-1, (1-c_0)k}$.
	Thus, for any good $(x,y)$, using~\eqref{eq:pushing_11} (and $D_{x, y} \subset \Slice_{y-1} \cup \Slice_{y} \cup \Slice_{y+1}$), we have
	\begin{align*}
	\bbP_{D_{x,y}}^{\zeta}[{\rm Bottom}_{x,y} \xlra{h \geq k \text{ in $Q_{x,y}$}}_{\times} {\rm Top}_{x,y}]
	&\leq \bbP_{D_{xy}}^{(1-c_0)k-1, (1-c_0)k}[{\rm Bottom}_{xy} \xlra{h \geq k-1 \text{ in $Q_{x,y}$}}  {\rm Top}_{xy}] \\
	&\leq \bbP_{D_{x,y}}^{0,1}[{\rm Bottom}_{x,y} \xlra{h \geq c_0 k \text{ in $Q_{x,y}$}} {\rm Top}_{x,y}]\\
	&\leq 2\,\bbP_{\bbZ \times [(y-1)r,(y+2)r] }^{0,1}[{\rm Bottom}_{xy} \xlra{h \geq c_0k-2 \text{ in $Q_{x,y}$}} {\rm Top}_{x,y}]\\
	&\leq 2\,\bbP_{\bbZ \times [-r,2r]}^{0,1} \big[[0, \lfloor \delta r \rfloor ]\times \{0\} \xleftrightarrow{ h \geq c_0k - 2\text{ in $\bbZ \times [0,r]$}}\bbZ\times \{ r\}\big].
	\end{align*}
	The last inequality follows from the fact that $(x,y)$ is good, and therefore $\overline Q_{x,y}$ is tight, 
	which is to say that ${\rm Bottom}_{x,y}$ is shorter than $\delta r$. 
	
	Finally, since $\calR$ imposes that ${\rm Bottom}_{x,y} \xlra{h \geq k} {\rm Top}_{x,y}$ occurs in every domain $D_{x,y}$ and
	since there are at least $mn/6$ good pairs $(x,y)$, using the independence of the measures inside the domains $D_{x,y}$ and the computation above, we find that  
	\begin{align*}
			\bbP^{(0)}_{\bbO_{N,M}} [\calR \,|\, \calF \text{ and } \calX = X] \leq 
		\Big(2\bbP_{\bbZ \times [-r,2r]}^{0,1}\big[[0, \lfloor \delta r \rfloor  ]\times \{0\} \xleftrightarrow{ h \geq c_0k - 2\text{ in $\bbZ \times [0,r]$}}\bbZ\times \{ r\}\big]\Big)^{\tfrac{mn}{6}}, 
	\end{align*}
	as required.
\end{proof}

\begin{proof}[Theorem~\ref{thm:RSW_origins}]
	In this proof we require that $k >k_0(c_0, c_1)$ is large enough so that Lemmas~\ref{lem:fences} and~\ref{lem:ridges} as well as~\eqref{eq:second RSW input to main thm} below apply; we also require $r > 2k/\eta$ in order to apply Proposition~\ref{prop: Mano}.
	
	 Using elementary probability in the first step, and then Lemma~\ref{lem:fences} as well as~\eqref{eq:hu1} and~\eqref{eq:hu2} in the second, we have
	\begin{align*}
		\bbP^{(0)}_{\bbO_{N,M}} [\calR \,|\, \calF \text{ and } \calX = X]
		&\geq \bbP^{(0)}_{\bbO_{N,M}} [\calF \,|\, \calR \text{ and } \calX = X]\bbP^{(0)}_{\bbO_{N,M}} [\calR\,|\, \calX = X]\\
		&\geq c_0^{n m} \bbP^{(0)}_{\bbO_{N,M}} [ \calA(S,n,k) ].
	\end{align*}
	Applying now Lemma~\ref{lem:ridges}, we deduce that 
	\begin{align}
	\label{eq:almost there}
		\bbP_{\bbZ \times [-r,2r]}^{0,1}\big[[0, \lfloor \delta r \rfloor ]\times \{0\} \xleftrightarrow{ h \geq c_0k -2\text{ in $\bbZ \times [0,r]$}}\bbZ\times \{ r\}\big] 
		\geq  \frac{c_0^{6}}{2} \, \bbP^{(0)}_{\bbO_{N,M}} [ \calA(S,n,k) ]^{\tfrac{6}{n m}}.
	\end{align}
	Observe that the left-hand side does not depend on $M$ and $N$, while the right-hand one does. 
	Recall next that our choice of the parameter $\alpha  = \eta k /r$ (and the relation $r > 2k/\eta$)
	was matched for applying Proposition~\ref{prop: Mano}, which gives
	\begin{align*}
	 \bbP^{(0)}_{\bbO_{N,M}} [ \calA(S,n,k) ] \geq \exp \big( NM ( f_{\mathbf c}(\alpha)- f_{\mathbf c}(0) ) + o(NM) \big)
	\end{align*}
	as $M \to \infty$ and then $N \to \infty$.
	Applying this and the definitions~\eqref{eq:MN} of $m$ and $n$, the factor on the right-hand side of~\eqref{eq:almost there} becomes
	\begin{align}\label{eq:almost there RHS}
		 \bbP^{(0)}_{\bbO_{N,M}} [ \calA(S,n,k) ]^{\tfrac{6}{n m}}
		\geq \exp \left( 6 \eta r^2 ( f_{\mathbf c}(\tfrac{k}{\eta r} ) -f_{\mathbf c}(0) ) + o (1) \right) .
	\end{align}	
	 For the left-hand side of~\eqref{eq:almost there}, we apply Theorem~\ref{thm:RSW} (recall that $\delta$ was chosen so that it applies) to deduce that there exist absolute constants $c_1, C_1 > 0$ such that
	 \begin{align}\label{eq:second RSW input to main thm}
	 	\big(\tfrac{1}{c_1} \bbP_{\Lambda_{12 r}}^{0,1}[\calO_{h\ge c_1 k}( 6 r , 12 r)] \big)^{1/C_1 }
	 	\geq \bbP_{\bbZ \times [-r,2r]}^{0,1}\big[[0, \lfloor \delta r \rfloor ]\times \{0\} \xleftrightarrow{ h \geq c_0k -2\text{ in $\bbZ \times [0,r]$}}\bbZ\times \{ r\}\big],
	 \end{align}
	 for all $k > k_0(c_0, c_1)$ large enough.
	
	Injecting~\eqref{eq:almost there RHS} and~\eqref{eq:second RSW input to main thm} into~\eqref{eq:almost there}, we get that for suitable absolute constants $c, C > 0$,
	\begin{align*}
	\bbP_{\Lambda_{12 r}}^{0,1}[\calO_{h\ge c k}( 6 r , 12 r)]		
	\geq  c   \exp \big(C \,r^{2} [f_{\mathbf c}(\tfrac{k}{\eta r})-f_{\mathbf c}(0)] \big).
	\end{align*}
	This finishes the proof.
\end{proof}

\subsection{Proof of Theorem~\ref{thm:bxp}}

Observe first that by inclusion of events, it suffices to prove the claim when $k$ is larger than some constant.
Second, by height shift and~\eqref{eq:CBC-h}, 
\begin{align*}
	\bbP_{D}^{\xi}[\calO_{h\ge k}( n ,2n)]
	= \bbP_{D}^{\xi +\ell }[\calO_{h\ge k+\ell}( n,2n)]
	\geq 
		\bbP_{D}^{0,1}[\calO_{h\ge k+\ell}( n ,2n)],
\end{align*}
so, by adjusting $k$, it suffices to prove the claim for $\xi = \{0,1\}$. Third, observe that by Corollary~\ref{cor:pushing_modified} (or Proposition~\ref{prop:pushing_1} if the conditions $\xi$ and $\{0,1\}$ above were only imposed on a subset of $\partial D$), when $D \supset \Lambda_{2n}$, we have
\begin{align*}
	\bbP_{D}^{0,1}[\calO_{h\ge k }( n ,2n)]
	\geq \tfrac12 \bbP_{\Lambda_{2n}}^{0,1}[\calO_{h\ge k + 2}( n ,2n)].
\end{align*}
Thus, (adjusting $k$ again) it suffices to prove claim for the $D=\Lambda_{2n}$. We thus turn to proving the claim for $k$ large enough, $\xi=\{0,1\}$, and $D=\Lambda_{2n}$.

Fix now $\mathbf c \in [1,2]$ and $k$ large enough for Theorem~\ref{thm:RSW_origins} to apply.
Let $\eta, c,C > 0 $ and $C_0 > 0$ be the constants appearing in~Theorem~\ref{thm:RSW_origins} and~~\eqref{eq:BA}, respectively.
Applying~\eqref{eq:RSW_origins} and~\eqref{eq:BA} gives
\begin{align*}
	\bbP_{\Lambda_{12 r}}^{0,1}[\calO_{h\ge ck}( 6 r ,12 r)]
    	\ge c \exp\big[C r^2\big(f_\mathbf c (\tfrac{ k }{\eta r} )-f_\mathbf c(0)\big)\big]
	\ge c \exp\big[-C  C_0 k^2/\eta^{2}\big] > 0.
\end{align*}
This directly implies the claim when $n=6r$ is a multiple of $6$.
For general $n$, let $n'$ be the smallest multiple of $6$ with $n' \geq n$. Note that under $\bbP_{\Lambda_{2n'}}^{0,1}$, we necessarily have $h \leq 11$ on $\partial \Lambda_{2n}$. Thus, by~\eqref{eq:SMP} and~\eqref{eq:CBC-h}, we have
\begin{align*}
\bbP_{\Lambda_{2n'}}^{0,1}[\calO_{h\ge k+10}(  n' ,2 n')]
\leq
\bbP_{\Lambda_{2n}}^{10, 11}[\calO_{h\ge k+10}( n' ,2n)]
\leq
\bbP_{\Lambda_{2n}}^{0, 1}[\calO_{h\ge k}( n,2n)],
\end{align*}
where the second step used a shift of boundary conditions and inclusion of events. This concludes the proof.
%
%
\hfill $\square$

\section{Logarithmic bounds on variance of height functions}\label{sec:5}

Throughout this section, we restrict our attention to the six-vertex model with $1\le \mathbf c\le 2$. 

\subsection{Lower bounds}\label{sec:variance_lb}

\subsubsection{Proof of the lower bound in Corollary~\ref{cor:variance_D}}\label{sec:variance_lb_D}

The proof will be based on the following quantity:
\begin{align*}
v_n:= \min_{\xi : \partial \Lambda_n \to \{-1, 0, +1\}} \bbE^{\xi}_{\Lambda_n} [h(0)^2],
\end{align*}
where the minimum is taken over all functions $\xi: \partial \Lambda_n \to \{-1, 0, +1\}$ that take odd values on odd faces and even on even faces (all such $\xi$ are admissible boundary conditions).

\begin{lemma}\label{lem:v_n}
   Fix $\mathbf{c} \in [1, 2]$. There exist $R\ge 1$ such that for every $n$ large enough
    \begin{align*}
    	v_{Rn} \geq v_n + 1.
    \end{align*}
\end{lemma}

Before proving this lemma, let us explain how it implies the lower bound in Corollary~\ref{cor:variance_D}.

\begin{proof}[the lower bound in Corollary~\ref{cor:variance_D}]
We will suppose hereafter that $x= 0$.
Let $D$ be a discrete domain $D$ containing the box $\Lambda_n$ and $\xi$ be some boundary condition on $\partial D$ with $|\xi|\le \ell$.
Using~\eqref{eq:CBC-h} and $ |\bbE^{\xi + \ell}_D [h(0)]|\le 2\ell$ (by Corollary~\ref{cor:E[h]>0}), we get that 
\begin{align} \label{eq:adaa}
	\mathrm{Var}^\xi_D [h(0)] = \mathrm{Var}^{\xi+ \ell}_D [h(0)]& \geq   \bbE^{\xi + \ell}_D [h(0)^2]-4\ell^2.\end{align}
	Now, $\xi+\ell$ is of definite sign and we may apply
\eqref{eq:CBC-|h|} and~\eqref{eq:FKG-|h|} to find
\begin{align} \mathrm{Var}^\xi_D [h(0)]
	&\geq \bbE^{0,1}_D [h(0)^2]-4\ell^2  \geq \bbE^{0,1}_D [h(0)^2\,|\, |h| \leq 1 \text{ on $D \setminus \La_n$}]-4\ell^2.\nonumber
\end{align}
By the spatial Markov property, the last expectation value above is an average of quantities $\bbE^{\xi}_{\La_n} [h(0)^2]$ over boundary conditions $\xi$ with values in $\{-1,0,1\}$. As such, it is bounded from below by $v_n$. 

It is an immediate consequence of Lemma~\ref{lem:v_n} that $v_{n}\ge c \log n$ for some constant $c>0$ and all $n \geq 1$. 
Since $n$ may be chosen as the distance from $0$ to $\partial D$, this concludes the proof. 
\end{proof}

The rest of the section is dedicated to proving Lemma~\ref{lem:v_n}. 
We start by stating a consequence of Theorem~\ref{thm:bxp} which may also be of independent interest.  

For integers $N\ge n> 0$, recall that $A(n,N):=\Lambda_N \setminus \Lambda_{n}$ and 
$\calO_{h \geq k} (n,N)$ (resp.~$\calO_{|h| \geq k} (n,N)$) is the event that there exists a path of $h\ge k$ (resp.~$|h|\ge k$) in $A(n,N)$ forming a circuit  around $0$. 

\begin{lemma}\label{lem:circ_increase}
    Fix $\mathbf{c} \in [1, 2]$. For every $k \geq 0$, there exist $c, C, n_0 >0$ such that for all $N/2\ge n>n_0$, 
    \begin{align*}
    \bbP^{0,1}_{ \Lambda_N} [\calO_{ \vert h \vert \geq k} (n, N)] \geq 1- C(n/N)^c.
    \end{align*}
\end{lemma}

The necessity of the lemma comes from the fact that for the proof of Lemma~\ref{lem:v_n}, it does not suffice to show that circuits of a given height occur with positive probability in annuli (which is the conclusion of Theorem~\ref{thm:bxp}); we need circuits to occur with high probability, when the ratio between the inner and outer radii of the annulus is large. 

\begin{proof} 
Let us denote $\bbP^{0,1}_{N} :=\bbP^{0,1}_{\Lambda_N}$ for simplicity.
Below, we show  by induction that there exists $\delta=\delta(k)>0$  that for every $n>n_0(k)$ and $i\ge1$,
\begin{equation}
\label{eq:inductive hypothesis}
\bbP^{0,1}_{2^in} [ \calO_{ \vert h \vert \geq k} ( n, 2^in)^c ] \le (1-\delta)^i.
\end{equation}
The claim for $N=2^i n$ then directly follows from~\eqref{eq:inductive hypothesis}. To treat general $2^i n \leq N < 2^{i+1}n$, compute
\begin{align*}
\bbP_{N}^{0,1} [\calO_{ \vert h \vert \geq k} ( n, N) ] 
& \geq
\bbP_{N}^{0,1}  [\calO_{ \vert h \vert \geq k} ( n, 2^i n ) ] 
\qquad \text{(by inclusion)}
\\
& \geq
\bbP_{N}^{0,1}  [\calO_{ \vert h \vert \geq k} ( n, 2^i n )  \; \vert \; \vert h \vert \leq 1 \text{ on } \partial \Lambda_{ 2^i n } ]
\qquad \text{(by~\eqref{eq:FKG-|h|})}
\\
& \geq 
\min_{\xi: \partial \Lambda_{ 2^i n } \to \{ 0 , \pm 1 \} }
\bbP_{2^i n }^{\xi + 2 }  [\calO_{ \vert h \vert \geq k + 2} ( n, 2^i n )  ]
\qquad \text{(by~\eqref{eq:SMP})}
\\
& \geq 
\bbP_{2^i n }^{0,1}  [\calO_{ \vert h \vert \geq k + 2} ( n, 2^i n )  ]
\qquad \text{(by~\eqref{eq:CBC-|h|}),}
\end{align*}
and the claim follows from the case of $N=2^i n$ by adjusting $k$. We thus turn to the proof of~
\eqref{eq:inductive hypothesis}.

For $i=1$, using the inclusion of events in the first inequality, Theorem~\ref{thm:bxp}  implies that for $n>n_0$,
\begin{align*}
\bbP^{0,1}_{2 n} & [\calO_{ \vert h \vert \geq k} (n, 2 n)] 
\geq
 \bbP^{0,1}_{2n} [\calO_{ h \geq  k+2} (n, 2n)] \ge \delta
\end{align*}
for some constant $\delta > 0$ depending on $k$ only, and which we now fix.
Let us now assume that~\eqref{eq:inductive hypothesis} holds true for $i-1$ and then prove it for $i$. 
By inclusion of events and conditioning, we get
\begin{align*}
\bbP^{0,1}_{2^in}  [\calO_{ \vert h \vert \geq k} (n, 2^in)^c]
& \leq
\bbP^{0,1}_{2^in}  [ \calO_{ \vert h \vert \geq k} (2n, 2^in)^c \cap \calO_{ \vert h \vert \geq k} (n, 2n)^c] \\
&=\bbP^{0,1}_{2^in}  [\calO_{ \vert h \vert \geq k} (2n, 2^in)^c]\,
\underbrace{\bbP^{0,1}_{2^in}  [\calO_{ \vert h \vert \geq k} (n, 2n)^c\,|\,\calO_{ \vert h \vert \geq k} (2n, 2^in)^c]}_{P}.
\end{align*}
Using the inductive hypothesis, it thus suffices to show that $P\le 1-\delta$. Now, since $\calO_{ \vert h \vert \geq k} (2n, 2^in)^c$ depends only on $|h|$ on $A(2n,2^in)$, one may further condition on the precise value of $|h|$ in $A(2n-1,2^in)\supset A(2n,2^in)$. The measure thus obtained involved only conditioning on $|h|$, except on $\partial\Lambda_{2^in}$, where we have $h \in \{0,1\}$. We can therefore use FKG for $|h|$ to deduce that
\begin{align*}
P 
&\leq \bbP^{0,1}_{2^in}  [\calO_{ \vert h \vert \geq k} (n, 2n)^c\,|\,|h(x)| \le 1,\forall x\in A(2n-1,2^in)] \\
& \leq \bbP^{0,1}_{2^in}  [\calO_{ h \geq k} (n, 2n)^c\,|\,|h(x)| \le 1,\forall x\in A(2n-1,2^in)]\\
& \le \bbP^{0,-1}_{2n}  [\calO_{ h \geq k} (n, 2n)^c ]\\
& = 1- \bbP^{0,1}_{2n}  [\calO_{ h \geq k+2} (n, 2n) ] \\
& \leq 1 - \delta,
\end{align*}
where the additional manipulations were based on inclusion of events, spatial Markov property and comparison of boundary conditions, shift of boundary conditions, and our choice of $\delta$ above, 
respectively.
\end{proof}

\begin{proof}[Lemma~\ref{lem:v_n}]
Fix $k = 4$ and let $R > 1$ be such that 
\begin{align}\label{eq:Rchoice}
	 \bbP^{0,1}_{\Lambda_{Rn}} [\calO_{|h|\ge k}(n,{Rn})] \geq 1/2,
\end{align}
for all $n$ large enough.  

Fix $n \ge 1$ large enough for~\eqref{eq:Rchoice} to hold 
and let $\xi$ be a boundary condition on $\partial \La_{Rn}$ taking values in $\{-1,0,1\}$ 
that minimises $\bbE^{\xi}_{\Lambda_{Rn}} [h(0)^2]$. 
By symmetry, we may choose $\xi$ so that $\bbE^{\xi}_{\Lambda_{Rn}} [h(0)] \leq 0$.
Then, we have
\begin{align}
\nonumber
	v_{Rn} = \bbE^{\xi}_{\Lambda_{Rn}} [h(0)^2] 
	&\geq \bbE^{\xi}_{\Lambda_{Rn}} [(h(0)+2)^2] - 4 
	\\
	\nonumber
    &= \bbE^{\xi+2}_{\Lambda_{Rn}} [h(0)^2] - 4
    \\
	& \geq \bbE^{0,1}_{\Lambda_{Rn}} [h(0)^2] - 4,\label{eq:circ_no_circ0}
\end{align}
where the last step used~\eqref{eq:CBC-|h|}.

Hereafter we focus on bounding $\bbE^{0,1}_{\Lambda_{Rn}} [h(0)^2]$.
We have
\begin{align}\label{eq:circ_no_circ1}
	\bbE^{0,1}_{\Lambda_{Rn}} [h(0)^2] 
	 = \bbE^{0,1}_{\Lambda_{Rn}} [h(0)^2 \ind_{\calO_{|h|\ge k}(n,{Rn})}] + \bbE^{0,1}_{\Lambda_{Rn}} [h(0)^2 \ind_{\calO_{|h|\ge k}(n,{Rn})^c}],
\end{align}
and we will bound separately the two terms on the right-hand side of the above.

If $\calO_{|h|\ge k}(n,{Rn})$ occurs, let $\Gamma$ be the outer-most circuit with $|h|\geq k$ around $\Lambda_{Rn}$.
Write $\calD$ for the random domain formed of the faces on or surrounded by $\Gamma$. 
Notice that $\calD$ is measurable in terms of the values of $|h|$ on $\Gamma = \partial \calD$ and $\calD^c$. 
As such, the measure in $\calD$ is $\bbP^{\zeta}_\calD$, with $\zeta$ taking values either
$k$ and $k+1$ or $-k$ and $-k-1$. Thus
\begin{align}
	\bbE^{0,1}_{\Lambda_{Rn}} [h(0)^2 \ind_{\calO_{|h|\ge k}(n,{Rn})}] 
	&= \sum_{D'}	 \bbE^{k,k+1}_{D'} [h(0)^2] \bbP^{0,1}_{\Lambda_{Rn}} [\calD = D'] \nonumber\\
	&= \sum_{D'}	 \bbE^{0,1}_{D'} [(h(0)+k)^2] \bbP^{0,1}_{\Lambda_{Rn}} [\calD = D'] \nonumber\\
	&\geq k^2 \bbP^{0,1}_{\Lambda_{Rn}} [\calO_{|h|\ge k}(n,{Rn})]  + \sum_{D'}	 \bbE^{0,1}_{D'} [h(0)^2] \bbP^{0,1}_{\Lambda_{Rn}} [\calD = D'] \nonumber\\
	&\geq (k^2 + v_n)\bbP^{0,1}_{\Lambda_{Rn}} [\calO_{|h|\ge k}(n,{Rn})].  \label{eq:circ_no_circ2}
\end{align}
In the first equality, we used the symmetry $h\leftrightarrow -h$ and in the first inequality the positivity of $\bbE^{0,1}_{D'} [h(0)]$ (see Corollary~\ref{cor:E[h]>0}).
In the last inequality, we used~\eqref{eq:FKG-|h|} to bound $\bbE^{0,1}_{D'} [h(0)^2]$ by $v_n$, in the same way as after~\eqref{eq:adaa}. 

We turn to the second term of~\eqref{eq:circ_no_circ1}. This term is an average of quantities of the type 
	$\bbE^{0,1}_{\Lambda_{Rn}} [h(0)^2 \,|\, |h| = \zeta \text{ on $\La_n^c$}],$
where $\zeta$ runs through all values of $|h|$ outside $\La_n^c$ such that $\calO_{|h|\ge k}(n,{Rn})$ fails. 
Notice that by~\eqref{eq:FKG-|h|}, for any such $\zeta$, 
\begin{align*}
	\bbE^{0,1}_{\Lambda_{Rn}} [h(0)^2 \,|\, |h| = \zeta \text{ on $\La_n^c$}]
	\geq \bbE^{0,1}_{\Lambda_{Rn}} [h(0)^2 \,|\, |h| = 0 \text{ or } 1 \text{ on $\La_n^c$}]
	\geq v_n.
\end{align*}
In conclusion, 
\begin{align}\label{eq:circ_no_circ3}
	\bbE^{0,1}_{\Lambda_{Rn}} [h(0)^2 \ind_{\calO_{|h|\ge k}(n,{Rn})^c}] \geq v_n	\bbP^{0,1}_{\Lambda_{Rn}} [\calO_{|h|\ge k}(n,{Rn})^c].
\end{align}

Inject now~\eqref{eq:circ_no_circ2} and~\eqref{eq:circ_no_circ3} into~\eqref{eq:circ_no_circ1}, then use~\eqref{eq:circ_no_circ0} to conclude that
\begin{align*}
	v_{{Rn}} \geq v_n + k^2\, \bbP^{0,1}_{\Lambda_{Rn}} [\calO_{|h|\ge k}(n,{Rn})] - 4.
\end{align*}
Due to~\eqref{eq:Rchoice} and the fact that $k=4$, the right hand side is larger than $v_n +1$. 
\end{proof}

\subsubsection{Proof of the lower bound in Theorem~\ref{thm:variance_torus}}

Fix $N$ and $x, y \in F(\bbT_{N})$. 
Fix a representative of the equivalence class of each homomorphism by setting $h(x)=0$.
Using the FKG inequality for $|h|$ (recall that it does indeed hold for the balanced six-vertex model on the torus) we find 
\begin{align*}
	\mathbb E^{(\mathrm{bal})}_{\bbT_N}[(h(y)-h(x))^2]
	&=\mathbb E^{(\mathrm{bal})}_{\bbT_N}[h(y)^2|h(x)=0]\\
	&\ge \mathbb E^{(\mathrm{bal})}_{\bbT_N}[h(y)^2\,|\,|h(u)|\le 1\text{ for every }u\notin \Lambda_{ \lfloor d (x,y)/2 \rfloor }(y)]\\
	&\ge \min_{|\xi|\le 1}\mathbb E^\xi_{\Lambda_{ \lfloor d (x,y)/2 \rfloor }}[h(y)^2]\\
	&\ge c\log ( d (x,y)/2 ) .
\end{align*}
In the second inequality we used the spatial Markov property and in the third Lemma~\ref{lem:v_n}.
The lower bound of Theorem~\ref{thm:variance_torus} may be obtained by adapting the constant $c$. 

\subsection{Upper bounds}

In this section we prove the logarithmic upper bounds for the variance of Corollary~\ref{cor:variance_D} and Theorem~\ref{thm:variance_torus}. 
We start in Section~\ref{sec:variance_ub_D_sc} with the upper bound of Corollary~\ref{cor:variance_D} for simply-connected  domains.
The case of the torus (Theorem~\ref{thm:variance_torus}) is very similar to that of
simply connected domains, but with additional technical difficulties. We sketch it in Section~\ref{sec:variance_ub_torus}. 
Finally, the case of non simply-connected domains follows easily from the result on the torus, as shown in Section~\ref{sec:variance_ub_D}.
 
\subsubsection{The upper bound of Corollary~\ref{cor:variance_D} for simply-connected domains}\label{sec:variance_ub_D_sc}

We start by defining the counterpart of the quantity $v_n$ of Section~\ref{sec:variance_lb_D}. For $n\geq 1$, set 
\begin{align*}
	w_n:= \sup_{ \partial D \cap \Lambda_n \neq \emptyset  } \bbE^{0, 1}_{D} [h(0)^2],
\end{align*}
where the supremum is taken over simply-connected discrete domains $D$ with $ \partial D \cap \Lambda_n \neq \emptyset$.

\begin{lemma}\label{lem:w_n}
	Fix $\mathbf{c} \in [1, 2]$. There exists $C > 0$ such that for all $n\ge1$,
	\begin{align}\label{eq:w_n}
		w_{2n} \leq w_{n} + C.
	\end{align}
\end{lemma}

Let us show how the above implies the upper bound in  Corollary~\ref{cor:variance_D} for simply-connected domains.

\begin{proof}[the upper bound in  Corollary~\ref{cor:variance_D} for simply-connected domains]
    We may assume $x = 0$. Fix a simply connected domain $D$ containing $0$ and a boundary condition $\xi$ with $|\xi| \leq \ell$.
    Let $n$ be the distance from $0$ to $D^c$. We have
    \begin{align*}
    	\mathrm{Var}_D^{ \xi } (h (0))  = \mathrm{Var}_D^{ \xi } (h (0) + \ell) 
    	 \leq  \bbE^\xi_D  [ (h (0) + \ell)^2] =   \bbE^{\xi+ \ell}_D  [ h (0) ^2]. 
	\end{align*}
	Then,~\eqref{eq:CBC-|h|} and Corollary~\ref{cor:E[h]>0} imply that
	\begin{align*}
    	 \bbE^{\xi+ \ell}_D  [ h (0) ^2]  \leq \bbE^{2\ell, 2 \ell - 1}_D  [ h (0) ^2]  \leq \mathrm{Var}^{2\ell, 2 \ell - 1}_D  ( h (0) ) + 4 \ell^2 = \mathrm{Var}^{0, 1}_D  ( h (0) ) + 4 \ell^2 \leq w_n + 4\ell^2.
    \end{align*}
	Finally, it is a direct consequence of Lemma~\ref{lem:w_n} that $w_n \leq C \log n$ for some constant $C$ and $n \geq 2$ and the claim thus follows from the previous two displayed equations. 
\end{proof}

To prove Lemma~\ref{lem:w_n}, we will use the following result which may also be of independent interest. 

\begin{lemma}\label{lem:circ_decrease}
    Fix $\mathbf{c} \in [1, 2]$. 
    There exist $c, C > 0$ such that for all $k$ and $n$ and any simply connected domain $D$ containing $\La_n$ but not $\La_{2n}$, 
	\begin{align}\label{eq:circ_decrease}
        \bbP^{0,1}_{D} [\partial D \xlra{|h| \leq k} \La_n] \ge 1 - C e^{-ck}.
    \end{align}
\end{lemma}

\begin{remark}
It is useful to adopt the dual view of Remark~\ref{rem:dual_crossings} to Lemma~\ref{lem:circ_decrease}: equivalently
\begin{align*}
\bbP^{0,1}_{D} [\calO^{\times}_{|h| \geq k + 1} (n) ] \leq  C e^{-ck},
\end{align*}
where $\calO^{\times}_{|h| \geq k+1} (n)$ denotes the event that there exists a $\times$-circuit of $|h| \geq k+1$ in $D$ that winds around~$\Lambda_n$. 
\end{remark}

\begin{proof}
First, by the union bound and~\eqref{eq:CBC-h}
\begin{align*}
    \bbP^{0,1}_{D} [\calO^{\times}_{|h| \geq k} (n) ] 
    \leq \bbP^{0,1}_{D} [\calO^{\times}_{ h \geq k} (n) ]  
    + \bbP^{0,1}_{D} [\calO^{\times}_{ h \leq -k} (n) ] \leq 2 \bbP^{0,1}_{D} [\calO^{\times}_{ h \geq k} (n) ].
\end{align*}
We will prove that for some universal constant $c >0$ to be chosen below
\begin{align}\label{eq:the_induction}
	\bbP^{0,1}_{D} [\calO^{\times}_{ h \geq 2k} (n) ] \leq e^{-ck},
\end{align}
for all $k \geq 0$, by induction on $k$.
The statement is trivial for $k =0$, and we focus on the inductive step. 
Assume that \eqref{eq:the_induction} holds for some integer $k$.

When $\calO^{\times}_{ h \geq 2k} (n)$ occurs, let $\mathcal{Q}$ be the random discrete domain formed of faces inside the exterior-most $\times$-loop of $h=2k$, and the faces sharing a corner with this interior. Then,
\begin{align}\label{eq:begin induction proof}
    \bbP^{0,1}_{D} [\calO^{\times}_{ h \geq 2k + 2} (n) \; \vert \; \calO^{\times}_{ h \geq 2k} (n) \text{ and } \mathcal{Q} = Q ]
    & = \bbP^{2k,2k-1}_{Q} [\calO^{\times}_{ h \geq 2k + 2} (n) ] \\ 
    &= \bbP^{1, 2}_{Q} [\calO^{\times}_{ h \geq 4} (n) ] .   \nonumber
\end{align}

\begin{figure}
	\begin{center}
	\includegraphics[width=0.55\textwidth]{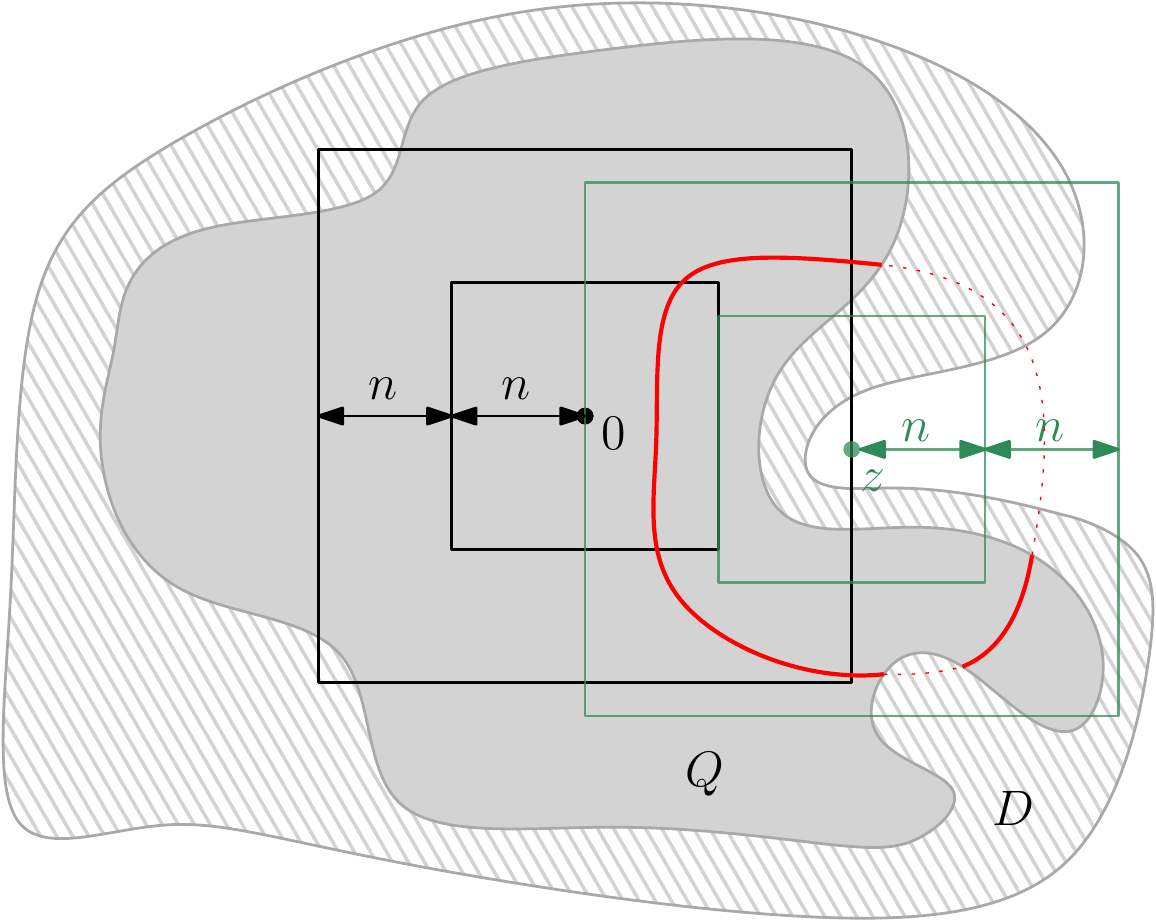}
	\caption{The event $\calQ = Q$ is determined by the value of $h$ on  $D\setminus Q$. 
	For $z$ as above, any circuit disconnecting $\partial D$ from $\Lambda_n$ must cross the annulus $A(n, 2n) + z$ from inside to outside. 
	By duality, any circuit in $\La_{2n}(z)$ surrounding $\La_n(z)$ induces a crossing between $\partial Q$ and~$\La_n$.}
		\label{fig:disconnecting_circuits_planar}
	\end{center}
\end{figure}

Fix now any $z \in \bbZ^2 $ on the boundary of $ \Lambda_{2n}$ (viewed as a continuous domain) and not inside $D$; such a $z$ exists as $D$ does not contain $\Lambda_{2n}$. Remark that any circuit contained in $Q$ and which surrounds $\Lambda_n$ necessarily includes a path between $\partial\Lambda_n(z)$ and $\partial \Lambda_{2n}(z)$ --  see Figure~\ref{fig:disconnecting_circuits_planar}.
Hence, we have
\begin{align*}
 \bbP^{1, 2}_{Q} [\calO^{\times}_{ h \geq 4} (n) ]
 & \leq \bbP^{1, 2}_{Q} [\partial \Lambda_n (z) \xlra{h \geq 4}_{\times} \partial \Lambda_{2n} (z) ].
\end{align*}
Let $R$ be a simply connected domain such that $Q\cup \Lambda_{2n}(z)\subset R$.
Write $\calO_{h \leq 3} (A(n, 2n) + z)$ for the event that there exists a circuit 
of faces of height at most $3$ in $\La_{2n}(z)$ that surrounds $\La_n(z)$. 
By duality (Remark~\ref{rem:dual_crossings}) and \eqref{eq:SMP} we then have
\begin{align*}
  	\bbP^{1, 2}_{Q} [\partial \Lambda_n (z) \xlra{h \geq 4}_{\times} \partial \Lambda_{2n} (z) ]
	=   1 - \bbP^{1,2}_{R} [\calO_{h \leq 3} (A(n, 2n) + z) \,|\, h \in \{1,2\} \text{ on } R\setminus Q^\circ],
\end{align*}
where $Q^\circ = Q \setminus \partial Q$ is the interior of $Q$.
Let $\xi$ be the maximal boundary condition on $\partial \La_{2n}(z)$ that takes values $\{1,2\}$ on $\partial \Lambda_{2n}(z) \setminus Q^\circ$, and that is smaller or equal to $6$ overall. 
Then, by Proposition~\ref{prop:pushing_1} applied to $-h$, 
\begin{align*}
	&\bbP^{1,2}_{R} [\calO_{h \leq 3} (A(n, 2n) + z)  \,|\, h \in \{1,2\} \text{ on } R\setminus Q^\circ] \\
	& \qquad \geq \bbP^{\xi}_{ \Lambda_{2n}(z) } [\calO_{h \leq 3} (A(n, 2n) + z)  \,|\, h \in \{1,2\} \text{ on } \La_{2n}(z)\setminus Q^\circ].
\end{align*}
Observe that the condition of Proposition~\ref{prop:pushing_1} that allows to remove the  multiplicative factor $2$ is indeed satisfied, 
as any path realising $\calO_{h \leq 3} (A(n, 2n) + z)$ must intersect  $\La_{2n}(z)\setminus Q^\circ$.
Applying duality on both sides of the previous inequality, we conclude that
\begin{align*}
  	\bbP^{1, 2}_{Q} [\partial \Lambda_n (z) \xlra{h \geq 4}_{\times} \partial \Lambda_{2n} (z) ]
	&\leq  \bbP^{\xi}_{\La_{2n}(z)} [\partial \Lambda_n (z) \xlra{h \geq 4}_{\times} \partial \Lambda_{2n} (z) \,|\, h \in \{1,2\} \text{ on } \La_{2n}(z)\setminus Q^\circ ]\\
	&\leq \bbP^{\xi}_{\La_{2n}(z)} [\partial \Lambda_n (z) \xlra{h \geq 3} \partial \Lambda_{2n} (z) \,|\, h \in \{1,2\} \text{ on } \La_{2n}(z)\setminus Q^\circ ].
\end{align*}
Using \eqref{eq:FKG-h} and \eqref{eq:FKG-|h|}, we find 
\begin{align*}
	\bbP^{\xi}_{\La_{2n}(z)}& [\partial \La_n (z) \xlra{h \geq 3} \partial \La_{2n} (z) \,|\, h \in \{1,2\} \text{ on } \La_{2n}(z)\setminus Q^\circ]\\
&=\bbP^{\xi-2}_{\La_{2n}(z)}[\partial \La_n (z) \xlra{h \geq 1} \{ \partial  \La_{2n} (z) \cap \xi \geq 3 \} \,\big|\, h \in \{-1,0\} \text{ on } \La_{2n}(z)\setminus Q^\circ]\\
&	\leq \bbP^{\xi-2}_{\La_{2n}(z)}[\partial \La_n (z) \xlra{|h| \geq 1} \{ \partial  \La_{2n} (z) \cap \xi \geq 3 \} \,|\,|h|\le 1\text{ on } \La_{2n}(z)\setminus Q^\circ]\\
&	\leq \bbP^{\xi-2}_{\La_{2n}(z)} [\partial \La_n (z) \xlra{|h| \geq 1} \{ \partial  \La_{2n} (z) \cap \xi \geq 3 \} ]\\
&	= \bbP^{\xi-2}_{\La_{2n}(z)} [\partial \La_n (z) \xlra{h \geq 1} \partial \La_{2n} (z)],
\end{align*}
{where in the second step, as well as in the last one, we observed that for any path realising the event for $|h|$, the sign of $h$ on that path is determined to be $+$ by the boundary condition.}
Finally, duality allows us to bound the above as 
\begin{align*}
 	\bbP^{\xi-2}_{\La_{2n}(z)} [\partial \Lambda_n (z) \xlra{h \geq 1} \partial \Lambda_{2n} (z)]
  	\leq 1 - \bbP^{3,4}_{\La_{2n}} [\calO_{h \leq 0}^\times (n)]
	\leq 1 - \bbP^{0,1}_{\La_{2n}} [\calO_{h \geq 4} (n)]
	\leq e^{-c},
\end{align*}
where $c > 0$ is independent of $n$ and is generated by Theorem~\ref{thm:bxp}.
Summarizing the chain of inequalities starting from~\eqref{eq:begin induction proof}, we have
\begin{align*}
\bbP^{0,1}_{D} [\calO^{\times}_{ h \geq 2k + 2} (n) \; \vert \; \calO^{\times}_{ h \geq 2k} (n) \text{ and } \mathcal{Q} = Q ]
\leq e^{-c}
\end{align*}
for all $Q$. Using the induction hypothesis and averaging over $Q$, we conclude that \eqref{eq:the_induction} also holds for $k+1$, and thus for all $k$. This implies~\eqref{eq:circ_decrease} after adjusting the constants.
\end{proof}

\begin{proof}[Lemma~\ref{lem:w_n}]
Let $D$ be a simply-connected domain such that $\partial D \cap \Lambda_{2n} \neq \emptyset$.
Define the random variable
\begin{align*}
	K := \inf \{ k \geq 1\,:\, \partial D \xlra{|h| \leq k} \La_n \}.
\end{align*}

Denote by
$\sfC_{k}$ the connected component of faces with $|h| \leq k$ of $\partial D$, and the $\times$-circuits of $\vert h \vert = k+1$ bounding them.
Then, $\sfC_k$ may be determined by only exploring the faces in it.
Explore $\sfC_K$ by revealing $\sfC_1$ then $\sfC_2$, etc, until the first cluster that reaches $\La_n$.  
Write $\Omega$ for the faces in $D\setminus \sfC_K$ or sharing a corner with a face in $D\setminus \sfC_K$.
Then,
\begin{align}\label{eq:KQ}
	\bbE^{0,1}_{D} [h(0)^2] 
	= \sum_{(Q, \zeta)} \bbP^{0,1}_{D} [h(0)^2 \,|\, \Omega = Q,\, h = \zeta \text{ on $\sfC_K$}]  
	\bbP^{0,1}_{D} [\Omega = Q,\, h = \zeta \text{ on $\sfC_K$}],
\end{align}
where the sum runs over all the possible realizations of $(\Omega, h_{|\sfC_K})$.
When $0 \not \in \Omega$, we have $h(0)^2\leq K^2$. 
Fix now $(Q,\zeta)$ such that $0 \in \Omega$. 
Write $k$ for the value of $K$ in the realization $\zeta$.
Then the values of $\zeta$ on $\partial Q$ are either $k$, $k+1$, $-k$ or $-k-1$.  
The sign of the boundary conditions may depend on the connected component of $Q$, however the quantity of interest, $h(0)^2$, is invariant under sign flip. Hence we can as well assume that $\zeta$ is positive on $\partial Q$.
Finally, observe that, due to the definition of $K$, $Q$ necessarily intersects $\La_n$.
Then,
\begin{align*}
    \bbE^{0,1}_{D} [h(0)^2 \,|\, \Omega = Q \text{ and } h = \zeta \text{ on $\sfC_K$}]  
    &= \bbE^{k,k+1}_{Q} [h(0)^2]\\
    &= \bbE^{0,1}_{Q} [(h(0)+k)^2]\\  
    &= \bbE^{0,1}_{Q} [h(0)^2]  + 2k \bbE^{0,1}_{Q} [h(0)]  + k^2 \\
    &\leq w_n  + 2k  + k^2.
\end{align*}
Plugging the above into~\eqref{eq:KQ}, we find 
\begin{align*}
	\bbE^{0,1}_{D} [h(0)^2] \leq  w_n  + \bbE^{0,1}_{D} [2K + K^2].
\end{align*}
Finally, Lemma~\ref{lem:circ_decrease} implies that $\bbE^{0,1}_{D} [2K + K^2] \leq C$ 
for some constant $C>0$ which is independent of $n$ or $D$. 
This proves~\eqref{eq:w_n}.
\end{proof}


\subsubsection{The upper bound of Theorem~\ref{thm:variance_torus}}\label{sec:variance_ub_torus}

\newcommand{\diam}{{\rm diam}}

Throughout this proof we fix $\mathbf{c}$ and $N$, and operate on the torus $\bbT_ {N} =: \bbT $.
For $B \subset F(\bbT )$, denote
\begin{align*}
\bbE^{0, 1}_{B^c} [ \;  \cdot \;  ]  = 
\bbE_{\bbT }^{\mathrm{(\mathrm{bal})}} [ \;  \cdot \; \vert \; h_{\vert B} \in \{ 0 , 1\}].
\end{align*}
For $u \in F(\bbT)$ and $1 \leq n \leq N/2$, write $\La_n(u)$ for the lift of 
$\La_n$ to $\bbT$, translated so that it is centered at the bottom-left corner of $u$. 

Let $x, y \in F(\bbT)$ be the faces appearing in the statement.
Due to the triangular inequality, it suffices to prove the bound for $d(x,y) \leq N/16$, and we will assume this henceforth. 
Write $d = d(x,y)$ and for simplicity assume that $d$ is a power of $2$ (small adaptations allow to overcome this assumptions). 

In analogy to Section~\ref{sec:variance_ub_D_sc}, for $n \leq N/4$, define
\begin{align*}
	w_n &:= \sup \{ \bbE^{0, 1}_{B^c} [h(x)^2] \,: B \subset F(\bbT) \text{ connected, intersecting $\La_n(x)$, with diameter $\geq 4n$ }\} \\
	u_n &:=\sup \{ \bbE^{0, 1}_{B^c} [h(x)^2] \,:\,B \subset F(\bbT) \text{ connected, containing $y$, and intersecting $\partial\La_n(y)$} \}.
\end{align*}
The result below controls the growth of $w_n$ and $u_n$, similarly to Lemma~\ref{lem:w_n} in the previous section. 

\begin{lemma}\label{lem:w_n torus}
	Fix $\mathbf{c} \in [1,2]$. There exists $C > 0$ such that for all $N$ and all $x, y \in \bbT_N$ with $d (x,y) \leq N/16$, we have
	\begin{align}
	w_{ 2n} &\leq w_{ n } + C \qquad \text{for all } n \leq N/8,  \label{eq:w_n2}\\
	u_{ n} &\leq u_{ 2n } + C \qquad \text{for all } n \leq N/8,  \label{eq:u_n2}\\
	u_{4d}  &\leq w_{d}.\label{eq:uw}
	\end{align}
\end{lemma}

Before outlining the proof of this lemma, let us see how it implies the upper bound of Theorem~\ref{thm:variance_torus}.

\begin{proof}[the upper bound of Theorem~\ref{thm:variance_torus}]
	By the definition of $u_1$, we have 
	\begin{align*}
		\bbE_{\bbT_{N}}^{\mathrm{(\mathrm{bal})}} [ ( h(x)-h(y) )^2 ]  
		\leq u_{1} 
		\stackrel{\eqref{eq:u_n2}}{\leq} u_{4d} + C \log 4d
		\stackrel{\eqref{eq:uw}}{\leq} w_{d} + C \log 4d
		\stackrel{\eqref{eq:w_n2}}{\leq} w_{1} + 2C \log 4d .
	\end{align*}
	Since $w_1\leq 2$, the desired bound is attained. 
\end{proof}

\paragraph{Proof outline for Lemma~\ref{lem:w_n torus}:}
The relations \eqref{eq:w_n2} and \eqref{eq:u_n2} are proved in the same way as in Lemma~\ref{lem:w_n}
and hinge on the following two statements (which correspond to Lemma~\ref{lem:circ_decrease}).
\begin{itemize}
	\item There exist $c, C > 0$ such that for all $k$ and $n \leq N/8$ and any 
	$B \subset F(\bbT)$ connected, intersecting $\La_{2n}(x)$ and with diameter at least $8n$, 
	\begin{align*}
        \bbP^{0,1}_{B^c} [B \xlra{|h| \leq k} \La_n(x)] \ge 1 - C e^{-ck}.
    \end{align*}
	\item There exist $c, C > 0$ such that for all $k$ and $n \leq N/8$ and any 
	$B \subset F(\bbT)$ connected, with $y \in B$ and intersecting $\partial\La_{n}(y)$, 
	\begin{align*}
        \bbP^{0,1}_{B^c} [B \xlra{|h| \leq k} \partial\La_{2n}(y)] \ge 1 - C e^{-ck}.
    \end{align*}
\end{itemize}
Both of these statements are proved in the same way as Lemma~\ref{lem:circ_decrease}. 

Finally \eqref{eq:uw} follows directly from the definition of $u_n$ and $w_n$, 
since any set appearing in the supremum defining $u_{4d}$ also appears in that defining $w_d$.
\hfill $\square$

\subsubsection{The upper bound of Corollary~\ref{cor:variance_D} for arbitrary domains}\label{sec:variance_ub_D}

Fix a finite planar domain $D$, a face $x$ of $D$, and a boundary condition $\xi$ on $\partial D$ with $|\xi| \leq \ell$ for some $\ell$. By two trivial steps and then~\eqref{eq:CBC-|h|} 
\begin{align*}
	\mathrm{Var}_D^{ \xi } (h (x)) = \mathrm{Var}_D^{ \xi + \ell } (h (x))  
    \leq  \bbE^{\xi + \ell}_D  [ h (x) ^2]
    \leq  \bbE^{\zeta + 2\ell }_D  [ h (x)^2]
    = \mathrm{Var}_D^{\zeta + 2\ell } ( h(x) ) + \bbE^{\zeta + 2\ell }_D  [ h (x)]^2,
\end{align*}
for any boundary conditions $\zeta$ taking values $-1$, $0$ and $1$, and with the same parity as $\xi + \ell$.
Let $\pm \zeta$ be the condition minimizing $\bbE^{\zeta  }_D  [ h (x)^2]$, with the sign chosen so that $\bbE^{\zeta  }_D  [ h (x)] \leq 0$; whence, by Corollary~\ref{cor:E[h]>0} and the above, we have
\begin{align*}
	\mathrm{Var}_D^{ \xi } (h (x)) \leq
    \mathrm{Var}_D^{\zeta  } ( h(x) ) + 4 \ell ^2.
     \end{align*}
     
     Let $y$ be the even face of $\partial D$ closest to $x$; note that thus $d(x, y) \leq d(x, \partial D) + 1$.
Furthermore, embed $D$ in the  torus $\bbT_{N}$ for some $N$ larger than twice the diameter of $D$, and for $\bbE_{\bbT_N}^{({\rm bal})}$ normalize height functions by $h(y)=0$. Using the choice of $\zeta$ above and the embedding of $D$ in $\bbT_N$, we have
\begin{align*}
	\Var_D^\zeta ( h(x) ) 
	&\leq \bbE_{\bbT_N}^{({\rm bal})}[h(x)^2 \,|\, |h|\leq 1 \text{ on $\partial D$}] \\ 
	&\leq \bbE_{\bbT_N}^{({\rm bal})}[h(x)^2 \,|\, h(y)= 0]&\text{ by~\eqref{eq:FKG-|h|}}\\
	&= \bbE_{\bbT_N}^{({\rm bal})}[(h(x) - h(y))^2].
\end{align*}
By Theorem~\ref{thm:localisation}, the latter is bounded by $C \log d_{\bbT_N} (x,y)$, where $d_{\bbT_N} (x,y)$ is the distance between $x$ and $y$, when embedded in the torus. 
Notice however that, due to our choice of $N$ and $y$, $d_{\bbT_N} (x,y) = d(x, y) \leq d(x, \partial D) + 1$. The claim follows by adjusting $C$.
\hfill $\square$

\paragraph*{Acknowledgements}
This research was funded by an IDEX Chair from Paris Saclay, by the NCCR SwissMap from the Swiss NSF and the ERC grant 757296 CRIBLAM. The second author is supported by the ERC CRIBLAM and by the Academy of Finland (grant \#339515). The third author is supported by the Swiss NSF. 
We thank Piet Lammers for useful discussions.

\appendix

\section{Proofs of the statements in Section~\ref{sec:monotonicity}}\label{app: pf of FKG and CBC}
\label{appendix:FKG stuff}

\subsection{Preliminaries}

In this preliminaries, we recall the classical Holley criterion, and also draw a connection between our model and the Ising model.

\subsubsection{Holley and FKG criteria}

Fix some discrete domain $D$ and $\mu$ and $\mu'$ denote two probability measures on $\heightfcns_D$.
We say that $\mu'$ \emph{stochastically dominates} $\mu$, denoted $\mu \leq_{st} \mu'$, 
if there exists a probability measure $\nu$ on $(h,h')\in \heightfcns_D \times \heightfcns_D$ such that the first and second marginal distributions are respectively $\mu$ and $\mu'$, and $\nu[h \preceq h']=1$.
Note that if $\mu \leq_{st} \mu'$, then, for all increasing $F: \heightfcns_D \to \bbR$,
\begin{align*}
	\mu [F(h)]  \leq\mu' [F(h) ].
\end{align*}

We say that $\mu$ is \emph{irreducible} if for any two $h, h' \in \heightfcns_D$ with $\mu[h]>0$ and $\mu[h'] > 0 $, there exists a finite sequence of height functions $h=h_0, h_1, \ldots, h_m = h'$, such that for every $1 \leq i \leq m$, $\mu[h_i] > 0$ and $h_i$ differs from $h_{i-1}$ on one face only. 

We now recall the classical Holley and FKG criteria. For details see the extensive discussion of these criteria in~\cite{Gri06}.

\begin{lemma}[Holley's criterion]\label{lem:Holley}
    Consider two measures $\mu$ and $\mu'$ such that
    \begin{itemize}[noitemsep]
        \item $\mu$ and $\mu'$ are irreducible,
        \item there exists $h \preceq h' \in \heightfcns_D$ such that $\mu[h]>0$ and $\mu'[h'] > 0$, 
        \item for every face $x \in D$, every $k \in \bbZ$, 
        $\mu$-almost every $\chi \in \heightfcns_{D \setminus \{x\} }$, 
        and $\mu'$-almost every $\chi' \in \heightfcns_{D \setminus \{x\} }$ with $\chi \preceq \chi'$, 
        \begin{align}\label{eq:Holley}
            \mu[ h(x) \geq k \; \vert \; h_{\vert D \setminus \{x\}} = \chi ]
            \leq \mu'[ h(x) \geq k \; \vert \; h_{\vert D \setminus \{x\}} = \chi' ],
        \end{align}
    \end{itemize}
    then $\mu \leq_{st} \mu' $.
\end{lemma}

\begin{lemma}[FKG criterion]\label{lem: FKG}
Suppose that $\mu$ is irreducible. If for every face $x \in D$, every $k \in \bbZ$, and $\mu$-almost every $\chi \in \heightfcns_{D \setminus \{x\} }$ and $\chi' \in \heightfcns_{D \setminus \{x\}}$ with $\chi \preceq \chi'$, \begin{align}
\label{eq:FKG criterion}
 \mu[ h(x) \geq k \; \vert \; h_{\vert D \setminus \{x\}} = \chi ]
 \leq 
  \mu[ h(x) \geq k \; \vert \; h_{\vert D \setminus \{x\}} = \chi' ],
\end{align}
then for all increasing functions $F, G: \heightfcns_D \to \bbR$,
\begin{align*}
\mu [F(h)G(h)] \geq \mu [F(h) ] \mu [ G(h)].
\end{align*}
\end{lemma}

\subsubsection{Signs of six-vertex height functions and the Ising model}
\label{subsubsec: Ising tools}

Let $D$ be a discrete domain and $H \in \heightfcns_D$ be non-negative.
Let $G=G(H)=(V, E)$ be the following \nolinebreak{(multi-)graph:} 
the vertices $V$ are labelled by the clusters of $H > 0$ on the graph~$D$;
between any two vertices $u, v \in V$ place as many edges as there are vertices of $D$
that are adjacent to a face in each of the clusters corresponding to $u$ and $v$. 
Notice that any vertex of $D$ that corresponds to an edge of $G$ necessarily has two adjacent faces for which $H =0$. 
%
For $v \in V$, the sign of any height function $h \in \heightfcns_D$ with $|h|=H$ is constant on the cluster of $H>0$ associated with $v$. We denote this sign as $\sigma_{h}(v)$, yielding a function $\sigma_{h} :V \to \{\pm1\}$. 

Define the Ising model on $G$ via the following weights $W_{\mathrm{Ising}}$ and  probability measure $\sfP_{\mathrm{Ising}}$: for $\sigma\in\{\pm1\}^{V}$,
\begin{align*}
W_{\mathrm{Ising}, H} (\sigma) &:= \prod_{e = \langle u, v \rangle \in E} \mathbf c^{\mathbbm{1} [\sigma(u) = \sigma(v) ] }, \\
\sfP_{\mathrm{Ising},H} [ \sigma ] &:= \tfrac{1}{Z} W_{\mathrm{Ising},H} (\sigma).
\end{align*}

\begin{lemma}
\label{lem: signs of h have ising distribution}
Let $h, H \in \heightfcns_D$ satisfying $\vert h \vert = H$. Then, in the above notation
\begin{align*}
W_{6V} (h) = \mathbf c^{N(H)}\,W_{\mathrm{Ising},H} (\sigma_{h}),
\end{align*}
where 
$N(H)$ is the number of type 5--6 vertices of $D$ in $H$ that are not edges of $G$.
\end{lemma}

\begin{proof}
Any  type 5--6 vertex of $h$ is also a type 5--6 vertex in $H$. 
Conversely, any type 5--6 vertex of $H$ which does not correspond to an edge of $G$ is also a type 5--6 vertex in~$h$. 
The other type 5--6 vertices of $H$ however may correspond to either type 1--4 or type 5--6 vertices of $h$, 
depending on the choice of the signs in $h$ of the two clusters of $H > 0$ meeting there.
Indeed, they are of type 5--6 only if the two clusters have same sign. 
We deduce that 
\begin{align}
\nonumber
W_{6V} (h) &:=
\mathbf c^{N(H)}\prod_{e = \langle u, v \rangle \in E}
\mathbf c^{\mathbbm1 [ \sigma_{h} (u) = \sigma_{h} (v)]} = \mathbf c^{N(H)}\,W_{\mathrm{Ising},H} (\sigma_{h}).
\end{align}
\end{proof}

Let now $H,H' \in \heightfcns_D$ be two height functions with $H' \ge H \geq 0$. 
Let $G'=(V', E')=G(H')$.
Note that every cluster of $H > 0$ is thus contained in a unique cluster of $H' > 0$.
Let $\pi: V \to V'$ be the projection corresponding to this inclusion, and define also  
 the preimage map $\pi^{-1}$  of this projection, from $V'$ to subsets of $V$.

\begin{lemma}
\label{lem: Ising on contracted graph is conditional Ising}
Condition the Ising model $\sfP_{{\rm Ising},H}$ on $G$ on the event that $\sigma(\cdot)$ is constant on $\pi^{-1}(v)$ for every $v \in V'$; then the law of  $\sigma \circ \pi^{-1}$ (this is a slight abuse of notation) is $\sfP_{{\rm Ising},H'}$. 
\end{lemma}

\begin{proof}
Consider an edge $e' = \langle u', v' \rangle \in E'$ corresponding to a local configuration of $H'$ given by $\tiny{
\begin{array}{c }
0~1 \\
1~0
\end{array}
}$
or
 $\tiny{
\begin{array}{c }
1~0 \\
0~1
\end{array}
}$.
Since $0 \leq H \preceq H'$, $H$ has the same local configuration, and thus $e'$ corresponds to a unique edge $e \in E$, where furthermore $e = \langle v, u \rangle$ satisfies $\pi(v)=v'$ and $\pi(u)= u'$. We denote this injective map by $\iota: E' \to E$.
We claim that the restriction of $\iota$ is a bijection \[\iota:\{ e' = \langle u', v' \rangle  \in E' \; : \; u' \neq v'\}\longrightarrow\{ e = \langle u, v \rangle  \in E \; : \; \pi(u) \neq \pi(v)\}\] (we use a slight abuse of notation and write $\iota$ for the restriction as well). Indeed, first, for $e' = \langle u', v' \rangle  \in E' $ with $ u' \neq v'$ the image $\iota(e') = \langle v, u \rangle$ satisfies $\pi(v)=v' $ and $ \pi(u)= u'$, so $\pi(u) \neq \pi(v)$. Second, given $e = \langle u, v \rangle  \in E $, the additional condition $ \pi(u) \neq \pi(v)$ implies that the local configuration $\tiny{\begin{array}{c }
0~1 \\
1~0
\end{array}
}$
or
 $\tiny{
\begin{array}{c }
1~0 \\
0~1
\end{array}
}$
of $H$ corresponding to $e$ must be the same in $H'$. Hence, there exists $e' \in E'$, labelled by this local configuration of $H'$, that maps $\iota(e') = e$. This proves the bijectivity, as $\iota$ is by construction injective.

Suppose now that $\sigma(\cdot)$ is constant on $\pi^{-1}(v)$ for every $v \in V'$. Compute
\begin{align}
\nonumber
W_{\mathrm{Ising},H} (\sigma) &= 
\prod_{e = \langle u, v \rangle \in E} \mathbf{c}^{\mathbbm{1}[\sigma(u) = \sigma(v)]} 
\\
\nonumber
&= \underbrace{
\Big( 
\prod_{\substack{e = \langle u, v \rangle \in E \\ \pi(u) = \pi(v) }} \mathbf c
\Big)
}_{K(H, H')}
\times
\prod_{\substack{e = \langle u, v \rangle \in E \\ \pi(u) \neq \pi(v) }}  \mathbf{c}^{\mathbbm{1}[\sigma(u) = \sigma(v)]} 
\\
\nonumber
&= K(H,H') \times \prod_{\substack{e' = \langle u', v' \rangle \in E' \\ u' \neq v' }} \mathbf c^{\mathbbm{1}[\sigma \circ \pi^{-1} (u') = \sigma \circ \pi^{-1} (v')]}
\\
\nonumber
&=
 K(H,H')\mathbf c^{- \# \{ \text{loop edges of }E' \} } W_{\mathrm{Ising},H'} (\sigma \circ \pi^{-1}),
\end{align}
where in the third equality we re-labeled the product using the bijection $\iota$, and used the observation that for an edge $\langle u', v' \rangle \in E'$ in the new labeling, the corresponding $\langle u, v \rangle \in E$,  for which $\iota (\langle u, v \rangle ) = \langle u', v' \rangle  $, satisfies $u \in \pi^{-1} (u')$ and $v \in \pi^{-1} (v')$.
The claimed equality of distributions now follows from the previous displayed equation.
\end{proof}

\subsection{Proof of~\eqref{eq:FKG-h} and~\eqref{eq:CBC-h}}
\label{subsubsec: Proof of CBC and FKG}

We will check the assumption of Lemma~\ref{lem:Holley} for $\mu = \bbP^{\xi}_D$ and $\mu' = \bbP^{\xi'}_D$ where $\xi \preceq \xi'$. In the special case when $\xi = \xi'$, the assumptions of Lemma~\ref{lem:Holley} become those of Lemma~\ref{lem: FKG}. These two lemmas then directly imply~\eqref{eq:CBC-h} and~\eqref{eq:FKG-h}, respectively.

We start by showing the irreducibility of $\bbP^{\xi}_D$. Consider two height functions $h, h'$ which are admissible for $\bbP^{\xi}_D$.
It is easy to check that their point-wise maximum $h \vee h'$ is also admissible. Thus, it suffices to consider the case $h \preceq h'$, 
which is what we do next. 

Assuming that $h \neq h'$, the function $h'-h$ has at least one face of strictly positive value. 
Write $m := \max \{h'(z) - h(z): z\in D\}$ and let $x$ be a face of maximal $h'$-value among the faces $z$ with $h'(z) - h(z) = m$. 
By this maximality, one readily deduces that $h'$ takes values $h'(x)-1$ on all faces adjacent to $x$. 
Thus, the function $h_1$ which is equal to $h'$ on $D \setminus \{x\}$ and equal to $h'(x) -2$ at $x$ is also admissible. 
Applying repeatedly this type of modification, we construct a decreasing sequence of admissible height functions $h' = h_1, \ldots, h_m = h$, 
with $h_{i+1}$ differing from $h_i$ at only one face. In conclusion $\bbP^{\xi}_D$ is irreducible. 
(The monotonicity is unimportant here, but crucial when repeating the same argument for absolute values.) 
The same holds for $\bbP^{\xi'}_D$.

To check the second condition of Lemma~\ref{lem:Holley}, 
let $h$ and $h'$ be arbitrary admissible height functions for $\bbP^{\xi}_D$ and $\bbP^{\xi'}_D$, respectively. 
Then, the point-wise minimum and maximum $h \wedge h'$ and $h \vee h'$ 
are also admissible height functions for $\bbP^{\xi}_D$ and $\bbP^{\xi'}_D$, respectively.
These two height functions satisfy the second condition of Lemma~\ref{lem:Holley}.


We now check~\eqref{eq:Holley}. Let $\chi$ and $\chi'$ as in the assumption of Lemma~\ref{lem:Holley}. Let $N_x$ be the set of faces of $D$ adjacent to $x$ in $D$ (there are between 2 and 4 of them). Let $m:=\min_{y\in N_x} \chi (y)$, $M:=\max_{y\in N_x} \chi (y)$, and $m',M'$ similarly defined for $\chi'$. By assumption, we have that $m\le m'$ and $M\le M'$.

Moreover since $\chi$ and $\chi'$ are admissible, we have $M \in \{m,m+2\}$ and $M' \in \{m',m'+2\}.$
If $M = m+2$, then $h(x) = m+1$ with  $\bbP^{B, \xi}_D [\cdot |  h_{| D \setminus \{ x\} } = \chi ]$- probability $1$.
Otherwise $h(x) \in \{m-1,m+1\}$.  
As a consequence, if either $M>m$ and $M'>m'$, then ~\eqref{eq:Holley} holds trivially. 
The same is true when $m = M  < m' = M'$.

The only remaining case is when $m=m'=M=M'$. 
In this case, for both measures, we know that  $h(x)\in \{m-1, m+1\}$, and it thus remains to show that 
\[
\bbP^{B,\xi'}_D [h(x)=m+1 \; | \; h_{| D \setminus \{ x\} } = \chi' ]  \geq
\bbP^{B,\xi}_D [h(x)=m+1 \; | \; h_{| D \setminus \{ x\} } = \chi] .\] 
Let $N_x^\times$ be the set of faces in $D \setminus \{x\}$ that share a corner with $x$. 
On $N^\times_x$, $\chi$ takes a value in $\{m-1, m, m+1\}$. Define $n_- = \#\{y\in N^\times_x, \chi (y) = m-1\}$, $n_+ = \#\{y\in N^\times_x, \chi (y) = m+1\}$ and $n_-', n_+'$ similarly for $\chi'$. By computing the weights of the different height functions extending $\chi$, we get
	\begin{align*}
	\bbP^{B,\xi}_D [h(x)=m+1 \; | \; h_{| D \setminus \{ x\} } = \chi ] &= \frac{\mathbf c^{n_+}}{\mathbf c^{n_+}+\mathbf c^{n_-}},\\
	\bbP^{B,\xi'}_D [h(x)=m+1 \; | \; h_{| D \setminus \{ x\} } = \chi' ] &= \frac{\mathbf c^{n_+'}}{\mathbf c^{n_+'}+\mathbf c^{n_-'}}.
\end{align*}
Observe that the assumption $\chi \preceq \chi'$ implies $n_+ \le n'_+$ and $n_- \ge n'_-$, and as $\mathbf c\ge 1$, we thus deduce~\eqref{eq:Holley} in this case as well.
\hfill $\square$

\subsection{Proof of~\eqref{eq:FKG-|h|} and~\eqref{eq:CBC-|h|}}
\label{subsubsec: proof of CBC and FKG for abs(h)}

As before, we focus on proving the three properties of Lemma~\ref{lem:Holley} for the laws $\mu$ and $\mu'$ of $|h|$ under $\bbP^{\xi}_D$ and $\bbP^{\xi'}_D$.

For irreducibility, observe that, since $\xi \succeq 0$, 
$\bbP^{\xi}_D[| h | = H] > 0$ if and only if $\bbP^{\xi}_D[ h  = H] > 0$.
The irreducibility of the law of $|h|$ follows from that of $\bbP^{\xi}_D$.
The same holds for $\bbP^{\xi'}_D$.
The second property  of Lemma~\ref{lem:Holley} for $|h|$ is derived in a similar way from that for the law of~$h$. 

Finally, let us prove~\eqref{eq:Holley}.
Fix $0\le \chi \preceq \chi'$. Let $N_x$ be as in the proof of Proposition~\ref{prop: CBC and FKG}. 
Let $m:=\min_{y\in N_x} \chi (y)$, $M:=\max_{y\in N_x} \chi (y)$, and $m',M'$ similarly for $\chi'$. 
Then, $m'\ge m \ge 0$ and $M'\ge M \ge 0$.
Identically to the proof of Proposition~\ref{prop: CBC and FKG}, 
one can show that the only non trivial case is $m=m'=M=M'$, 
which we now assume is the case. 
We divide the proof in three cases depending on whether the common value $m=m'=M=M'$ is equal to $0$, $1$ or larger than or equal to $2$.
	
If $m=0$, then we must have $|h(x)|=1$ under both measures, and we therefore have nothing to prove. 

Suppose now that $m\ge 2$.
As in the proof of Proposition~\ref{prop: CBC and FKG}, let $N^\times_x$ be set of faces sharing a corner with $x$ and
$n_- := \#\{y\in N^\times_x: \chi (y) = m-1 \}$, $n_+ := \#\{y\in N^\times_x: \chi (y)= m+1\}$ and $n_-', n_+'$ similarly for $\chi'$. 
Given that $\chi$ and $\chi'$ only take values in $\{m-1,m,m+2\}$, the sign of $h$ is constant on $N^\times_x$. 
In particular, the types of the vertices at the corners of the square $x$
 only depend on the absolute value $| h |$, not on the sign of $h$. 
 One can thus directly compute the weights of the different possible configurations of $h$ and obtain
\begin{align*}
	\bbP^{\xi}_D [ | h (x) |=m+1 ~|~ | h_{| D \setminus \{ x\} } | = \chi ] &= 
	\frac{\mathbf c^{n_+}}{\mathbf c^{n_+}+\mathbf c^{n_-}},\\
	\bbP^{\xi'}_D [ | h (x) |=m+1 ~|~ | h_{| D \setminus \{ x\} } | = \chi' ] &= 
	\frac{\mathbf c^{n_+'}}{\mathbf c^{n_+'}+\mathbf c^{n_-'}}.
\end{align*} 
As in the proof of Proposition~\ref{prop: CBC and FKG}, $\chi \preceq \chi'$ implies $n_+ \le n'_+$
and $n_- \ge n'_-$, which in turn implies~\eqref{eq:Holley} since $\mathbf c\ge 1$.

There remains the case where $m=1$, which is the core of the proof 
and for which we use the connection to the Ising model mentioned in Section~\ref{subsubsec: Ising tools}.
In this case, there are only two possible values for $|h (x)|$, namely $0$ and $2$. 
We wish to show
\begin{align}\label{eq:to show for Holley}
	\bbP^{\xi}_D \big [| h(x) | = 2 \; \big| \; | h_{| D \setminus \{ x\} } | = \chi \big] 
	& \leq \bbP^{\xi'}_D \big [| h(x) | = 2 \; \big| \; | h_{| D \setminus \{ x\} } | = \chi' \big].
\end{align}

Let $H_0 \in \heightfcns_D$ (resp. $H_2$) be the height functions equal to $0$ (resp. 2) at $x$ and coinciding with $\chi$ on $D\setminus\{x\}$. Define
\begin{align}\label{eq:def of Z}
    Z_0 = \sum_{\substack{ h \in \heightfcns_D \\ | h | =H_0 \\ h \geq 0 \text{ on } B } } W_{6V} (h)
    \text{\qquad and \qquad}
    Z_2 = \sum_{\substack{ h \in \heightfcns_D \\ | h | = H_2 \\ h \geq 0 \text{ on } B } } W_{6V} (h).
\end{align}
Then
\begin{align*}
	\bbP^{\xi}_D \big[ | h(x) | = 2 \; \big| \; | h_{| D \setminus \{ x\} } | = \chi \big] = \frac{Z_2}{ Z_0 + Z_2 }.
\end{align*}
A similar formula is obtained for the ``primed'' configurations. 
To deduce~\eqref{eq:to show for Holley}, one needs to show that 
\begin{align}\label{eq:to show for Holley 2}
	Z_2 \big/ Z_0 \leq  Z'_2 \big/ Z'_0.
\end{align}

Now follows a simple but crucial observation. There is an injection $\mathbf T$ from the height functions $h$ contributing to $Z_2$ to the height functions $h$ contributing to $Z_0$: simply change the value $\pm 2$ of $h(x)$ to $0$. The image of this injection is exactly those $h$ contributing to $Z_0$ for which in addition $h$ has constant sign\footnote{And this sign tells whether the preimage takes value $+2$ or $-2$ at $x$, which implies the injectivity.} on $N_x$. Set $n_0 := \# \{ y \in N_x^\times \; : \; \chi(y) = 0 \}$ and $n_2:= \# \{ y \in N_x^\times \; : \; \chi(y) = 2 \}$. Under this injection the six-vertex weights become
\begin{align*}
W_{6V} (\mathbf T(h)) = \mathbf c^{n_0 - n_2} W_{6V} (h).
\end{align*}
We can thus express $Z_2$ using this up-to-constant weight-preserving injection as
\begin{align*}
Z_2 =  \mathbf c^{n_2 - n_0} \sum_{\substack{ h \in \heightfcns_D \\ | h | = H_0 \\ h \geq 0 \text{ on } B \\
\mathrm{sign}(h) \text{ cst. on }N_x  } } W_{6V} (h),
\end{align*}
and finally, using~\eqref{eq:def of Z},
\begin{align}
\label{eq:Z2/Z0 using sign probas}
Z_2 \big/ Z_0 = 
\mathbf  c^{n_2 - n_0} 
\bbP^{\xi}_D [ \mathrm{sign}(h) \text{ cst. on }N_x \; | \; | h | = H_0 ].
\end{align}
A similar formula holds for the ``primed'' configurations.

Recall again that $\mathbf c \geq 1$ and that $n_2 - n_0 \leq n'_2 - n'_0$. Using~\eqref{eq:Z2/Z0 using sign probas} and its ``primed'' analogue, we observe that for~\eqref{eq:to show for Holley 2} to hold it thus suffices that 
\begin{align}
\label{eq:to show for Holley 3}
\bbP^{\xi}_D [ \mathrm{sign}(h) \text{ cst. on }N_x \; | \; | h | = H_0 ]
 \leq
 \bbP^{\xi'}_D [ \mathrm{sign}(h') \text{ cst. on }N_x \; | \; | h' | = H'_0 ].
\end{align} 
Let us now study the conditional probability appearing on the left. Lemma~\ref{lem: signs of h have ising distribution} gives
\begin{align*}
\bbP^{\xi}_D [ \mathrm{sign}(h) \text{ cst. on }N_x \; | \; | h | = H_0 ]
&
=
\sum_{\substack{ h \in \heightfcns_D \\ | h | = H_0 \\ h \geq 0 \text{ on } B \\
\mathrm{sign}(h) \text{ cst. on }N_x } } W_{6V} (h) 
\Bigg/
\sum_{\substack{ h \in \heightfcns_D \\ | h | = H_0 \\ h \geq 0 \text{ on } B} } W_{6V} (h)
\\
& \stackrel{
\tiny{ \text{} }
}
{=}
\sum_{\substack{ \sigma \in \{\pm1 \}^{V} \\
\sigma = + 1 \text{ on } B \\
\sigma \text{ cst. on }N_x } } 
W_{\mathrm{Ising},H_0} ( \sigma ) 
\Bigg/
\sum_{\substack{ \sigma \in \{\pm1 \}^{V} \\
\sigma = + 1 \text{ on } B
} }
 W_{\mathrm{Ising},H_0} ( \sigma ) 
\\
&= \sfP_{\mathrm{Ising},H_0} [ \sigma \text{ cst. on } N_x \; | \; \sigma = + 1 \text{ on } B ].
\end{align*}
where the Ising model is as in Section~\ref{subsubsec: Ising tools}, and by ``$\sigma$  cst.~on $N_x$'' we mean that $\sigma$ is constant on the vertices of $v$ labeled by clusters of $H_0 > 0$ intersecting $N_x$; ``$\sigma = + 1$ on $B$'' should be interpreted analogously. 

A similar reasoning together with  
 Lemma~\ref{lem: Ising on contracted graph is conditional Ising} applied to $H_0 \preceq H_0'$  gives that
\begin{align*}
& \bbP^{\xi'}_D [ \mathrm{sign}(h') \text{ cst.~on }N_x \; | \; | h' | = H'_0 ]\\
\nonumber  &=\sfP_{\mathrm{Ising},H_0'} [ \sigma' \text{ cst.~on } N_x \; | \; \sigma' = + 1 \text{ on } B ]\\
&=
\sfP_{\mathrm{Ising},H_0} [ \sigma  \text{ cst.~on } N_x \; | \; \{\sigma  \text{ cst.~on } \pi^{-1}(v') \text{ for each }v' \in V'\}\cap\{
\sigma = + 1 \text{ on } B\}].
\end{align*}
Plugging the two previous displayed equations in~\eqref{eq:to show for Holley 3}, we see that it suffices to show that 
\begin{align}
\label{eq:to show for Holley 4}
\sfP_+ [ \sigma \text{ cst.~on } N_x ]
\leq
\sfP_+ [ \sigma  \text{ cst.~on } N_x \; \vert \; \sigma  \text{ cst.~on } \pi^{-1}(v') \text{ for each }v' \in V' ],
\end{align}
where $\sfP_+$ denotes $\sfP_{\mathrm{Ising},H_0}[\,\cdot\,|\sigma = + 1 \text{ on } B ]$. 

Denote by $N$ the vertices of $V$ that correspond to clusters intersecting  $N_x$, and denote the sets $\pi^{-1}(v')$ by $U_i$.  Equivalently to~\eqref{eq:to show for Holley 4}, we want to prove
\begin{align*}
\mathrm{Cov}_{\sfP_+} & \big( \mathbbm{1} [ \sigma \text{ cst.~on } N ], 
\prod_{i=1}^m  \mathbbm{1} [ \sigma \text{ cst.~on } U_i ]
\big) 
 \geq 0.
\end{align*}
Now, note that we have
\begin{align*}
\mathbbm{1} \{ \sigma \text{ cst.~on } A \}
= \prod_{u,v \in A}\frac{1+\sigma_u\sigma_v}{2}
= \sum_{U \subset A} a_U \prod_{u \in U} \sigma_u,
\end{align*}
where  $a_U \geq 0$ for every $U \subset A$. Applying this formula for $A=N$ and $A=U_i$, we get 
\begin{align*}
\mathrm{Cov}_{\sfP_+} & \big( \mathbbm{1} [ \sigma \text{ cst.~on } N ], 
\prod_{i=1}^m  \mathbbm{1} [ \sigma \text{ cst.~on } U_i ]
\big) 
=
\sum_{U \subset V}
\sum_{U' \subset N} a_U b_{U'}
\mathrm{Cov}_{\sfP_+} \big( \prod_{u' \in U'}  \sigma_{u'} ,  \prod_{u' \in U'} \sigma_{u}
 \big)
 \geq 0,
\end{align*}
where in the last step we observed that $a_U, b_{U'} \geq 0$ and that
by Griffiths' second inequality~\cite{Gri67}, each individual covariance term in the sum is non-negative. This finishes the proof.

\bibliographystyle{abbrv}
\def\cprime{$'$}

\end{document}